\documentclass[a4paper]{article}
\usepackage[utf8]{inputenc}

\usepackage{amsmath,amssymb,amsthm,bbm}
\usepackage[margin=3cm]{geometry}

\usepackage{hyperref}

\usepackage[table]{xcolor}
\usepackage{color}
\usepackage{graphicx}
\usepackage{booktabs}
\usepackage{float}
\usepackage{overpic}

\usepackage{diagbox}

\usepackage{pgfplots}
\pgfplotsset{compat=newest}
\usepackage{pgfplotstable}
\usepgfplotslibrary{fillbetween}
\usetikzlibrary{external}
\usepackage{xfrac}

\usepackage[shortlabels]{enumitem}
\usepackage{accents}
\newcommand*{\dt}[1]{%
  \accentset{\mbox{\large\bfseries .}}{#1}}
\usepackage{romanbar}

\newtheorem{lemma}{Lemma}
\newtheorem{proposition}{Proposition}
\newtheorem{theorem}{Theorem}
\newtheorem{remark}{Remark}

\newtheorem{definition}{Definition}

\newcommand{\T}{T_{\rJ, \alpha, \beta}}
\newcommand{\invT}{T_{\rJ, \alpha, \beta}^{-1}}

\newcommand{\xzero}{\xi}

\newcommand{\J}{\mathbf{J}}
\newcommand{\rJ}{\mathrm{J}}
\newcommand{\dd}{\mathrm{d}}
\renewcommand{\div}{\nabla \cdot}
\newcommand{\avgux}{\langle \mathbf{u}_x \rangle}
\newcommand{\avguxt}{\langle \tilde{\mathbf{u}}_x \rangle}

\definecolor{brightcerulean}{rgb}{0.11, 0.67, 0.84}
\definecolor{brightpink}{rgb}{1.0, 0.0, 0.5}
\definecolor{blue-green}{rgb}{0.0, 0.87, 0.87}
\definecolor{caribbeangreen}{rgb}{0.0, 0.8, 0.6}

\title{A PDE model for unidirectional flows: stationary profiles and asymptotic behaviour}
\author{Annalisa Iuorio\footnote{Radon Institute for Computational and Applied Mathematics, Austrian Academy of Sciences, Vordere Zollamtstr. 3, 1030 Vienna, Austria.} \and Gaspard Jankowiak\footnotemark[1]  \and Peter Szmolyan\footnote{Technische Universit\"at Wien, Institute for Analysis and Scientific Computing, Wiedner Hauptstr. 8-10, 1040 Wien, Austria.} \and Marie-Therese Wolfram\footnote{University of Warwick, Mathematics Institute, CV47AL Coventry UK.}$\ ^{,} $\footnotemark[1]}
\date{}

\begin{document}

\maketitle

\begin{abstract}
In this paper, we investigate the stationary profiles of a convection-diffusion model for unidirectional pedestrian flows in domains with a single entrance and exit. The inflow and outflow conditions at both the entrance and exit as well as the shape of the domain have a strong influence on the structure of stationary profiles, in particular on the formation of boundary layers. We are able to relate the location and shape of these layers to the inflow and outflow conditions as well as the shape of the domain using geometric singular perturbation theory. Furthermore, we confirm and exemplify our analytical results by means of computational experiments.
\end{abstract}



\section{Introduction}

\noindent In this paper, we consider a continuum model for unidirectional pedestrian flows in domains with a single exit and entrance. The geometry of the domain as well as the inflow and outflow rates lead to the formation of different boundary layers. We analyse the respective stationary pedestrian profiles using {Geometric Singular Perturbation Theory} (GSPT) to understand the influence of those factors on the formation and location of layers. This gives interesting insights on how for example bottlenecks due to construction sites can result in congestion.

\noindent Pedestrian simulations are nowadays standard practice to ensure the safety of public buildings and events. The respective software packages are based on a tremendous amount of research in model development and simulation tools in the last decades. Mathematical models for crowds can be divided into microscopic and macroscopic approaches. In the microscopic framework, people are treated as individual entities. The two most popular approaches are force-based and cellular automata models. In the former, the dynamics of each individual is described by  a nonlinear ordinary differential equation (ODE) leading to a high-dimensional system of coupled ODEs. The social force model is one of the most prominent in this context \cite{helbing1995social}. In cellular automata, people move on a discrete lattice with interaction dependent transition rates. Microscopic models form the basis of most pedestrian simulation software in the engineering and transportation research community, see for example \cite{BHRW2016, BMP2011}. In contrast to microscopic models, macroscopic models describe the entire crowd as an entity. The dynamics of the crowd is governed by nonlinear conservation laws, in which the average velocity of the crowd usually depends on a desired direction (for example to reach an exit or target), interactions with others, and stochastic fluctuations. In general, macroscopic models are mathematically more amendable and allow to analyse how certain factors, such as the inflow and outflow rates or the geometry, influence the stationary profiles. Stationary profiles are of particular interest in the context of pedestrian dynamics, as they allow to identify regions of high densities and give insights into rooms capacities. For an overview on the mathematical modelling of pedestrian crowds we refer to \cite{CPT2014}. More recently, the connection between microscopic data and the respective mean-field models has been investigated more thoroughly. For example Gomes et al.~\cite{GSW2019} used the Bayesian framework to identify the maximum velocity in the same model that we are investigating here, using individual trajectories.\\
The model considered in this paper is a parabolic convection-diffusion equation, which was proposed by Burger and Pietschmann in \cite{burger_flow_2016}. The authors start with a cellular automata model describing a large pedestrian crowd, which enters a domain at a given rate through an entrance and exits at a possibly different rate through the exit, and then formally derive the respective mean-field model. The derived equation corresponds to a 2D viscous Burgers' type equation with nonlinear inflow and outflow boundary conditions. The Burgers' equation is well known and has been thoroughly investigated, especially in the context of traffic flow. It can be obtained by scaling the Lighthill-Whitham-Richards (LWR) model, see \cite{lighthill1955kinematic, richards1956shock}. Its analysis is well understood on the real line; however, there are fewer results in higher space dimension or for more complicated boundary conditions. Burger and Pietschmann have characterised the stationary profiles for straight corridors, showing the formation of boundary layers at the entrance and/or the exit depending on the inflow and outflow rate. We generalise their results by deriving and analysing a 1D reduction for axially symmetric 2D domains. Furthermore, we investigate the stationary profiles in the vanishing diffusion limit using GSPT, which allows us to analyse the structure of stationary regimes for more general domains. This gives interesting insights on how boundary conditions and structural features lead to the formation of high and low density regimes inside the corridor as well as at entrances and exits. Our results allow us to identify regions in the parameter space of the boundary conditions in which solutions have the same structure - that is the same type and location of boundary layers - depending on the inflow and outflow conditions. The shape of these regions depends on the geometry of the domain; transitions between them can lead to very different profiles. Our analysis allows us to identify conditions under which small changes to the inflow rate, the outflow rate or the geometry may lead to very different stationary states.  \\

\noindent 
GSPT is a dynamical systems approach to singularly perturbed ordinary differential equations
started by the pioneering work of Fenichel \cite{Fenichel_1979}. The most common form of GSPT
considers slow-fast systems of the form
\begin{equation} \label{eq:slow}
 \begin{aligned}
  \dt{u} &= f(u,v), \\
  \varepsilon \dt{v} &= g(u,v),
 \end{aligned}
\end{equation}
where $u$ and $v$ are functions of $t$ and $0 < \varepsilon \ll 1$.
Often $t$ has the interpretation of time but it may represent equally well a spatial variable.
For $f =O(1)$ and $g =O(1)$ the variable $u$ varies on the slow time-scale $t$ and the variable $v$ on the fast time-scale
$\tau := \frac{t}{\varepsilon}$, which explains the name slow-fast system.
Written on the fast time-scale the equation has the form 
\begin{equation} \label{eq:fast}
 \begin{aligned}
  u' &= \varepsilon f(u,v), \\
  v' &= g(u,v).
 \end{aligned}
\end{equation}
Under suitable assumptions solutions of system (\ref{eq:slow}) 
for small values of $\varepsilon$ can be constructed as perturbation of concatenations of solutions of the two limiting problems obtained by setting
$\varepsilon =0$ in systems (\ref{eq:slow}) and (\ref{eq:fast}), which are referred to as the
reduced problem and the layer problem, respectively.
In GSPT these constructions are carried out in the framework of dynamical systems theory;
with the theory of invariant manifolds playing a particularly important role.
In the specific problem analysed in this paper well established results and methods
from GSPT are used. Therefore, and also because of lack of space, we do not give a more detailed
summary of GSPT, but refer to  \cite{Jones_1995, Kuehn_2015} 
for more background on GSPT and its many applications.
The necessary concepts and results are explained in Section 4 as needed in the context of the specific problem at hand.
GSPT has been used extensively to construct global solutions with interesting dynamics, e.g. 
relaxation oscillations, see e.g. \cite{Szmolyan_2004}
and heteroclinic or homoclinic orbits, see e.g. 
\cite{Jones_1995}. 
Applications to boundary value problems on finite intervals are less common,  but see e.g. \cite{Hayes_1998} for a basic example and \cite{Iuorio_MEMS}. 
In many of its applications GSPT allows a full description of bifurcation diagrams up to 
their singular limit $\varepsilon =0$, where numerical tools have difficulties or fail.

\medskip
\noindent In this paper, we propose and analyse a mean-field model for unidirectional pedestrian flows using PDE techniques as well as GSPT. Its main contributions can be summarised as follows:
\begin{itemize}[itemsep=2pt, topsep=2pt]
    \item Proposal of a 1D area averaged model for unidirectional pedestrian flows in axially symmetric domains, accessible to mathematical analysis.
    \item For a restricted set of geometries: characterisation of stationary pedestrian profiles depending on the inflow and outflow rate as well as the geometry. Analysis of the limiting profiles using GSPT.
    \item Validation and computation of stationary profiles using numerical experiments for various geometries.
\end{itemize}

\noindent This paper is organised as follows: we introduce the mathematical model and perform a reduction to one dimensional by averaging of the cross-sectional area in Section \ref{sec:model}. Then, we discuss existence and uniqueness of stationary solutions
of this area averaged model in Section \ref{sec:analysis}. Section \ref{sec:vanvis} focuses on the vanishing viscosity limit of the area averaged model, which allows us to characterise
the stationary profiles when the width is monotonic in the direction of movement. We complement our results with computational experiments in Section \ref{sec:numerics}, and give an outlook on future research directions in Section \ref{sec:conclusion}.

\section{A PDE model for unidirectional flows in corridors}\label{sec:model}

We start by introducing the PDE model for unidirectional pedestrian flows proposed by Burger and Pietschmann in \cite{burger_flow_2016} before deriving the 1D area averaged model investigated in this paper.\\
\noindent We consider a large pedestrian crowd with density $\rho = \rho(x,t)$, evolving in a bounded domain $\Omega \subset \mathbb{R}^n$. Through parts of the boundary $\partial\Omega$, denoted by $\Gamma$ and $\Sigma$ respectively, the crowd can enter and leave the domain.
We assume that $0 \le \rho \le \rho_\mathrm{max} \in \mathbb{R}$ and that the dynamics of the crowd is driven by transport and diffusion.
It is assumed that the transport velocity depends only on the local density. This relationship was first proposed by Greenshields \cite{greenshields_photographic_1934} in the context of vehicular traffic and was coined \emph{fundamental diagram} by Haight in 1963 \cite{haight_mathematical_1963}. It is now a classical working hypothesis in both traffic and crowd motion modelling. The precise relation depends on experimental conditions and is still an active research topic.
Following experimental data, \emph{e.g.} \cite{SSKLB2007}, we will assume that at low densities, individuals can walk with a maximum velocity $v_\mathrm{max}$. However, their speed decreases as the density increases and approaches zero at a certain density (denoted by $\rho_\mathrm{max}$) due to overcrowding. The form of the fundamental diagram suggests a linear decrease of velocity as a function of density; hence we set
$$v = v_{\max} \left(1-\frac{\rho}{\rho_\mathrm{max}}\right).$$
Different values for $\rho_\mathrm{max}$ and $v_\mathrm{max}$ can be found in the literature, ranging from $3.8$ to $10$ pedestrians per square meter for $\rho_\mathrm{max}$ and $0.98$ m/s to $1.5$ m/s for $v_\mathrm{max}$, see for example \cite{SSKLB2007}.
We set without loss of generality  $v_\mathrm{max} = \rho_\mathrm{max} = 1$ in the following.
Moreover, we assume that all individuals share the common objective of moving from the entrance to the exit.
This is modelled by the choice of the normalized vector field $\mathbf{u} : \mathbb{R}^n \mapsto \mathbb{R}^n$ which defines the direction of transport, see \eqref{eq:general flux}.
For example, $\mathbf{u}(x)$ can be chosen to be the unit tangent vector to the geodesic from an interior point $x$ to the exit.
Let $\J$ denote the total flux. The previous considerations can be formalized as the following convection-diffusion system:
\begin{subequations}
\label{eq:original_system}
\begin{gather}
    \partial_t \rho + \nabla \cdot \J= 0\,,
    \label{eq:general continuity}
    \\
    \J = -\varepsilon \nabla \rho + \rho \left(1-\rho\right) \mathbf{u}\,,
    \label{eq:general flux}
\end{gather}
where $\varepsilon$ denotes the diffusion coefficient. \eqref{eq:original_system} can also be derived rigorously as the continuous limit of an exclusion process on lattices \cite{giacomin_phase_1997}.
Individuals are assumed to exit the corridor at $\Sigma$ with a given rate $\beta \in (0,1)$.
At $\Gamma$, they enter with rate $\alpha \in (0,1)$ but are also subject to volume exclusion, which makes entering less likely as $\rho$ approaches $1$. This results in an effective inflow rate $\alpha\left(1-\rho\right)$ and gives the following boundary conditions:
\begin{alignat}{2}
    \label{eq:no flux J 2D}
    \J \cdot \mathbf{n} &= 0 &\quad\text{ on } &\partial\Omega \setminus (\Gamma \cup \Sigma)\,,
    \\
    -\J \cdot \mathbf{n} &= \alpha (1-\rho) &\quad \text{ on } &\Gamma\, \label{eq:2d_influx},
    \\
    \J \cdot \mathbf{n} &= \beta \rho &\quad \text{ on }& \Sigma\,.
\end{alignat}
\end{subequations}
Existence of stationary solutions to \eqref{eq:original_system} was studied in \cite{burger_flow_2016}, for smooth $\Omega \subset \mathbb{R}^n$, $n \in \{1,2,3\}$.
Because of the non-standard boundary conditions, classical results from the literature cannot be applied. The authors present two different existence proofs: the first one for a divergence free vector field $\mathbf{u}$, while the second one assumes that $\mathbf{u} = \nabla V$, where $V$ is a given potential. In the latter case, the equation can be interpreted as a
Wasserstein gradient flow.
Burger and Pietschmann have also characterised
the stationary profiles of \eqref{eq:original_system} in one spatial dimension. If  $\Omega = [0, L]$, they have proved uniqueness of the solution for $u \equiv 1$
and have shown that there are three distinct regimes:
\begin{itemize}
    \item \emph{influx-limited}, where the overall density is low and higher towards the exit $\Sigma$.
    \item \emph{outflux-limited}, where the overall density is high and lower towards the entrance $\Gamma$.
    \item \emph{maximal flux}, where $\rho$ is close to $\frac{1}{2}$ with boundary layers at the entrance and exit.
\end{itemize}
The occurrence of these regimes depends on the values of the inflow and outflow rates $\alpha$ and $\beta$.
These regimes are characterised precisely in \cite{burger_flow_2016}, the first two for $\alpha,\beta < \frac{1}{2}$, the third for $\alpha,\beta > \frac{1}{2}$.

\subsection{The area averaged 1D model}
Next we propose a possible reduction of \eqref{eq:original_system}, which can be used to study analytically how the geometry and the boundary conditions influence the stationary profiles. To this aim we reduce the problem to one spatial dimension by averaging the flow along the cross section. A similar approach was used for example in \cite{burger_nonlinear_2012} in the context of ions flowing through radially symmetric nanopores.\\
For this purpose, we restrict our attention to corridor-shaped domains $\Omega \in \mathbb{R}^2$ of the form
\begin{align}\label{e:omega}
    \Omega := \left\{(x,y) \ : \ x \in [0, L],\, y \in \frac{1}{2}[-w(x), w(x)]\right\}\,,\tag{Assumption 1}
\end{align}
where $w:[0,L] \mapsto (0,+\infty)$ is the width in the $y$ direction, see Figure~\ref{fig:geometry}. In the following analysis, $w$ is assumed to be smooth. By definition, the considered domain $\Omega$ is symmetric w.r.t. the $x$-axis. From now on we consider smooth vector fields $\mathbf{u}$ which are symmetric w.r.t.~the $x$-axis as well, i.e.~satisfying
\begin{equation}
    \begin{gathered}
        \mathbf{u}_x(x,-y) = \mathbf{u}_x(x, y) \text{ and } \mathbf{u}_y(x,-y)=-\mathbf{u}_y(x,y)\,,
    \\
    |\mathbf{u}| = 1\,,
    \\
    \mathbf{u} \cdot n = \begin{cases} -1 & \text{on } \Gamma, \\
    1 & \text{on } \Sigma.\end{cases}
    \end{gathered}
    \label{ass:u}
    \tag{Assumption 2}
\end{equation}

The last condition in~\ref{ass:u} is not crucial here, but we introduce it to simplify the derivation of the model, see Remark~\ref{rem:boundary u}. In particular, the choice of $\mathbf{u}$ in our numerical experiments does not meet this last assumption. 

\begin{figure}[ht!]
    \centering
    \vspace{.4cm}
    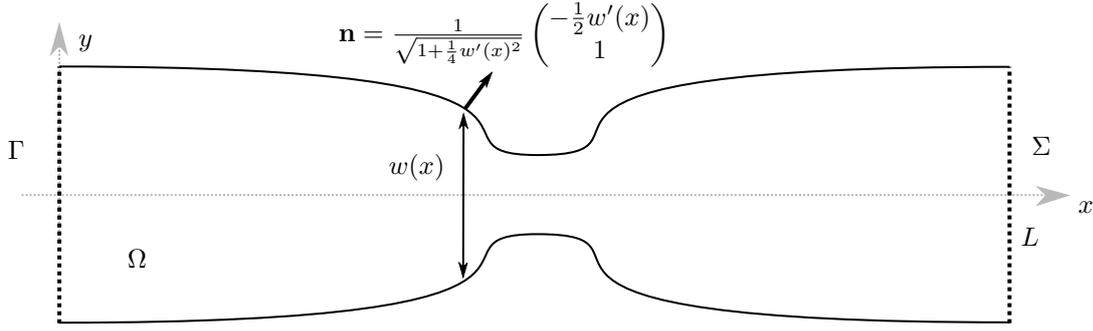
    \caption{Sketch of a typical domain $\Omega$, the entrance boundary $\Gamma$ and the exit boundary $\Sigma$. The width at a given point $x$ is $w(x)$ and the outward normal vector is $\mathbf{n}$.}
    \label{fig:geometry}
\end{figure}

Under Assumptions 1 and 2, the uniqueness of the steady state (see \cite{burger_flow_2016}) ensures that both $\rho$ and $\J$ are symmetric w.r.t.~the $x$-axis, i.e.~for all $(x,y) \in \Omega$:
\begin{equation*}
    \rho(x, -y) = \rho(x, y)\,,
    \quad \text{and} \quad
    \J(x, -y) =
    \begin{pmatrix} \J_x(x, -y) \\ \J_y(x, -y) \end{pmatrix}
    =
    \begin{pmatrix} \J_x(x, y) \\ -\J_y(x, y) \end{pmatrix}
    \,.
\end{equation*}
By integrating \eqref{eq:general continuity} with respect to $y$ we get
\begin{equation}
    \frac{\partial}{\partial t} \int_{-w(x)/2}^{w(x)/2} \rho \; \dd y + \int_{-w(x)/2}^{w(x)/2} \nabla \cdot \J \; \dd y = 0\,.
    \label{eq:general continuity integrated}
\end{equation}
Using \eqref{eq:no flux J 2D} on the side walls, we have that
\begin{align*}
    \int_{-\frac{w}{2}}^{\frac{w}{2}} \div \J(x,y) \; \dd y  &=
    \int_{-\frac{w}{2}}^{\frac{w}{2}} \partial_x \J_x(x,y) + \partial_y \J_y(x,y) \; \dd y
    \\
    &=
    \frac{\partial}{\partial x}\int_{-\frac{w}{2}}^{\frac{w}{2}} \J_x(x,y) \; \dd y - \frac{\partial_x w(x)}{2} \left(\J_x\left(x, \frac{w(x)}{2}\right) + \J_x\left(x, -\frac{w(x)}{2}\right)\right)
    \\
    &\quad+ \J_y\left(x, \frac{w(x)}{2}\right) - \J_y\left(x, -\frac{w(x)}{2}\right)
    \\
    &=
    \frac{\partial}{\partial x}\int_{-\frac{w}{2}}^{\frac{w}{2}} \J_x(x,y) \; \dd y
    + \begin{pmatrix}
        - \partial_x w(x)/2 \\ 1
    \end{pmatrix}
    \cdot \J \left(x, \frac{w(x)}{2}\right)
    \\
    &\;\;\quad\quad\quad\quad\quad\quad\quad\quad\quad- \begin{pmatrix}
        \partial_x w(x)/2 \\ 1
    \end{pmatrix}
    \cdot \J \left(x, \frac{-w(x)}{2}\right)
    \\
    &=
    \frac{\partial}{\partial x} \int_{-\frac{w}{2}}^{\frac{w}{2}} \J_x(x,y) \; \dd y
    =
    \nabla \cdot \int_{-\frac{w}{2}}^{\frac{w}{2}} \J(x,y) \; \dd y\,.
\end{align*}
We see that \eqref{eq:general continuity integrated} has in fact divergence form, that is
\begin{equation}
    \frac{\partial}{\partial t} \int_{-\frac{w}{2}}^{\frac{w}{2}} \rho \; \dd y + \nabla \cdot \int_{-\frac{w}{2}}^{\frac{w}{2}} \J \; \dd y = 0\,.
    \label{eq:general continuity integrated divergence}
\end{equation}

From now on, we make the following key approximation and neglect the variation of $\rho$ in the transversal direction $y$:
\begin{equation}
    \rho(x, y) = \rho(x)\,.
    \tag{Approximation 1}
\end{equation}
This can be justified in the limit of narrow corridors, see~\cite{burger_nonlinear_2012}.
Plugging Approximation 1 into~\eqref{eq:general continuity integrated divergence} we obtain:
\begin{equation}
    w(x) \partial_t \rho(x) + \partial_x \left(w(x) \left(-\varepsilon \partial_x \rho(x) + \rho(x)\left(1-\rho(x)\right)\avgux(x)\right)\right) = 0\,,
    \label{eq:stationary states rho}
\end{equation}
where $\langle \cdot \rangle = w^{-1}(x) \int \cdot \; \dd y$.
The corresponding boundary conditions are
 \begin{equation}
\begin{aligned}
    \varepsilon\, \partial_x \rho &= (\rho - \alpha) (1-\rho)\,,
 & \text{at } x &= 0\,,
 \\
    \varepsilon\, \partial_x \rho &= \rho ((1-\rho) - \beta)\,,
 & \text{at } x &= L\,.
\end{aligned}
\label{eq:boundary conditions rho}
\end{equation}
At this point $\langle \mathbf{u}_x\rangle$ can be scaled out by the change of variables $x \rightarrow \tilde{x} = \int_0^x \avgux (s) \dd s \in [0, \tilde{L}]$, where $\tilde{L} = \int_0^L \avgux(s) \dd s$.
The equation for the transformed variable $\tilde{\rho}(\tilde{x}) = \rho(x)$ reads as:
\begin{equation} \label{eq:aveq}
    \partial_t \tilde{\rho} + \partial_{\tilde{x}} \left( \tilde{w} \avguxt \left(-\varepsilon \partial_{\tilde{x}} \tilde{\rho} + \tilde{\rho}\left(1- \tilde{\rho}\right)\right)\right) = 0\,.
\end{equation}
We define $\tilde{k} = \tilde{w} \avguxt$, which then depends on the parameters $w$ and $\mathbf{u}$ of the original problem.
\begin{remark}[Boundary conditions for $\mathbf{u}$]
    \label{rem:boundary u}
    Under~\ref{ass:u}, the boundary conditions~\eqref{eq:boundary conditions rho} are unchanged by the change of variable, since $\avgux = 1$ for $x = 0,L$.

    One can also consider more general choices for $\mathbf{u}$, by introducing~$\tilde\alpha = \avgux(0)^{-1} \alpha$. Since $|u| = 1$, we have $\avgux(0)^{-1} > 1$ and $\alpha < \tilde\alpha$. By restricting $\tilde\alpha$ to the range $(0,1)$, the analysis of the next section still holds. A similar reasoning applies to $\beta$.
\end{remark}
After dropping the tilde notation, Equation \eqref{eq:aveq} then reads
\begin{equation*}
    \partial_t \rho + \partial_{{x}} \left( k(x) \left(-\varepsilon \partial_{{x}} {\rho} + {\rho}\left(1- {\rho}\right)\right)\right) = 0\,.
\end{equation*}
For stationary unidirectional pedestrian flows, we drop the time dependency and obtain the following
1D area averaged model
\begin{equation}
    \partial_x \rJ = \partial_x(k(x) j(x)) = 0\,,
    \label{eq:reduced equation rho}
\end{equation}
where the flux $j$ is defined by $j(x) = -\varepsilon \partial_x \rho + \rho \left(1-\rho\right)$,
and is subject to the boundary conditions:
\begin{equation}
\begin{aligned}
    j &= \alpha \left(1 - \rho\right)
 & \text{ at } x &= 0\,,
 \\
    j &= \beta \rho
 & \text{ at } x &= L\,.
\end{aligned}
\label{eq:boundary conditions reduced rho}
\end{equation}
We will focus on this system in the rest of the paper.

\section{Analysis of the 1D area averaged model}\label{sec:analysis}

In this section, we discuss existence and uniqueness of solutions for the area averaged model \eqref{eq:reduced equation rho} with boundary conditions
\eqref{eq:boundary conditions reduced rho}, as well as some of their properties, such as their symmetry and the validity of a maximum principle.

\subsection{Existence of solutions}

We start by defining weak solutions for our problem.

\begin{definition}[Weak solution]
    \label{def:weak solution}
    In the following, we say that $\rho \in H^1(0, L)$ is a \emph{weak solution} of the boundary value problem~\eqref{eq:reduced equation rho}-\eqref{eq:boundary conditions reduced rho} if
    \begin{equation*}
        \int_0^L k(x) \left( \varepsilon \partial_x \rho + \rho\left(1-\rho\right)\right) \partial_x \phi \;\dd x - \alpha k(0) (1-\rho(0)) \phi(0) + \beta k(L) \rho(L) \phi(L) = 0\,
    \end{equation*}
    for all $\phi\in H^1(0,L)$.
\end{definition}

We will adapt the arguments presented in \cite{burger_flow_2016} for the potential case to show existence and uniqueness of solutions to \eqref{eq:reduced equation rho}.
We start by defining the entropy
\begin{equation}
    \label{eq:entropy}
    E[\rho] = \int_0^{L} \left(\rho \log \rho + (1-\rho)\log(1-\rho) + \frac{x}{\varepsilon} \rho \right)\; \dd x\,,
\end{equation}
and rewrite \eqref{eq:reduced equation rho} as
\begin{equation}
    \label{eq:stationary states gradient flow formulation}
    \partial_x \cdot \left(\varepsilon k(x) \rho(1-\rho) \partial_x \frac{\delta E[\rho]}{\delta \rho} \right) = 0\,
\end{equation}
where $\frac{\delta}{\delta \rho}$ denotes the variational derivative.

\begin{theorem}[Existence]
    \label{thm:existence}
    Let $k \in L^\infty([0, L])$ be bounded away from zero and continuous at $0$ and $L$. Then, for any $0<\alpha<1$, $0<\beta<1$, there exists
    at least one weak solution $\rho\in H^1([0,L])$ to Equation \eqref{eq:reduced equation rho}
    satisfying the boundary conditions \eqref{eq:boundary conditions reduced rho}. Moreover, $0 < \rho < 1$.
\end{theorem}

\begin{proof}
The proof follows the lines of \cite{burger_flow_2016}. Hence we sketch its main steps only.\\
Without loss of generality, we can set $\varepsilon = 1$ by scaling and
define the entropy variable
\begin{equation*}
    \psi := \frac{\delta E[\rho]}{\delta \rho}= \log \rho - \log (1-\rho) + x\,.
\end{equation*}
Equation \eqref{eq:reduced equation rho} then reads as
\begin{equation}
    \partial_x \left(k(x) A(\psi) \partial_x \psi \right) = 0\,,
    \label{eq:stationary entropy formulation}
\end{equation}
where $A(\psi) = \frac{e^{\psi - x}}{\left(1 + e^{\psi - x}\right)^2}$, with boundary conditions
\begin{equation*}
    - A(\psi) \partial_x \psi =
    \begin{cases}
        \alpha \frac{1}{1 + e^{\psi - x}} & \text{ at } x = 0\,,
        \\
        \beta \frac{e^{\psi - x}}{1 + e^{\psi - x}} & \text{ at } x = L\,.
    \end{cases}
\end{equation*}
Following \cite{burger_flow_2016}, we introduce the modified operator $A_\delta := A + \delta$, and consider the regularised equation
\begin{equation*}
    \partial_x \left(k(x)\, A_\delta(\psi_\delta) \nabla \psi_\delta\right) + \delta \psi_\delta= 0\,,
\end{equation*}
for which we can use a fixed-point argument to show existence of weak solutions.
To this aim, we introduce $\tilde A_\delta(x) := A_\delta(\tilde\psi(x))$ for any $\tilde\psi \in L^2$  and
we consider the linearised problem
\begin{equation}
    \label{eq:existence regularized equation}
    \partial_x \left(k(x)\, \tilde{A}_\delta \nabla \tilde\psi_\delta\right) + \delta \tilde\psi_\delta = 0\,,
\end{equation}
with nonlinear boundary conditions
\begin{equation*}
    - \tilde{A}_\delta \partial_x \tilde{\psi}_\delta =
    \begin{cases}
        \alpha \frac{1}{1 + e^{\tilde{\psi}_\delta - x}} & \text{ at } x = 0\,,
        \\
        \beta \frac{e^{\tilde{\psi}_\delta - x}}{1 + e^{\tilde{\psi}_\delta - x}} & \text{ at } x = L\,.
    \end{cases}
\end{equation*}
This is the Euler-Lagrange equation corresponding to the energy functional
\begin{equation*}
    \mathcal{E}[\tilde{\psi}_\delta] :=
    \frac{1}{2} \int_0^L \left( k \tilde{A}_\delta |\nabla \tilde{\psi}_\delta|^2 +
        \delta |\tilde{\psi}_\delta|^2  \right) \dd x
        - \alpha F(\tilde\psi_\delta(0)) + \beta G(\tilde\psi_\delta(L))
        = 0\,,
\end{equation*}
where $F$ and $G$ are such that
\begin{equation*}
    \partial_\psi F(\psi) = - k(0) \frac{1}{1+e^{\psi-x}}\,, \quad
    \partial_\psi G(\psi) = - k(L) \frac{e^{\psi-x}}{1+e^{\psi-x}}\,.
\end{equation*}
By the convexity of $-F$ and $G$, we have that $\mathcal E$ is convex. It is also coercive in the $H^1$-norm, since
the operator $\tilde{A}_\delta$ satisfies $\delta \le \tilde{A}_\delta \le \delta + \frac{1}{4}$.
This enough to show existence of a unique minimiser $\tilde{\psi}_\delta \in H^1$ of $\mathcal E$, and thus a weak solution of \eqref{eq:existence regularized equation}.
Then, the bounds on $k$ allow us to follow the arguments of Burger and Pietschmann and take the limit $\delta \rightarrow 0$, see \cite{burger_flow_2016}.
\end{proof}

\begin{lemma}[Uniqueness]
    \label{lem:uniqueness}
    If Equation \eqref{eq:reduced equation rho} with boundary conditions \eqref{eq:boundary conditions reduced rho}
    admits a weak solution $\rho~\in~H^1$, then it is unique.
\end{lemma}
\begin{proof}
    Our strategy is along the lines of the one in \cite{burger_flow_2016}. Suppose that $\rho_1$ and $\rho_2$ are two solutions of \eqref{eq:reduced equation rho}-\eqref{eq:boundary conditions reduced rho} and define $v = \rho_1 - \rho_2$,
    which solves
    \begin{equation*}
        \partial_x \left(k(x) \left[-\varepsilon \partial_x v + (1-\rho_1-\rho_2)v \right]\right)
    \end{equation*}
    with the boundary conditions:
    \begin{align*}
    \begin{aligned}
        \varepsilon\, \partial_x v &= \alpha v + \left(1-\rho_1-\rho_2\right) v\,, & \text{ at } x &= 0\,,
        \\
        \varepsilon\, \partial_x v &= -\beta v + \left(1-\rho_1-\rho_2\right) v\,, & \text{ at } x &= L\,.
        \end{aligned}
    \end{align*}
    Let $V$ be such that $\varepsilon \partial_x V = \left(1-\rho_1-\rho_2\right)$ and $z$ such that $v = e^V z$.
    Then, $z$ is a weak solution of
    \begin{equation}\label{eq:z}
        -\varepsilon \partial_x \left(k(x)\, e^V \partial_x z\right) = 0
    \end{equation}
    with boundary conditions
    \begin{align*}
    \begin{aligned}
        \varepsilon\, \partial_x z &= \alpha z\,, & \text{ at } x &= 0\,,
        \\
        \varepsilon\, \partial_x z &= -\beta z\,, & \text{ at } x &= L\,,
    \end{aligned}
    \end{align*}
    Multiplying \eqref{eq:z} by $z$ and using the corresponding boundary conditions yields 
    \begin{equation*}
        \varepsilon \int_0^L k(x)\, e^V \left(\partial_x z\right)^2 \; \dd x
        + \alpha k(0) e^{V(0)} z(0)^2 + \beta k(L) e^{V(L)} z(L)^2 = 0\,,
    \end{equation*}
    and therefore $z = v \equiv 0$.
\end{proof}

\subsection{Properties of solutions}
\begin{lemma}[Symmetry]
    \label{lem:symmetry}
Let $\rho\in H^1$ be a weak solution of \eqref{eq:reduced equation rho} with boundary conditions \eqref{eq:boundary conditions reduced rho}. Then, the function $\tilde\rho$ given by $\tilde\rho(x) := 1 - \rho(L-x)$ is also a weak solution with $\tilde k(x) = k(L-x)$,
$\tilde\alpha = \beta$ and $\tilde\beta = \alpha$.
\end{lemma}

\begin{proof}
The proof is straightforward using $\partial_x \tilde{\rho}(x) = \partial_x \rho(L-x)$ and $\tilde{\rho}(x)\left(1-\tilde{\rho}(x)\right) = \rho(L-x)\left(1-\rho(L-x)\right)$ in the distributional sense.
\end{proof}

\begin{lemma}[Maximum principle]
    \label{lem:maximum principle}
    Assume that $k \in C^1$. Let $\rho$ be a $C^2$ solution to \eqref{eq:reduced equation rho} with $0 < \rho < 1$. If $\rho$ attains a local minimum (maximum) at $x_m \in (0, L)$ ($x_M$ in $(0,L)$), then $\partial_x k(x_m) > 0$ ($\partial_x k(x_M) < 0$).
\end{lemma}

\begin{proof}
    Since $\partial_x\rho(x_m) = 0$, we have that
    \begin{equation*}
        \varepsilon k(x_m) \partial_{xx} \rho(x_m) - \rho(x_m) (1-\rho(x_m)) \partial_x k(x_m) = 0\,,
    \end{equation*}
    with $\partial_{xx} \rho(x_m) > 0$, and therefore $\partial_x k(x_m) > 0$. The same argument holds for $x_M$.
\end{proof}

As a direct consequence of Lemma~\ref{lem:maximum principle}, we have the following

\begin{proposition}
If $k \in C^1$ with $\partial_x k < 0$ (resp. $\partial_x k > 0$) then any solution $\rho$ of \eqref{eq:reduced equation rho} such that $\rho \in C^2$ and is bounded between $0$ and $1$ has no minimum (resp. maximum) on $(0, L)$ and is such that $\rho \ge \min\left(\alpha, 1 - \beta\right)$ (resp. $\rho \le \max\left(\alpha, 1 - \beta\right)$).
\end{proposition}

\begin{proof}
    We will only describe the case $\partial_x k<0$. By Lemma~\ref{lem:maximum principle}, $\rho$ has no minimum on $(0,L)$. If the minimum is attained for $\xi=0$, then~\eqref{eq:boundary conditions reduced rho} gives
    \begin{equation*}
        \varepsilon \partial_x \rho = \left(\rho - \alpha\right) \left(1- \rho\right) > 0\,.
    \end{equation*}
    Similarly, if the minimum is attained for $x=L$ we have
    \begin{equation*}
        \varepsilon \partial_x \rho = \rho \left(1- \rho - \beta\right) < 0\,,
    \end{equation*}
    and since $0 \le \rho \le 1$, we have that $\rho \ge \min\left(\alpha, 1-\beta\right)$.
\end{proof}

We conclude by discussing the dependence of the flux $\rJ = kj$ on $\alpha$ and $\beta$.

\begin{lemma}[Monotonicity of $\rJ$]
    \label{lem:monotonicity j}
    Let $0 < \alpha,\beta < 1$. Moreover, let $k$ be a piecewise continuous function satisfying assumptions of Theorem~\ref{thm:existence}, and $\rho$ be the corresponding solution of Equations \eqref{eq:reduced equation rho}-\eqref{eq:boundary conditions reduced rho}. Then, $k \left(-\varepsilon \partial_x \rho + \rho(1-\rho)\right) \equiv \rJ \in \mathbb{R}$, where $\rJ$ is an increasing function of $\alpha$ and $\beta$.
\end{lemma}

\begin{proof}
    We argue by contradiction and outline the argument for $\beta$ only (since it is identical for $\alpha$).
    Let $(x_i)_{1\le i \le {n-1}}$ denote the points where $k$ is discontinuous. For consistency we define $x_n = L$.
    The boundary conditions are
    \begin{equation*}
        \rJ = k(0) \alpha(1-\rho(0)) = k(L) \beta \rho(L),
    \end{equation*}
    and we consider two solutions $(\rho_0, \rJ_0)$, $(\rho_1, \rJ_1)$ corresponding to $(\alpha, \beta_0)$ and $(\alpha, \beta_1)$ respectively, with $\beta_0 < \beta_1$.
    We recall from Theorem~\ref{thm:existence} that $\rho_0$ and $\rho_1$ are in $H^1(0, L)$.
    Suppose $\rJ_0 = \rJ_1$, in which case $\rho_0(0) = \rho_1(0)$ and $\rho_0(L) > \rho_1(L)$. Both are solution
    of the following initial value problem:
    \begin{equation*}
        (P_{t_0, f_0, G}) := \left\{
    \begin{aligned}
        \partial_x f(t) &= \varepsilon^{-1} \left(G k^{-1}(t) - f(t) (1-f(t))\right)\,,
        \\
        f(t_0) &= f_0\,.
    \end{aligned}\right.
    \end{equation*}
    with $t_0=0$, $f_0 = \rho_0(0) = \rho_1(0)$ and $G = \rJ_0 = \rJ_1$.
    By the Picard-Lindel\"of theorem, $\rho_0(x) = \rho_1(x)$ on $[0, x_1)$ and on $[0, x_1]$ by continuity.
    One can repeat the argument for $(P_{x_m, \rho_0(x_m), \rJ_0})$ for $1 \le m \le n-1$,
    to eventually get $\rho_0(L) = \rho_1(L)$, which is a contradiction, so that $\rJ_0 \neq \rJ_1$.
        \par Suppose $\rJ_0 > \rJ_1$, then $\rho_0(0) < \rho_1(0)$ and $\rho_0(L) > \rho_1(L)$. We denote $x^* := \min\{x:\rho_0(x) = \rho_1(x)\}$, which is non-empty by continuity. In particular, $\rho_0(x) < \rho_1(x)$ for $x < x^*$ locally.  The function $k$ can be discontinuous at $x^*$, but it still admits a left limit $k^-(x^*)$, so one can consider the left limits
        \begin{equation*}
            \partial_x^- \rho_i(x^*) = \lim_{x\rightarrow x^*} \rho_i(x)\,.
        \end{equation*}
        Since $\rJ_0 > \rJ_1$, we have that $\partial_x^- \rho(x^*) < \partial_x^- \rho_1(x^*)$, which is a contradiction.
\end{proof}


\section{Stationary profiles for monotonic width corridors}\label{sec:vanvis}
In this section, we investigate the stationary profiles of Equation \eqref{eq:reduced equation rho}-\eqref{eq:boundary conditions reduced rho} for closing corridors, i.e.~we consider $C^1$ positive functions $k$, which satisfy
\begin{equation} \label{eq:g}
    g(x):=\frac{1}{k(x)} \frac{dk(x)}{dx} < 0.
\end{equation}
Note that due to symmetry (using Lemma ~\ref{lem:symmetry}) all results of this section generalise to $g > 0$. We will, however, only consider the closing case, as it is of greater interest. Throughout this section we assume w.l.o.g.~that $L=1$.
In the following, we construct stationary profiles to Equations \eqref{eq:reduced equation rho}-\eqref{eq:boundary conditions reduced rho} using GSPT. We identify $6$ regions in the $(\alpha,\beta)$ parameter space, in which the stationary profiles have the same structure (with respect to boundary layers). Interestingly, it turns out that the location of the boundary between these regions depends only on the values of $k$ at the boundary of the spatial domain. \\
We add the trivial dynamics of $x$ to the system by introducing $\xi=x$ and including the equation $\dt{\xi}=1$, where $\dt{\phantom{x}}=\frac{d}{dx}$. This transforms Equations \eqref{eq:reduced equation rho}-\eqref{eq:boundary conditions reduced rho} into a system of first order ODEs:
\begin{equation} \label{eq:fosk}
 \begin{aligned}
  \dt{j} &= -g(\xi) j, \\
  \dt{\xi} &= 1, \\
  \varepsilon \dt{\rho} &= \rho(1-\rho)-j.
 \end{aligned}
\end{equation}
In this framework, the boundary conditions \eqref{eq:boundary conditions reduced rho} become
\begin{equation}
\begin{aligned}
    j &= \alpha \left(1 - \rho\right)
 & \text{ at } \xi &= 0,
 \\
    j &= \beta \rho
 & \text{ at } \xi &= 1.
\end{aligned}
\label{eq:bc_sf}
\end{equation}
System \eqref{eq:fosk} is a slow-fast system in slow 
form (\ref{eq:slow}), where $\rho$ is the only fast variable (since its dynamics evolve proportionally to $\frac{1}{\varepsilon} \gg 1$), while $j$ and $\xi$ are the slow ones  \cite{Fenichel_1979,Jones_1995,Kuehn_2015}.
As usual in GSPT, we rescale $x$ to the fast variable $\chi = \frac{x}{\varepsilon}$, and using the notation $'=\frac{d}{d \chi}$, we can rewrite System \eqref{eq:fosk} as
\begin{equation} \label{eq:fosk_fast}
 \begin{aligned}
  j' &= -\varepsilon g(\xi) j, \\
  \xi' &= \varepsilon, \\
  \rho' &= \rho(1-\rho)-j,
 \end{aligned}
\end{equation}
which describes the evolution of \eqref{eq:fosk} on the fast scale. Note that for $\varepsilon > 0$, Systems \eqref{eq:fosk} and \eqref{eq:fosk_fast} are equivalent. Using GSPT methods to investigate fast-slow systems is particularly advantageous, as it allows us to obtain detailed information about the structure of the solutions to the original problem \eqref{eq:reduced equation rho}-\eqref{eq:boundary conditions reduced rho} by separately analysing the singular limit $\varepsilon \to 0$ on the slow scale \eqref{eq:fosk} and on the fast scale \eqref{eq:fosk_fast}, and then subsequently matching the results. Letting $\varepsilon \to 0$ in Equations \eqref{eq:fosk} and \eqref{eq:fosk_fast} leads to two limiting subproblems -- i.e.~the \emph{reduced} problem and the \emph{layer} problem, respectively -- which are simpler to analyse. The layer problem ($\varepsilon=0$ in \eqref{eq:fosk_fast}) is given by
\begin{equation} \label{eq:laypb_k}
 \begin{aligned}
  j' &= 0, \\
  \xi' &= 0, \\
  \rho' &= \rho(1-\rho)-j, 
 \end{aligned}
\end{equation}
and describes the dynamics of the fast variable $\rho$ for fixed $j$ and $\xi$ values. The manifold of its equilibria is known as the \emph{critical manifold}
 \begin{equation} \label{eq:crit_man}
 \mathcal{C}_0 := \left\{ (j,\xi,\rho)~:~j=\rho(1-\rho) \right\}.
 \end{equation}
 It consists of the union of two submanifolds $\mathcal{C}_0^a$ ($\rho > \frac12$) and $\mathcal{C}_0^r$ ($\rho < \frac12$) -- which are attracting and repelling, respectively -- and a line of fold points
 \begin{equation} \label{eq:foldL}
  S := \left\{ (j,\xi,\rho)~:~j=\frac14,~\rho=\frac12 \right\},
 \end{equation}
as shown in Figure \ref{fig:C0_ff}.
 \begin{figure}[H]
    \centering
    \def\svgwidth{.4\textwidth}
\begingroup%
  \makeatletter%
  \providecommand\color[2][]{%
    \errmessage{(Inkscape) Color is used for the text in Inkscape, but the package 'color.sty' is not loaded}%
    \renewcommand\color[2][]{}%
  }%
  \providecommand\transparent[1]{%
    \errmessage{(Inkscape) Transparency is used (non-zero) for the text in Inkscape, but the package 'transparent.sty' is not loaded}%
    \renewcommand\transparent[1]{}%
  }%
  \providecommand\rotatebox[2]{#2}%
  \newcommand*\fsize{\dimexpr\f@size pt\relax}%
  \newcommand*\lineheight[1]{\fontsize{\fsize}{#1\fsize}\selectfont}%
  \ifx\svgwidth\undefined%
    \setlength{\unitlength}{360bp}%
    \ifx\svgscale\undefined%
      \relax%
    \else%
      \setlength{\unitlength}{\unitlength * \real{\svgscale}}%
    \fi%
  \else%
    \setlength{\unitlength}{\svgwidth}%
  \fi%
  \global\let\svgwidth\undefined%
  \global\let\svgscale\undefined%
  \makeatother%
  \begin{picture}(1,1.11111107)%
    \lineheight{1}%
    \setlength\tabcolsep{0pt}%
    \put(0,0){\includegraphics[width=\unitlength,page=1]{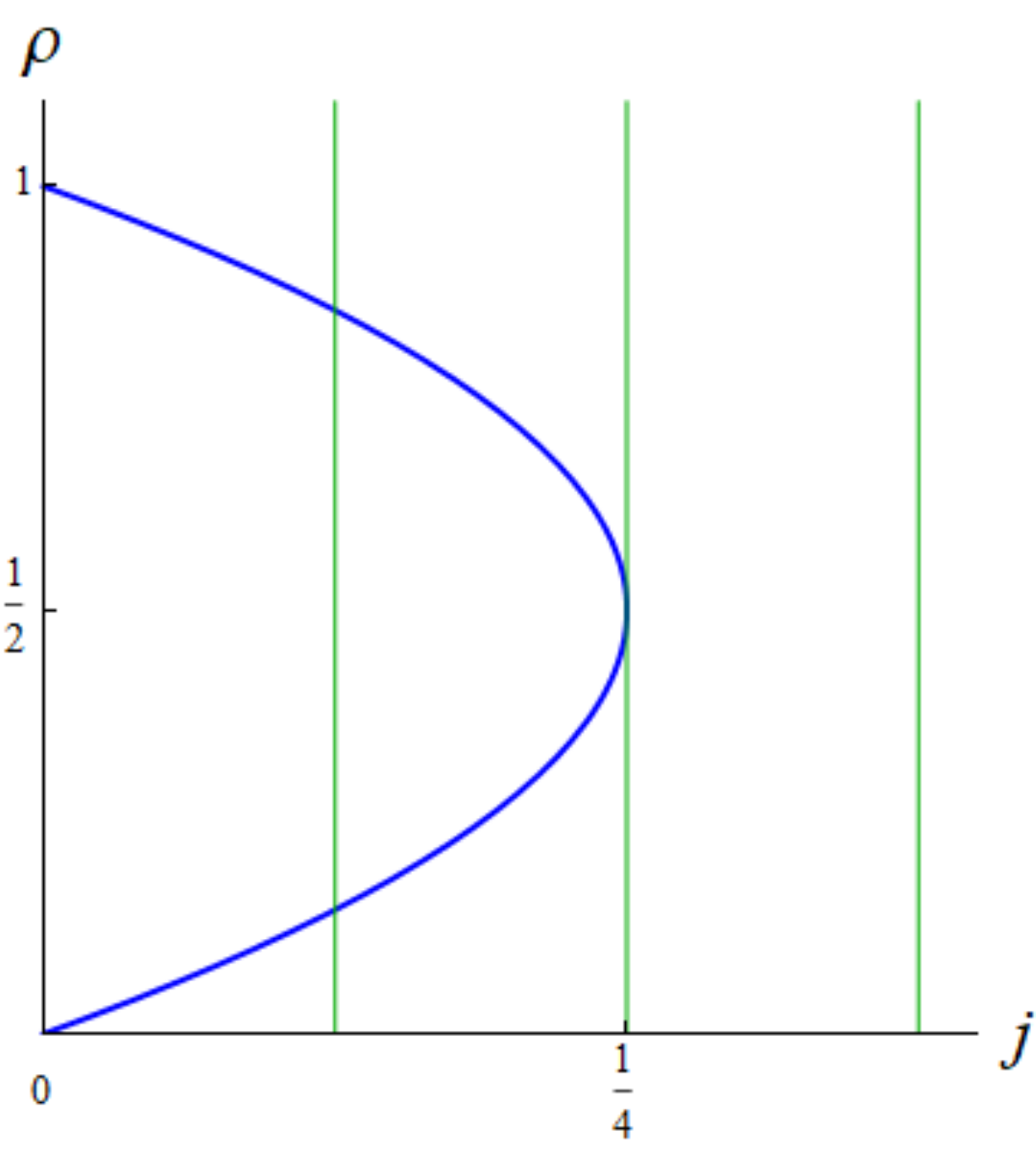}}%
    \put(0.50235224,0.86817624){\color[rgb]{0,0,0}\makebox(0,0)[lt]{\lineheight{1.25}\smash{\begin{tabular}[t]{l}$\mathcal{C}_0^a$\end{tabular}}}}%
    \put(0.50084279,0.16135979){\color[rgb]{0,0,0}\makebox(0,0)[lt]{\lineheight{1.25}\smash{\begin{tabular}[t]{l}$\mathcal{C}_0^r$\end{tabular}}}}%
    \put(0,0){\includegraphics[width=\unitlength,page=2]{C0_b_ff_2_new_2.pdf}}%
    \put(0.63921324,0.51966648){\color[rgb]{0,0,0}\makebox(0,0)[lt]{\lineheight{1.25}\smash{\begin{tabular}[t]{l}$S$\end{tabular}}}}%
    \put(0,0){\includegraphics[width=\unitlength,page=3]{C0_b_ff_2_new_2.pdf}}%
  \end{picture}%
\endgroup%

    \caption{Fast dynamics in $(j,\rho)$-space for each fixed value of $\xi$. The blue curve represents $\mathcal{C}_0$, divided in two branches $\mathcal{C}_0^a$ (attracting) and $\mathcal{C}_0^a$ (repelling). The green lines indicate orbits of the layer problem \eqref{eq:laypb_k}, while the blue dot represents the line of fold points $S$ \eqref{eq:foldL}.}
    \label{fig:C0_ff}
 \end{figure}
The reduced problem reads
\begin{subequations} \label{eq:k_redpb}
 \begin{align}
  \dt{j} &= -g(\xi) j, \label{eq:kj_const} \\
  \dt{\xi} &= 1.
 \end{align}
\end{subequations}
This system describes the dynamics of the slow variables $j$ and $\xi$ along $\mathcal{C}_0$. Equation \eqref{eq:crit_man} implies  that $\dt j = (1-2 \rho) \dt \rho$, hence it is advantageous to rewrite \eqref{eq:k_redpb} in terms of the variables $\xi$ and $\rho$ (see Figure \ref{fig:slow_flow})
\begin{equation} \label{eq:k_redpb_xrho}
\begin{aligned}
    \dt \rho &= -g(\xi)\frac{\rho (1-\rho)}{1-2 \rho},\\
    \dt \xi &= 1.
\end{aligned}
\end{equation}
System \eqref{eq:k_redpb_xrho} is singular on the fold line $S$, i.e.~for $\rho=\frac12$. In order to desingularise this system, we multiply the right hand-sides by $1-2\rho$, which is positive for $\rho < \frac12$. For $\rho > \frac12$ we have to reverse the evolution direction. This gives the following system, which has the same orbits as \eqref{eq:k_redpb_xrho} away from $S$:
\begin{equation} \label{eq:k_redpb_rholg12}
 \begin{aligned}
  \dt{\rho} &= -g(\xi) \rho(1-\rho), \\
  \dt{\xi} &= 1-2\rho,
 \end{aligned}
 \qquad \text{if } \rho<\frac12,
 \qquad
 \begin{aligned}
  \dt{\rho} &= g(\xi) \rho(1-\rho), \\
  \dt{\xi} &= 2\rho-1,
 \end{aligned}
 \qquad \hspace{.2cm}\text{if } \rho>\frac12.
\end{equation}
We observe that the reduced flow \eqref{eq:k_redpb_xrho} satisfies $\frac{d \rho}{d \xi}<0$ for $\rho>\frac12$ and $\frac{d \rho}{d \xi}>0$ for $\rho < \frac12$, i.e.~$\rho$ is decreasing on $\mathcal{C}_0^a$ and increasing on $\mathcal{C}_0^r$. Note that \eqref{eq:k_redpb_rholg12} is symmetric with respect to reflection on the fold line $\rho=\frac12$.
 \begin{figure}[H]
    \centering
    \includegraphics[scale=.5]{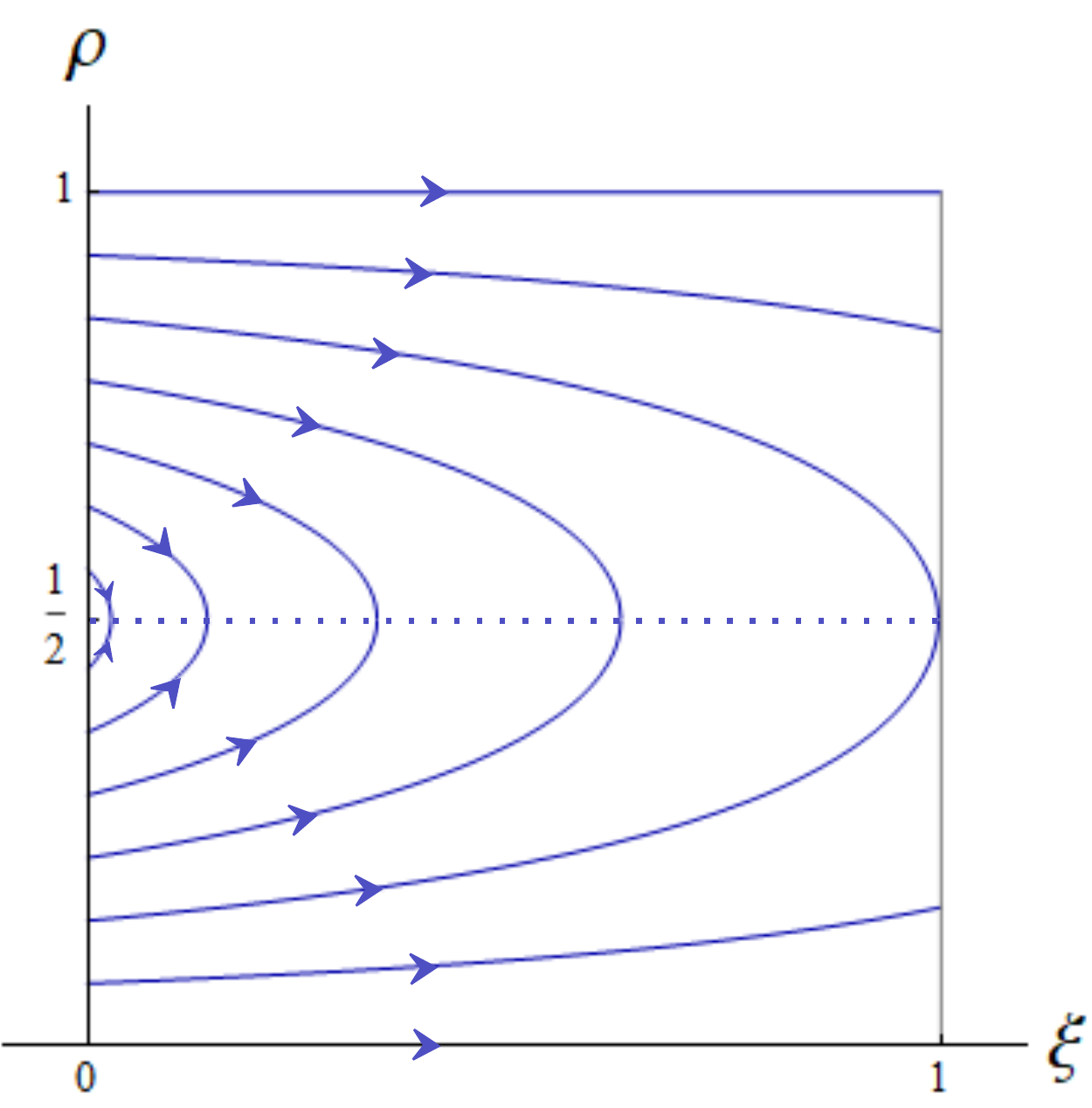} 
    \caption{Schematic representation of the reduced flow \eqref{eq:k_redpb_xrho} on $\mathcal{C}_0$ (blue lines) for $k(x)=1-\frac{x}{2}$.
    The dashed line indicates the line of fold points $S$.}
    \label{fig:slow_flow}
 \end{figure}
We start by constructing six singular orbits, denoted by $\Gamma^i$, $i=1,\dots,6$, in the following. These singular orbits are suitable combinations of orbit segments of the layer problem \eqref{eq:laypb_k} and orbits of the reduced flow \eqref{eq:k_redpb} which satisfy the boundary conditions \eqref{eq:boundary conditions reduced rho}. Their construction is based on a shooting strategy: we evolve the manifold of boundary conditions at $\xi=0$ forward and check whether it intersects the manifold of boundary conditions at $\xi=1$. This constructive procedure allows us to also identify the initial and final values of $\rho$ (which we call $\rho_0$ and $\rho_1$ in the following) for $\varepsilon=0$. These singular orbits will later be used in the construction of genuine solutions to \eqref{eq:reduced equation rho}-\eqref{eq:boundary conditions reduced rho} for $0 < \varepsilon \ll 1$ (see Theorem \ref{thm:k}). For a visualisation of this strategy in a specific case, we refer to Figure \ref{fig:singsol_1} - the details necessary to fully understand it will be given in the following. \\
In the dynamical systems framework, boundary conditions \eqref{eq:boundary conditions reduced rho} correspond to two lines in the $(j,\xi,\rho)$-space, satisfying $j=\alpha(1-\rho)$ at $\xi=0$ and $j=\beta \rho$ at $\xi=1$, respectively. However, due to the fast-slow structure, the set of admissible boundary conditions is restricted to  (see Figures \ref{fig:L_comb}-\ref{fig:R_comb})
\begin{subequations} \label{eq:LR}
\begin{align}
 \mathcal{L} &:= \left\{ \left( \alpha(1-s), \, 0, \, s \right) \ : \ \rho_\alpha \leq s \leq 1 \right\}, \\
 \mathcal{R} &:= \left\{ \left( \beta t, \, 1, \, t \right) \ : \ 0 \leq t \leq \rho_\beta \right\}.
\end{align}
\end{subequations}
Here
\begin{equation} \label{eq:rho_alpha}
 \rho_\alpha= \begin{cases} \alpha &\quad \text{if } \alpha \leq \frac12, \\ 1-\frac{1}{4 \alpha} &\quad \text{if } \alpha \geq \frac12, \end{cases}
\end{equation}
and
\begin{equation} \label{eq:rho_beta}
 \rho_\beta= \begin{cases} 1-\beta &\quad \text{if } \beta \leq \frac12, \\ \frac{1}{4\beta} &\quad \text{if } \beta \geq \frac12. \end{cases}
\end{equation}
The lower and upper bounds $\rho_\alpha$ and $\rho_\beta$ for the density $\rho$ are caused by the fast-slow structure of the flow: if we would consider a starting point $(\alpha(1-\rho), \, 0, \, \rho)$ with $0 \leq \rho < \rho_\alpha$, the orbit would be immediately repelled to infinity from $\mathcal{C}_0$, hence connecting to the boundary conditions at $\xi=1$ is impossible. Analogously, points satisfying $(\beta \rho, \, 1, \, \rho)$ with $\rho_\beta < \rho \leq 1$ cannot be endpoints of the singular orbits, since they are repelling for the layer problem.\\
Thus, the initial and final points of the singular orbits -- $p_0$ and $p_1$, respectively -- must satisfy
\begin{equation} \label{eq:p_0f}
 p_0 \in \mathcal{L}, \textrm{ and } p_1 \in \mathcal{R}.
\end{equation}
The manifold $\mathcal{L}$ intersects with $\mathcal{C}_0$ at $(0,0,1)$ and
 \begin{equation} \label{eq:p_l}
 l=(\alpha(1-\alpha),0,\alpha),
 \end{equation}
while $\mathcal{R}$ intersects with $\mathcal{C}_0$ at $(0,1,0)$ and
 \begin{equation} \label{eq:p_r}
 r=(\beta(1-\beta),1,1-\beta).
 \end{equation}
For $\varepsilon=0$, the variable $\xi$ evolves only on $\mathcal{C}_0$ according to the reduced flow \eqref{eq:k_redpb_xrho}. Therefore, in order for the singular solution to evolve from $\xi=0$ to $\xi=1$, we must connect $\mathcal{L}$ and $\mathcal{R}$ to $\mathcal{C}_0$. The points $l$ and $r$ already belong to $\mathcal{C}_0$. Other points on $\mathcal{L}$ and $\mathcal{R}$ can reach $\mathcal{C}_0$ using the layer problem \eqref{eq:laypb_k}. Tracking the evolution of $\mathcal{L}$ by means of the layer problem at $\xi=0$ until it reaches $\mathcal{C}_0$, and analogously the evolution of $\mathcal{R}$ backwards until the layer problem at $\xi=1$ intersects $\mathcal{C}_0$, yields two sets (shown in Figures \ref{fig:L_comb}-\ref{fig:R_comb}):
\begin{subequations} \label{eq:M010_i}
\begin{align}
 \mathcal{L}^+ &:= \left\{ \left( \alpha(1-s), \, 0, \, \rho(0,s) \right) \ : \ \rho_\alpha \leq s \leq 1 \right\},
 \label{eq:L+} \\
 \mathcal{R}^- &:= \left\{ \left( \beta t, \, 1, \, \frac12\left(1-\sqrt{1-4\beta t}\right) \right) \ : \ 0 \leq t \leq \rho_\beta \right\}.
 \label{eq:R-}
\end{align}
\end{subequations}
In the following, we use the symbol $\rho(0,s)$ to indicate the $\rho$-value reached by the point $(\alpha(1-s), \, 0, \, s)$ after its transition from $\mathcal{L}$ to $\mathcal{C}_0$ by means of the layer problem. If the solution $(\xi,\rho)$ of the reduced flow \eqref{eq:k_redpb_rholg12} starting at $(0,\rho(0,s))$ reaches $\xi=1$, we denote its value of $\rho$ at $\xi=1$ by $\rho(1,s)$. 
When $\alpha < \frac12$, the reduced flow can either start on $\mathcal{L}^+$ or at $l$, while for $\alpha > \frac12$ it must start on $\mathcal{L}^+$. Analogously, when $\beta<\frac12$, the reduced flow can either end on $\mathcal{R}^-$ or at $r$, while for $\beta > \frac12$ it must end on $\mathcal{R}^-$.
 \begin{figure}[H] 
 \centering
 \begin{minipage}{.4\textwidth}
  \centering
  \def\svgwidth{.8\textwidth}
  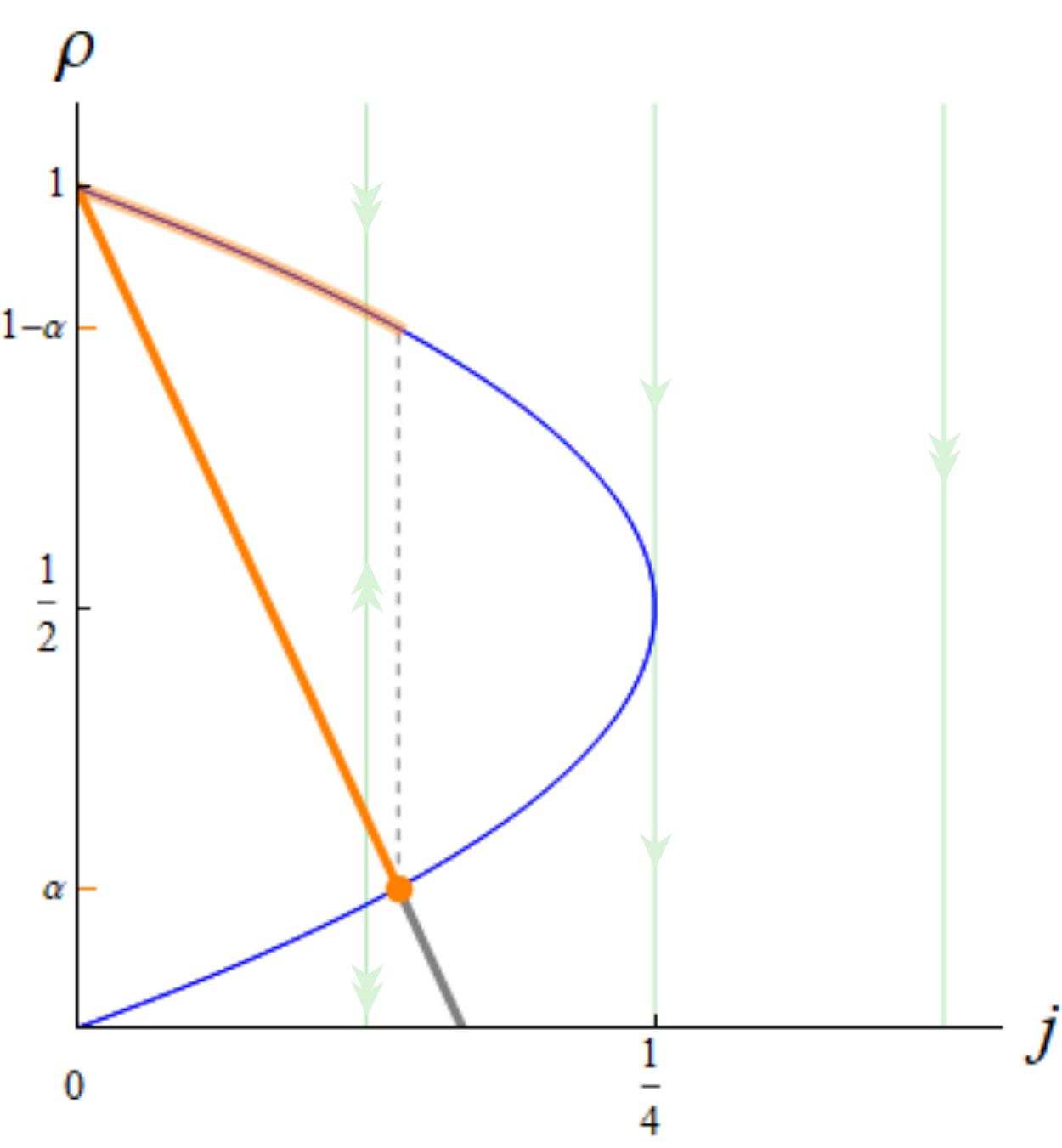\\
  (a)
 \end{minipage}
 \begin{minipage}{.4\textwidth}
  \centering
  \def\svgwidth{.8\textwidth}
  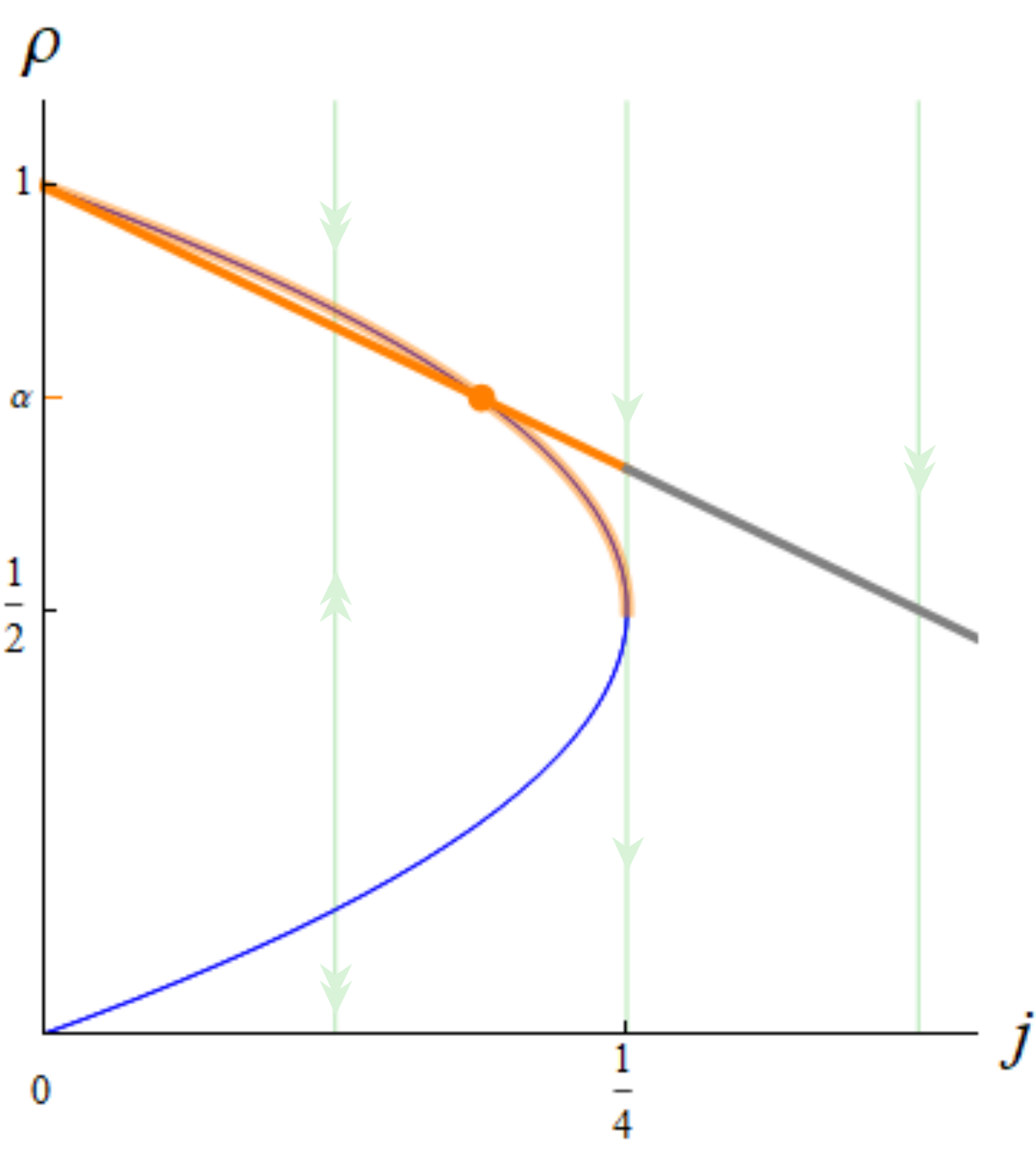\\
  (b)
 \end{minipage}
 \caption{Schematic representation of $\mathcal{L}$ (orange line) and $\mathcal{L}^+$ (orange curve) for (a) $0 <\alpha < \frac12$ and (b) $\frac12 <\alpha < 1$. The orange dot corresponds to $l$, the blue curve represents $\mathcal{C}_0$, and the green lines correspond to the orbits of the layer problem.}
 \label{fig:L_comb}
 \end{figure}
 
 \begin{figure}[H] 
 \centering
 \begin{minipage}{.4\textwidth}
  \centering
  \def\svgwidth{.8\textwidth}
  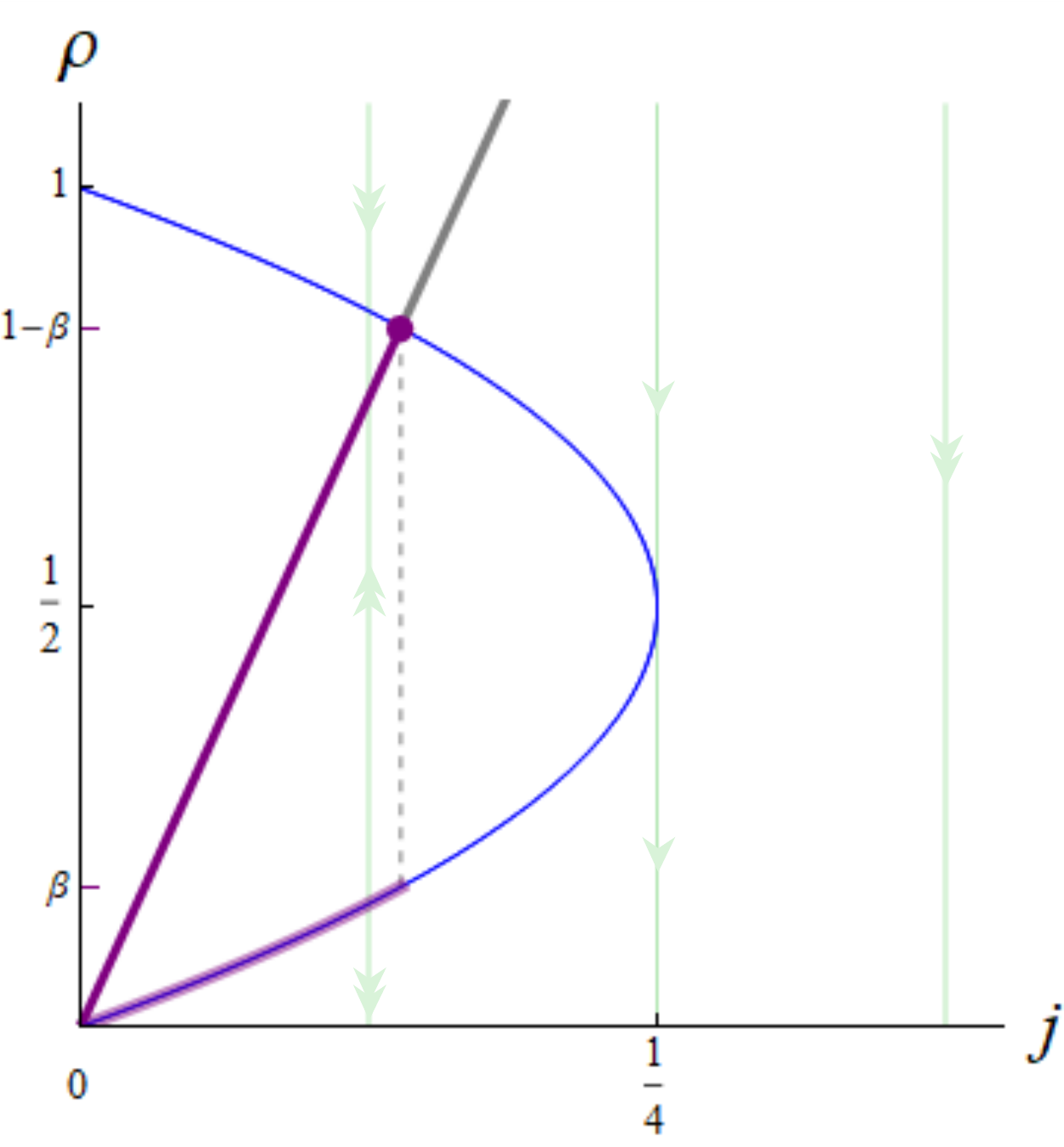\\
  (a)
 \end{minipage}
 \begin{minipage}{.4\textwidth}
  \centering
  \def\svgwidth{.8\textwidth}
  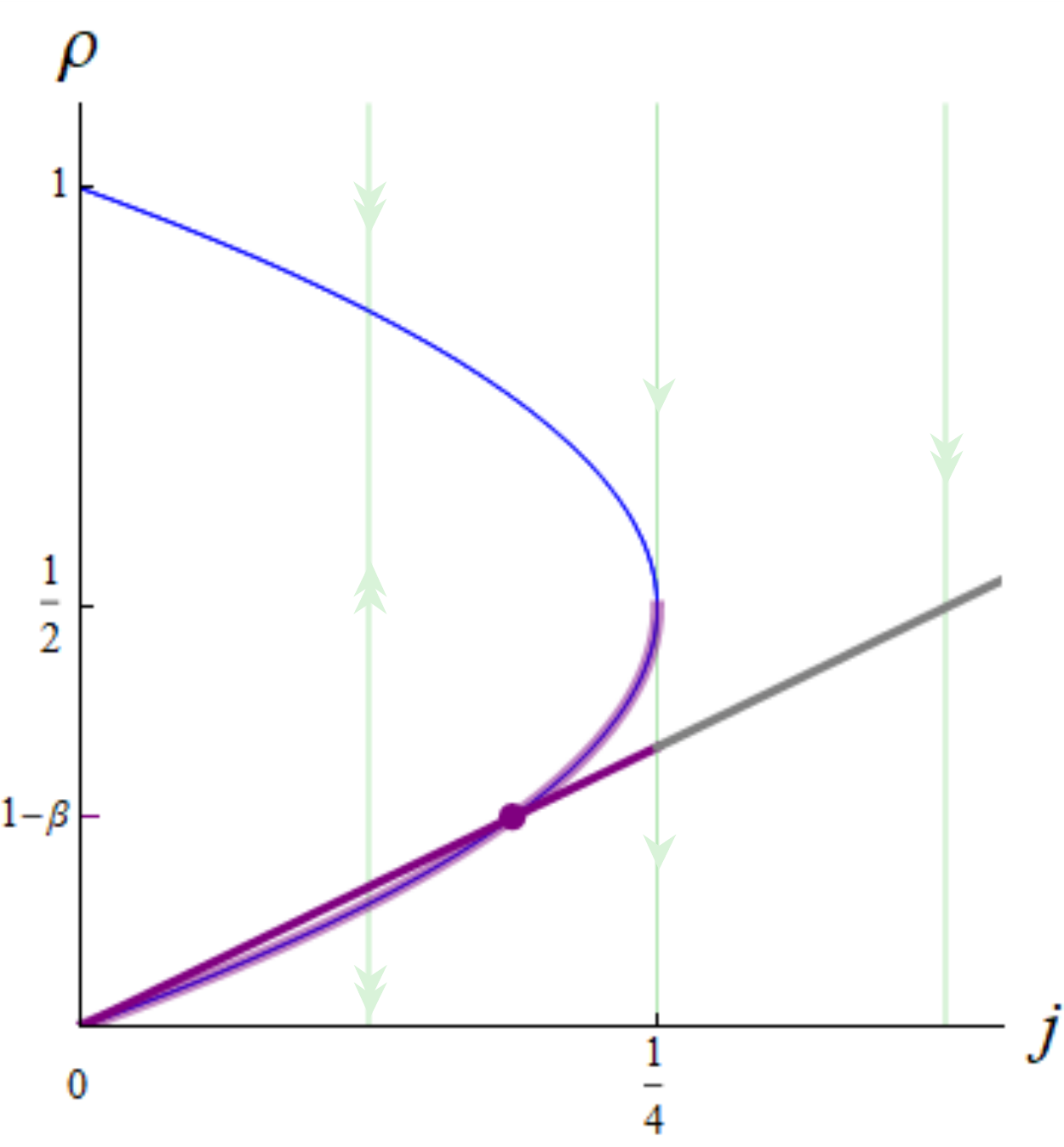\\
  (b)
 \end{minipage}
 \caption{Schematic representation of $\mathcal{R}$ (purple line) and $\mathcal{R}^-$ (purple curve) for (a) $0 < \beta < \frac12$ and (b) $\frac12 <\beta < 1$. The purple dot corresponds to $r$, the blue curve represents $\mathcal{C}_0$, and the green lines correspond to the orbits of the layer problem.}
 \label{fig:R_comb}
 \end{figure}

Based on this geometric interpretation of the boundary conditions, we proceed with the construction of the singular orbits by connecting $\mathcal{L}^+ \cup l$ and  $\mathcal{R}^- \cup r$ by means of the reduced flow \eqref{eq:k_redpb_xrho} on $\mathcal{C}_0$. In doing so, we first let $\mathcal{L}^+$ in \eqref{eq:L+} flow by means of the reduced flow until $\xi=1$: we call the corresponding set $\mathcal{L}_1^+$.
\begin{equation} \label{eq:L+_1}
 \mathcal{L}_1^+ := \left\{ \left( \rho(1,s)(1-\rho(1,s)),\, 1,\, \rho(1,s)  \right) \ : \ \rho_\alpha \leq s \leq 1 \right\}.
\end{equation}
If $\alpha \geq \frac12$ then $l \in \mathcal{L}^+$, and the evolution of $l$ by means of the reduced flow is already included in $\mathcal{L}_1^+$. If $\alpha < \frac12$ then $l \notin \mathcal{L}^+$, and therefore the corresponding point at $\xi=1$ must be defined separately as
\begin{equation} \label{eq:l_1}
 l_1 := \left\{ (\rho(1,\alpha)(1-\rho(1,\alpha)), \, 1,\, \rho(1,\alpha)) \right\}.
\end{equation}
Due to the structure of the reduced flow, this point does not exist for all values of $\alpha < \frac12$ (see Remark \ref{rem:l1}).\\
A singular orbit is then given by matching the slow and fast pieces obtained by investigating the reduced and layer problems, respectively. More specifically, a singular orbit exists if and only if the intersection between the sets $\mathcal{L}_1^+ \cup l_1$ and $\mathcal{R}^- \cup r$ is non-empty, and it is unique if this intersection consists of one point.\\
Our analysis of the existence and structure of singular orbits is based on six special orbits $S_1$, $S_2$, $S_3$, $\tilde{S}_1$, $\tilde{S}_2$, $\tilde{S}_3$ of the reduced flow, where $S_2$, $\tilde{S}_2$ depend on $\alpha$ and $S_3$, $\tilde{S}_3$ on $\beta$ (see Figure \ref{fig:sf_sp}):
\begin{itemize}
 \item The orbit $S_1$ is defined as the one starting at $\xi=0$ at a $\rho$-value below $\frac12$ and ending at the fold line $S$, that is $\rho=\frac12$ at $\xi=1$. We refer to the corresponding initial value at $\xi=0$ by $\rho_f$.
 \item The orbit $S_2$ is defined as the one starting at $\rho=\alpha$ at $\xi=0$ with $\alpha \in (0,1)$. For $\alpha \leq \rho_f$ or $\alpha \geq  1-\rho_f$, the corresponding final value of $\rho$ at $\xi=1$ is denoted by $\rho^\ast(\alpha)$. For $\rho_f \leq \alpha \leq 1-\rho_f$, $S_2$ ends on the fold line $S$.
 \item The orbit $S_3$ is defined as the one ending at $\rho=1-\beta$ at $\xi=1$ with $\beta \in (0,1)$. For $\beta \leq \rho^\ast(\alpha)$ or $\beta \geq 1-\rho^\ast(\alpha)$, the corresponding initial value of $\rho$ at $\xi=0$ is denoted by $\rho_\ast(\beta)$. When $\beta = \frac12$, we define $S_3 = \tilde{S}_1$.
 \item For $i=1,2,3$, we define $\tilde{S}_i$ as the reflection of the orbit $S_i$ with respect to $\rho=\frac12$.
\end{itemize}
Depending on the values of $\alpha$ and $\beta$, one of these orbits corresponds to the slow part of the singular orbits we will construct.\\
Changing $\alpha$ and $\beta$ influences the orbits $S_2$, $\tilde{S}_2$ and $S_3$, $\tilde{S}_3$. We will show in the following that the $\alpha$, $\beta$ dependent mutual position of these orbits determines the type of singular solution of the boundary value problem.
 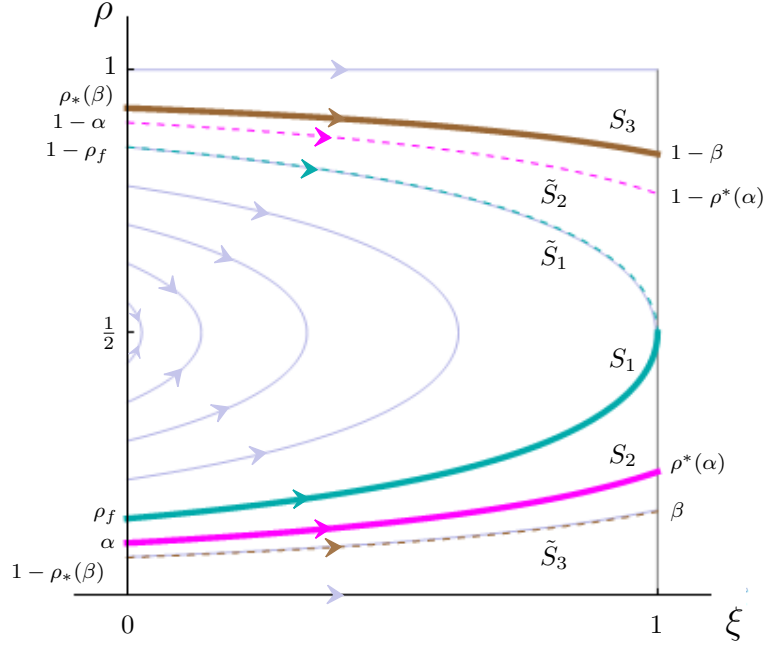
\begin{figure}[!ht]
  \centering
  	\begin{overpic}[scale=0.7]{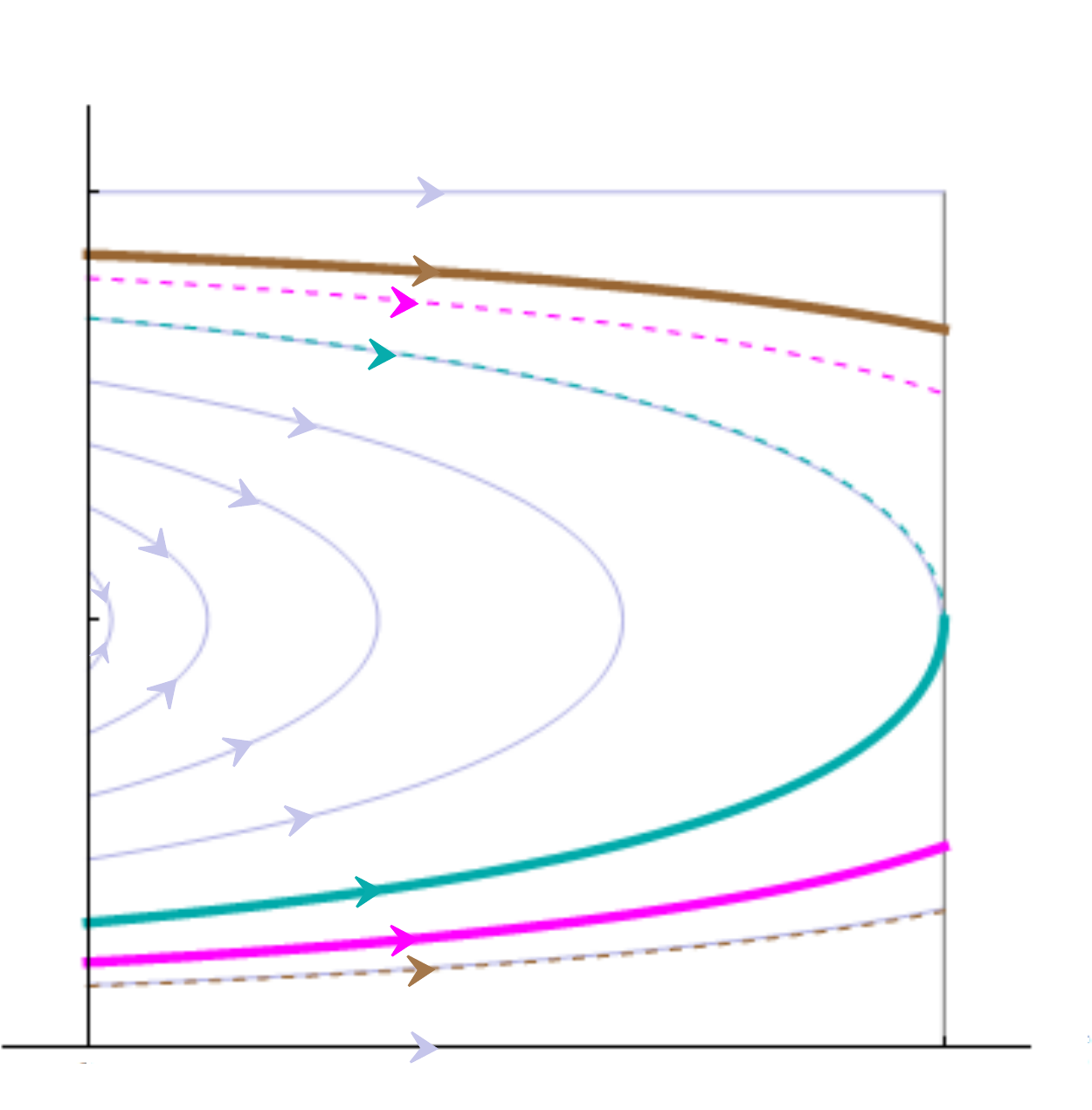}
  	\put(7,0){$0$}
  	\put(84,0){$1$}
  	\put(95,0){\Large$\xi$}
	\put(-9,8){\footnotesize $1-\rho_\ast(\beta)$}
    \put(4,12.5){\footnotesize $\alpha$}
    \put(3,17){\footnotesize $\rho_f$}
    \put(4,43){$\frac12$}
    \put(-4,70){\footnotesize $1-\rho_f$}
    \put(-3,74){\footnotesize $1-\alpha$}
    \put(-2,78){\footnotesize $\rho_\ast(\beta)$}
    \put(4.5,82){$1$}
    \put(3,90){\Large$\rho$}
    \put(87,17){\footnotesize $\beta$}
    \put(87,24){\footnotesize $\rho^\ast(\alpha)$}
    \put(87,63){\footnotesize $1-\rho^\ast(\alpha)$}
    \put(87,69.5){\footnotesize $1-\beta$}
    \put(78,39){$S_1$}
    \put(68,54){$\tilde{S}_1$}
    \put(78,25){$S_2$}
    \put(68,63){$\tilde{S}_2$}
    \put(78,74){$S_3$}
    \put(68,10){$\tilde{S}_3$}
	\end{overpic}
  \caption{Schematic illustration of the special orbits $S_1$, $S_2$, $S_3$, $\tilde{S}_1$, $\tilde{S}_2$, $\tilde{S}_3$ in $(\xi,\rho)$-space in the case $1-\rho_\ast(\beta)<\alpha<\rho_f$ and $\beta < \frac12$. The orbit $S_1$ (solid cyan curve) connects $(0,\rho_f)$ and $(1,\frac12)$, the orbit $S_2$ (solid magenta curve) connects $(0,\alpha)$ and $(1,\rho^\ast(\alpha))$, and the orbit $S_3$ (solid brown curve) connects $(0,\rho_\ast(\beta))$ and $(1,1-\beta)$ with $\alpha,\beta < \frac12$ as in \eqref{eq:sp_rho}. The dashed curves represent the orbits $\tilde{S}_1$ (cyan), $\tilde{S}_2$ (magenta), $\tilde{S}_3$ (brown), which are symmetric to the corresponding solid ones $S_1$, $S_2$, $S_3$ with respect to $\rho=\frac12$. Note that for $\alpha > \frac12$ the orbit $S_2$ lies in $\mathcal{C}_0^a$, and for $\beta > \frac12$ the orbit $S_3$ lies in $\mathcal{C}_0^r$.}
  \label{fig:sf_sp}
 \end{figure}
The respective values of $\rho_f$, $\rho^\ast(\alpha)$ and  $\rho_\ast(\beta)$
can be computed explicitly:
\begin{subequations} \label{eq:sp_rho}
\begin{align}
 \rho_f &:= \frac12 \left(1-\sqrt{1-\frac{k(1)}{k(0)}} \right), \\
 \rho^\ast(\alpha) &:= \frac12 \left( 1 - \sqrt{1-4\alpha(1-\alpha) \frac{k(0)}{k(1)}} \right), \\
 \rho_\ast(\beta) &:= \frac12 \left( 1 + \sqrt{1-4\beta(1-\beta) \frac{k(1)}{k(0)}} \right).
\end{align}
\end{subequations}
Note that $\alpha=1-\rho_\ast(\beta)$ is equivalent to $\beta=\rho^\ast(\alpha)$.
\begin{remark} \label{rem:l1}
 The point $l_1$ defined in \eqref{eq:l_1} exists if and only if $\alpha \leq \rho_f$.
\end{remark}
\noindent
Based on this, we divide the $(\alpha,\beta)$-parameter space into six regions $\mathcal{G}_i$, $i=1,\dots,6$ defined via the following curves $\gamma_{ij}$ (here the indices refer to the adjacent regions):
\begin{subequations}\label{eq:g_bound}
 \begin{align}
  \gamma_{12} &:= \left\{ (\alpha,\beta) \ : \;  0< \alpha  < \rho_f, \, \beta = 1-\rho^\ast(\alpha) \right\},\\
  \gamma_{15} &:= \left\{ (\alpha,\beta) \ : \; 0< \alpha < \rho_f, \, \beta = \rho^\ast(\alpha) \right\},\\
  \gamma_{23} &:= \left\{ (\alpha,\beta) \ : \; \alpha = \rho_f, \, \frac12 < \beta < 1 \right\},\\
  \gamma_{34} &:= \left\{ (\alpha,\beta) \ : \; \alpha = 1-\rho_f, \, \frac12 < \beta < 1 \right\},\\
  \gamma_{35} &:= \left\{ (\alpha,\beta) \ : \; \rho_f < \alpha < 1-\rho_f, \, \beta = \frac12 \right\},\\
  \gamma_{46} &:= \left\{ (\alpha,\beta) \ : \; 1-\rho_f < \alpha < 1, \, \beta = \frac12 \right\},\\
  \gamma_{56} &:= \left\{ (\alpha,\beta) \ : \; \alpha=\rho_\ast(\beta), \, 0 < \beta < \frac12 \right\}.
  \end{align}
 \end{subequations}
The above curves correspond to situations where some of the orbits $S_i$, $\tilde{S}_i$, $i=1,2,3$ defined above coincide. In particular:
\begin{itemize}
    \item for $(\alpha,\beta) \in \gamma_{12}$, we have $S_2=S_3$ (lying in $\mathcal{C}_0^r$);
    \item for $(\alpha,\beta) \in \gamma_{15}$, we have $S_2=\tilde{S}_3$;
    \item for $(\alpha,\beta) \in \gamma_{23}$, we have $S_1=S_2$;
    \item for $(\alpha,\beta) \in \gamma_{34}$, we have $\tilde{S}_1=S_2$;
    \item for $(\alpha,\beta) \in \gamma_{35} \cup \gamma_{46}$, we have $\tilde{S}_1=S_3$;
    \item for $(\alpha,\beta) \in \gamma_{56}$, we have $S_2=S_3$ (lying in $\mathcal{C}_0^a$).
\end{itemize}
\begin{remark}
Whenever two orbits coincide, their symmetric reflections with respect to $\rho=\frac12$ coincide as well.
\end{remark}
The seven curves in \eqref{eq:g_bound} split $(0,1)^2$ into $6$ regions $\mathcal{G}_i$, $i=1,\dots,6$ (shown in Figure \ref{fig:bd_sing}):
\begin{subequations} \label{eq:regions}
\begin{align}
    \mathcal{G}_1 &:= \left\{ (\alpha,\beta) \ : \; 0 < \alpha < \rho_f, \, \rho^\ast(\alpha) < \beta < 1-\rho^\ast(\alpha) \right\} \\
    \mathcal{G}_2 &:= \left\{ (\alpha,\beta) \ : \; 0 < \alpha < \rho_f, \, 1-\rho^\ast(\alpha) < \beta < 1 \right\}, \\
    \mathcal{G}_3 &:= \left\{ (\alpha,\beta) \ : \; \rho_f < \alpha < 1-\rho_f, \, \frac12 < \beta < 1 \right\}, \\
    \mathcal{G}_4 &:= \left\{ (\alpha,\beta) \ : \; 1-\rho_f < \alpha < 1, \, \frac12 < \beta < 1 \right\}, \\
    \mathcal{G}_5 &:= \left\{ (\alpha,\beta) \ : \; 1-\rho_\ast(\beta) < \alpha < \rho_\ast(\beta), \, 0 < \beta < \frac12 \right\}, \\
    \mathcal{G}_6 &:= \left\{ (\alpha,\beta) \ : \; \rho_\ast(\beta) < \alpha < 1, \, 0 < \beta < \frac12 \right\}.
\end{align}
\end{subequations}

\begin{figure}[!ht]
 \centering
\includegraphics[width=0.485\textwidth]{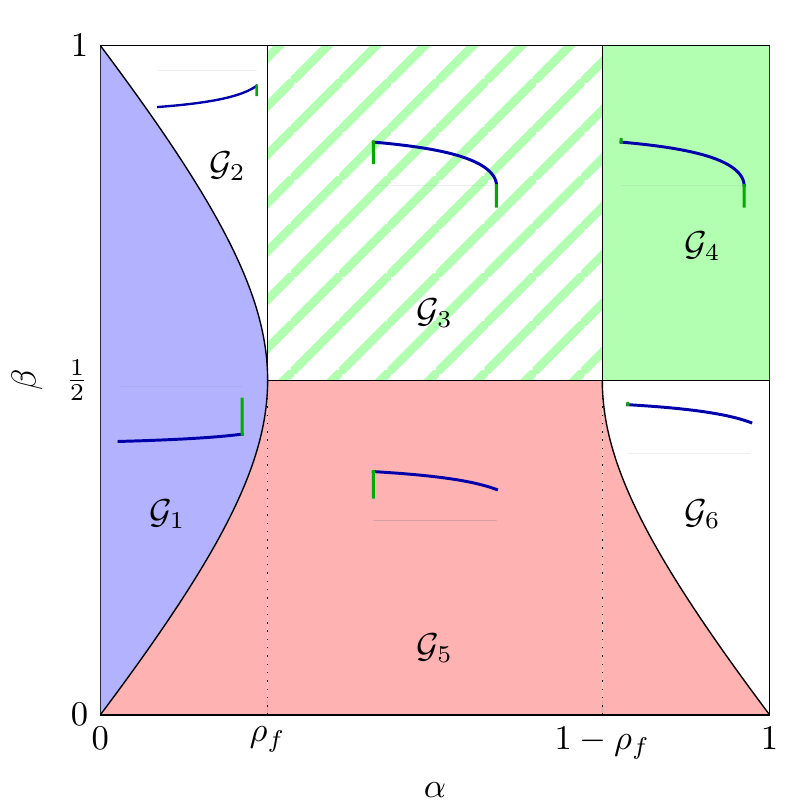}
 \caption{Schematic illustration of the regions $\mathcal{G}_i$ in $(\alpha,\beta)$-space defined in \eqref{eq:regions} with corresponding prototypical singular solutions of type $i$, $i=1,\dots,6$. The slow parts of the orbits are displayed in blue, while the fast ones (layers) in green. The gray line in the insets corresponds to $\rho=\frac12$. We remark that the layers at $\xi=0$ are particularly tiny in $\mathcal{G}_4$ and $\mathcal{G}_6$.
 }
 \label{fig:bd_sing}
\end{figure}
We will show (in Proposition \ref{prop:singsol}) that within each of those region the structure of the singular solutions is the same. Note that our construction of singular solutions works also on all the boundary curves defined in \eqref{eq:g_bound} except for $\gamma_{15}$ and $\gamma_{23}$ where singular solutions are not unique (see Remark \ref{rem:fmam}).\\
\noindent
To this aim, we introduce the following six types of singular solutions (see Figure \ref{fig:bd_sing}):
\begin{description}
\item{Type 1.} Singular solutions which start on $\mathcal{C}_0^r$ at $\xi=0$, follow the reduced flow on $\mathcal{C}_0^r$ (where $\rho$ increases), and have a layer at $\xi=1$ in which $\rho$ increases.
\item{Type 2.} Singular solutions which start on $\mathcal{C}_0^r$ at $\xi=0$, follow the reduced flow on $\mathcal{C}_0^r$ (where $\rho$ increases), and have a layer at $\xi=1$ in which $\rho$ decreases.
\item{Type 3.} Singular solutions which have a layer at $\xi=0$ in which $\rho$ increases, follow the reduced flow on $\mathcal{C}_0^a$ (where $\rho$ decreases), and have another layer at $\xi=1$ in which $\rho$ decreases.
\item{Type 4.} Singular solutions which have a layer at $\xi=0$ in which $\rho$ decreases, follow the reduced flow on $\mathcal{C}_0^a$ (where $\rho$ decreases), and have another layer at $\xi=1$ in which $\rho$ decreases.
\item{Type 5.} Singular solutions which have a layer at $\xi=0$ in which $\rho$ increases, and follow the reduced flow on $\mathcal{C}_0^a$ (where $\rho$ decreases).
\item{Type 6.} Singular solutions which have a layer at $\xi=0$ in which $\rho$ decreases and follow the reduced flow on $\mathcal{C}_0^a$ (where $\rho$ decreases).
\end{description}
 
\noindent More details about the construction and structure of these singular orbits are given in the following proof.
\begin{proposition} \label{prop:singsol}
 Let $k(x) \in C^1(\mathbb{R})$ be monotone decreasing. Then for each $(\alpha,\beta) \in \mathcal{G}_i$, $i=1, \ldots 6$  there exists a unique singular solution $\Gamma^i$ of type $i$ to \eqref{eq:reduced equation rho}-\eqref{eq:boundary conditions reduced rho} composed of segments of orbits of the layer problem \eqref{eq:laypb_k} and the reduced problem \eqref{eq:k_redpb_rholg12} satisfying the boundary conditions.
\end{proposition}
\begin{proof}
The proof is based on the shooting technique outlined above. Technically speaking, we show that the intersection between the sets $\mathcal{L}_1^+ \cup l_1$ in \eqref{eq:L+_1}-\eqref{eq:l_1} and $\mathcal{R}^- \cup r$ in \eqref{eq:p_r}-\eqref{eq:R-} is non-empty, and in particular consists of one point. This gives us the unique values of $\rho_0,\,\rho_1$ for which a singular orbit exists depending on $\alpha$ and $\beta$, which in turn allows us to identify the six types of singular solutions corresponding to the six regions defined in \eqref{eq:regions}.
While we claim the existence of singular solutions only in the open regions $\Gamma^i$, $i=1,\ldots,6$ we also comment on the singular configurations
where $(\alpha,\beta)$ lies on the curves $\gamma_{ij}$ from (\ref{eq:g_bound}).\\
In principle there are four possible ways for the intersection between $\mathcal{L}_1^+ \cup l_1$ and $\mathcal{R}^- \cup r$ to occur; three of these lead to two possible profiles, each corresponding to two regions in $(\alpha,\beta)$-parameter space, while the fourth case ($l_1 \cap r$) leads to an empty intersection, since $l_1$ and $r$ are separated from $\mathcal{L}_1^+$ and $\mathcal{R}^-$, respectively, only for $\alpha$ and $\beta$ both less than $\frac12$, and in this case they can never coincide.
Thus, we are left with:
\begin{description}
 \item[Case 1: $l_1 \cap \mathcal{R}^- \neq \emptyset$.] From the investigation of this case we obtain orbits of type 1 and 2.
 \item[Case 2: $\mathcal{L}_1^+ \cap \mathcal{R}^- \neq \emptyset$.] From the investigation of this case we obtain orbits of type 3 and 4.
 \item[Case 3: $\mathcal{L}_1^+ \cap r \neq \emptyset$.] From the investigation of this case we obtain orbits of type 5 and 6.
\end{description}
We point out that in our arguments below we use the fact that $\frac{k(0)}{k(1)}>1$, deriving from assumption \eqref{eq:g}. In the following, we examine Cases 1-3 in more detail.\\

\emph{Case 1: $l_1\,\cap\,\mathcal{R}^- \neq \emptyset$.} By definition of $l_1$, this occurs only when $\alpha \leq \rho_f$. In this case, we have $l_1 \in \mathcal{R}^-$, which implies that $p_0 = l$ and, consequently, $\rho_0=\alpha$. Moreover, since $\rho(1,s)=\rho^\ast(\alpha)$, following the flow of the layer problem until it hits $\mathcal{R}$ we obtain
\begin{equation} \label{eq:rho1_case1}
 \rho_1 = \frac{\alpha(1-\alpha)k(0)}{\beta k(1)}.
\end{equation}
By construction, we have that $p_0 \in \mathcal{C}_0^r$, so the singular orbit in this case consists in a slow motion along $\mathcal{C}_0^r$ where $\rho$ increases followed by a layer at $\xi=1$. The nature of this layer depends on $\alpha$ and $\beta$ as follows:
\begin{itemize}
 \item When $\alpha < \rho_f$ and $\rho^\ast(\alpha) < \beta < 1-\rho^\ast(\alpha)$, i.e.~for $(\alpha,\beta) \in \mathcal{G}_1$, $\rho$ increases along the boundary layer at $\xi=1$. The corresponding singular solution is therefore of type $1$ (see Figure \ref{fig:singsol_1}).
 \item When $\alpha < \rho_f$ and $\beta > 1-\rho^\ast(\alpha)$, i.e.~for $(\alpha,\beta) \in \mathcal{G}_2$, $\rho$ decreases along the boundary layer at $\xi=1$. Therefore, the corresponding singular solution is of type $2$ (see Figure \ref{fig:singsol_2}, \ref{app:sspl}). 
\end{itemize}
We note that when $\alpha < \rho_f$ and $\beta=1-\rho^\ast(\alpha)$ (i.e.~on $\gamma_{12}$) there is no layer at $\xi=1$.\\
 \begin{figure}[!ht] 
 \centering
 \begin{minipage}{.3\textwidth}
  \centering
  \def\svgwidth{1\textwidth}
  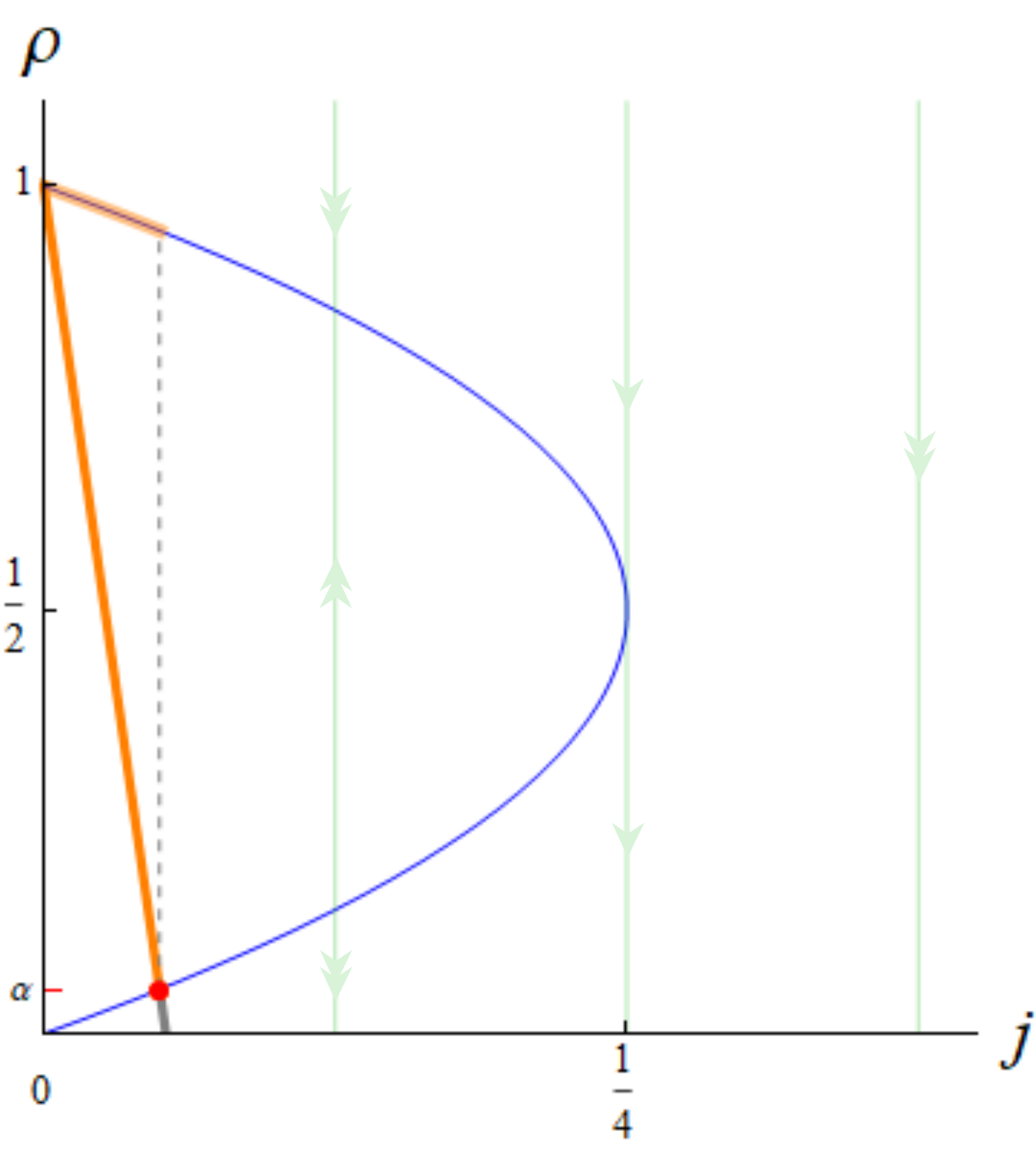\\
  (a) $\xi=0$
 \end{minipage}
 \hspace{.5cm}
 \begin{minipage}{.3\textwidth}
  \centering
  \def\svgwidth{1\textwidth}
  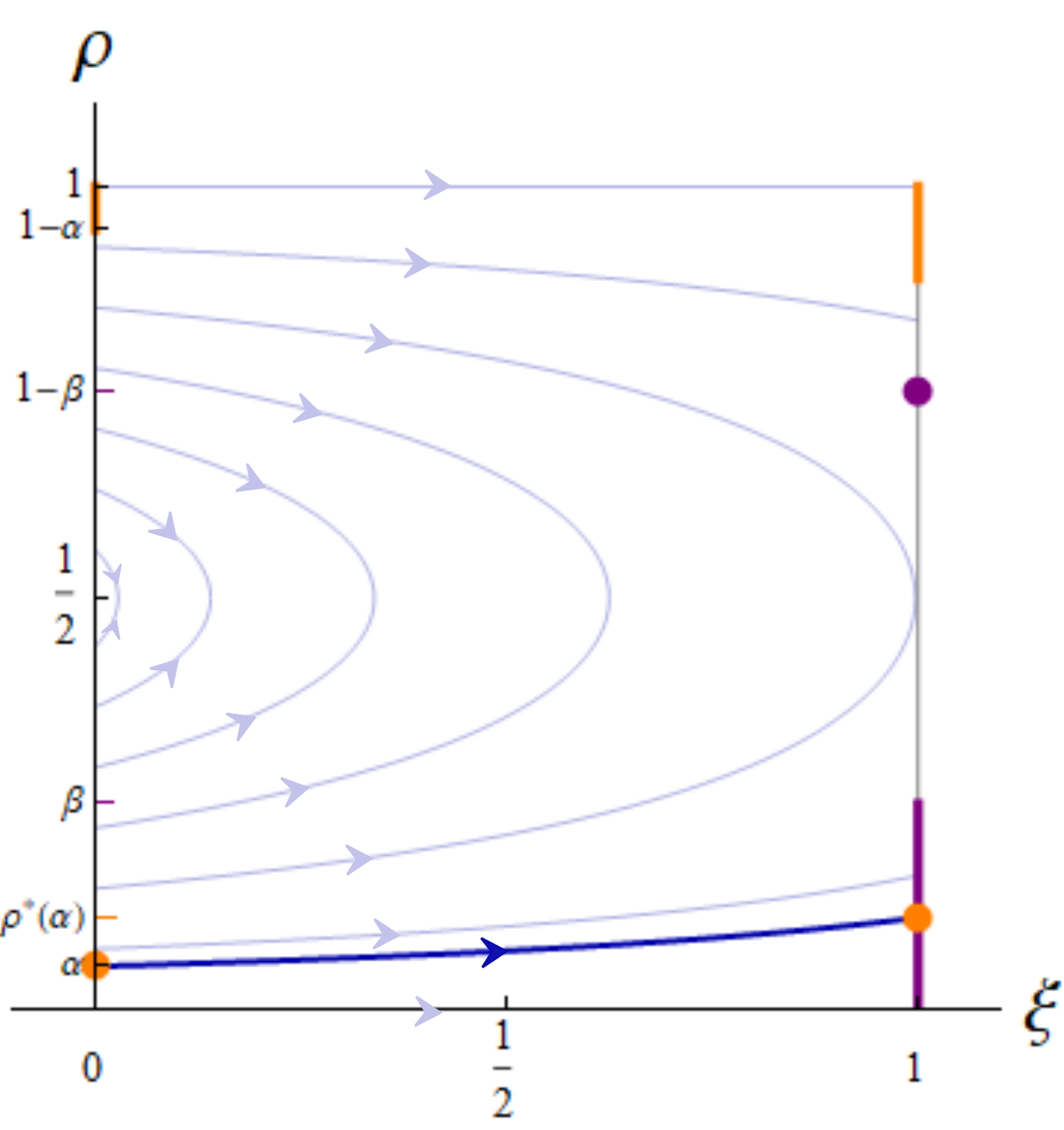\\
  (b) $\xi \in [0,1]$
 \end{minipage}
 \hspace{.5cm}
 \begin{minipage}{.3\textwidth}
  \centering
  \def\svgwidth{1\textwidth}
  \input{1_right_3_a.pdf_tex}\\
  (c) $\xi=1$
 \end{minipage}
  \begin{minipage}{.3\textwidth}
  \centering
  \def\svgwidth{1\textwidth}
  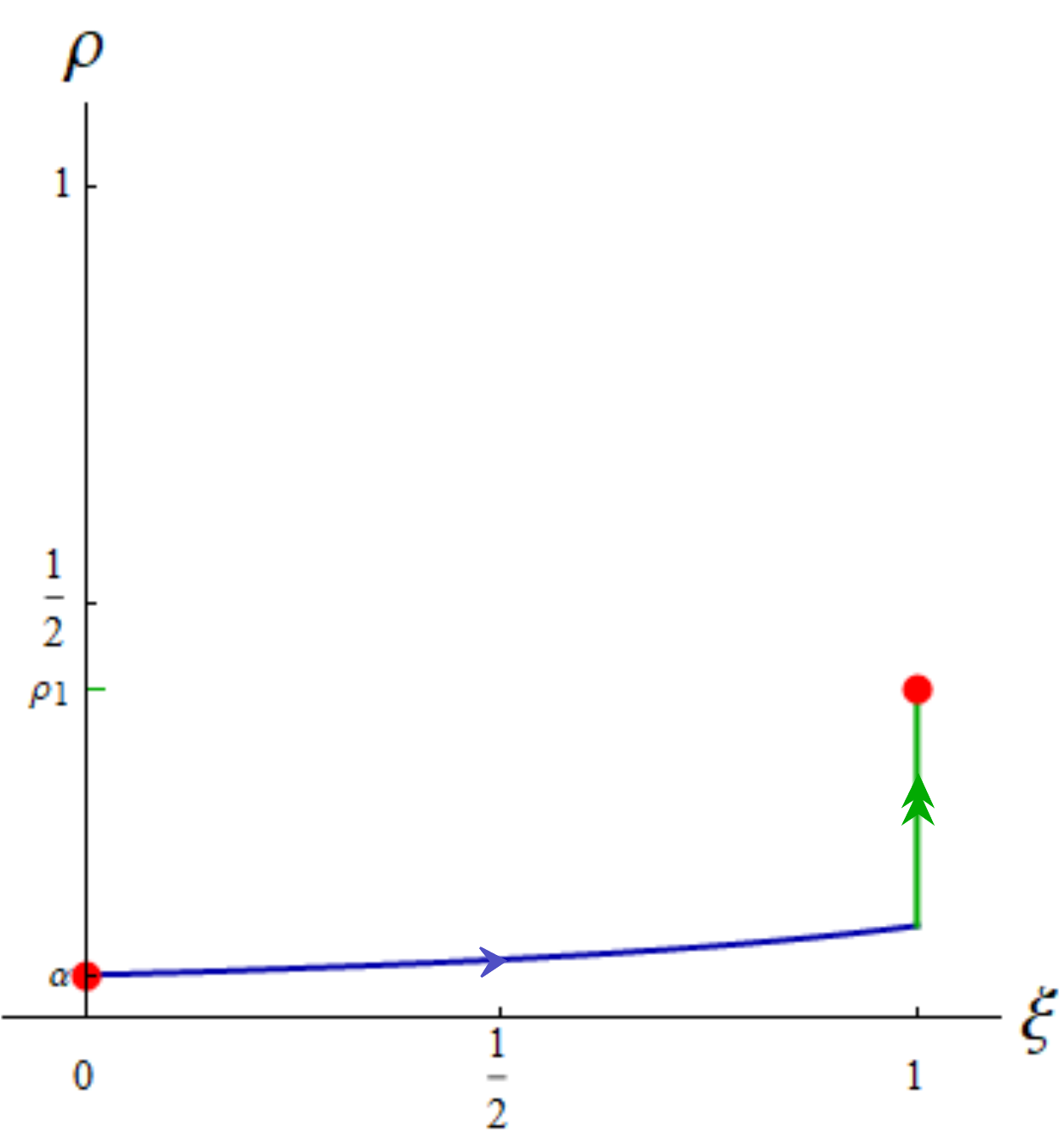\\
  (d) 
 \end{minipage}
  \hspace{.5cm}
 \begin{minipage}{.3\textwidth}
  \centering
  \def\svgwidth{1\textwidth}
  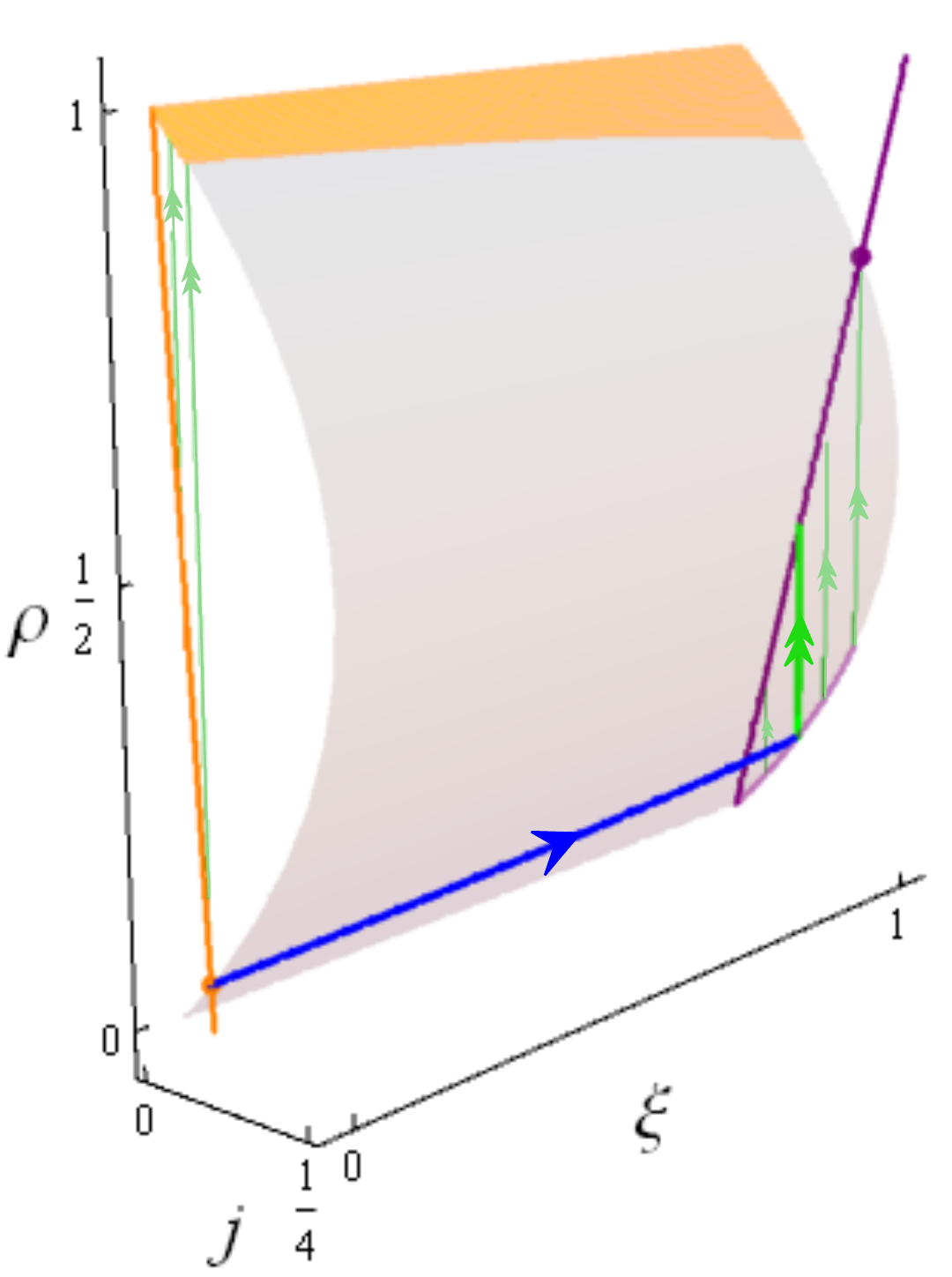\\
  (e)
 \end{minipage}
 \caption{Schematic representation of a singular solution of type 1. (a) Boundary conditions at $\xi=0$ in $(j,\rho)$-space: the orange line is $\mathcal{L}$, while the orange curve is $\mathcal{L}^+$. The red dot represents $p_0$, where in this case $p_0=l$. (b) Slow evolution on $\mathcal{C}_0$ from $(0,\alpha)$ to $(1,\rho^\ast(\alpha))$ (blue curve). The orange lines are the projection of $\mathcal{L}$ (at $\xi=0$) and $\mathcal{L}^+$ (at $\xi=1$) on $\mathcal{C}_0$, while the purple one represents the projection of $\mathcal{R}^-$ on $\mathcal{C}_0$ at $\xi=1$. The orange dots correspond to $l$ (at $\xi=0$) and $l_1$ (at $\xi=1$), while the purple dot corresponds to $r$. (c) Here, we consider $\xi=1$ in $(j,\rho)$-space. The red dot corresponds to $p_1$, while the purple line and curve represent the manifolds $\mathcal{R}$ and $\mathcal{R}^-$, respectively. The green line corresponds to the layer of the singular orbit where $\rho$ increases from $\rho^\ast(\alpha)$ to $\rho_1=\frac{\alpha(1-\alpha)k(0)}{\beta k(1)}$. (d) Singular solution of type 1 in $(\xi,\rho)$-space. (e) Singular solution of type 1 in $(j,\xi,\rho)$-space.}
 \label{fig:singsol_1}
 \end{figure}

\emph{Case 2: $\mathcal{L}_1^+ \cap \mathcal{R}^- \neq \emptyset$.} We observe that by definition $\mathcal{L}_1^+ \subset \mathcal{C}_0^a$ and $\mathcal{R}^- \subset \mathcal{C}_0^r$. Therefore, these two sets have a non-empty intersection if and only if $\alpha \geq \rho_f$ and $\beta \geq \frac12$, consisting in the point
 \begin{equation} \label{eq:foldp}
     q = \left( \frac14,1,\frac12 \right) \in S.
 \end{equation}
The fact that $\rho(1,s)=\frac12$, i.e.~$\rho(0,s)=1-\rho_f$, allows us to identify the starting and ending points of the orbit by following the flow of the layer problem (backwards at $\xi=0$ and forward at $\xi=1$). This leads to two layers at both endpoints, and
\begin{equation} \label{eq:rho01_case2}
 \rho_0 = 1-\frac{k(1)}{4 \alpha k(0)}, \qquad \rho_1 = \frac{1}{4\beta}.
\end{equation}
In particular, we have:
\begin{itemize}
 \item When $\rho_f < \alpha < 1-\rho_f$ and $\beta > \frac12$, i.e.~for $(\alpha,\beta) \in \mathcal{G}_3$, $\rho$ increases along the boundary layer at $\xi=0$. This implies that the singular orbit is of type $3$ (see Figure \ref{fig:singsol_3}, \ref{app:sspl}).
 \item When $\alpha > 1-\rho_f$ and $\beta > \frac12$, i.e.~in $\mathcal{G}_4$, $\rho$ decreases along the boundary layer at $\xi=0$, and the singular solution is of type $4$ (see Figure \ref{fig:singsol_4}, \ref{app:sspl}).
\end{itemize}
We note that when $\alpha=1-\rho_f$ and $\beta > \frac12$ (i.e.~on $\gamma_{34}$) we have no boundary layer at $\xi=0$.\\

\emph{Case 3: $\mathcal{L}_1^+ \cap r \neq \emptyset$.} By definition of $r$, this occurs only when $\beta < \frac12$. In this case, we have $r \in \mathcal{L}_1^+$, which implies that $p_1 = r$ and, consequently, $\rho_1=1-\beta$. Moreover, since $\rho(0,s)=\rho_\ast(\beta)$, following the layer problem backwards until it hits $\mathcal{L}$, we obtain
\begin{equation} \label{eq:rho0_case3}
 \rho_0=1-\frac{\beta(1-\beta)k(1)}{\alpha k(0)}.
\end{equation}
Consequently, there is a boundary layer at $\xi=0$. Since $\mathcal{L}_1^+ \subset \mathcal{C}_0^a$, $\rho$ decreases along the reduced flow. There are two possible kinds of layers at $\xi=0$, i.e.
\begin{itemize}
 \item If $1-\rho_\ast(\beta) < \alpha < \rho_\ast(\beta)$ and $\beta < \frac12$, i.e.~if $(\alpha,\beta) \in \mathcal{G}_5$, $\rho$ is increasing and the singular solution is of type $5$ (see Figure \ref{fig:singsol_5}, \ref{app:sspl}).
 \item If $\alpha > \rho_\ast(\beta)$ and $\beta < \frac12$, i.e.~if $(\alpha,\beta) \in \mathcal{G}_6$, $\rho$ is decreasing, and we have a singular solution of type $6$ (see Figure \ref{fig:singsol_6}, \ref{app:sspl}).
\end{itemize}
We note that when $\alpha=\rho_\ast(\beta)$ and $\beta < \frac12$ (i.e.~on $\gamma_{56}$), there are no boundary layers. We also remark that when $\alpha > \rho_f$ and $\beta = \frac12$ (i.e.~on $\gamma_{35}$ and $\gamma_{46}$), we have a unique solution with no boundary layer at $\xi=1$ satisfying
 \begin{equation}
  \rho(0,s)=1-\rho_f, \quad \rho(1,s)=\frac12.
 \end{equation}
\end{proof}

\begin{remark}
 The construction in Case 3 is essentially the same as the one in Case 1 upon reversal of the flow direction in \eqref{eq:fosk}.
\end{remark}

\begin{remark}[Excluded cases] \label{rem:fmam}
When $\alpha \leq \rho_f$ and $\beta=\rho^\ast(\alpha)$ - i.e.~when $(\alpha,\beta) \in \gamma_{15}$ - we have that both $\mathcal{L}_1^+ \cap r$ and $l_1 \cap \mathcal{R}^-$ are non-empty. Consequently, there are two possible reduced solutions, satisfying (see Figure \ref{fig:sp_cases}(a))
  \begin{equation}
  \text{(a)} \, \begin{cases}
              \rho(0,s)=\rho_0=\alpha,\\
              \rho(1,s)=\beta,
             \end{cases}
  \text{ or } \quad \text{ (b) } \begin{cases}
              \rho(0,s)=1-\alpha,\\
              \rho(1,s)=\rho_1=1-\beta.
             \end{cases}
 \end{equation}
 In this case, we have a continuum of singular solutions, since at any $\xi \in [0,1]$ it is possible to jump from the slow trajectory of the reduced flow in (a) to the one in (b) by means of the flow of the layer problem. Analogously, we obtain a continuum of singular solutions when $\alpha=\rho_f$, $\beta \geq \frac12$, i.e.~when $(\alpha,\beta) \in \gamma_{23}$. In this case, in fact, we have that both $\mathcal{L}_1^+ \cap \mathcal{R}^-$ and $l_1 \cap \mathcal{R}^-$ are non-empty, and therefore there are two possible reduced solutions (with jumps possible at any $\xi \in [0,1]$ via the flow of the layer problem) satisfying (see Figure \ref{fig:sp_cases}(b))
 \begin{equation}
  \text{(c)} \, \begin{cases}
              \rho(0,s)=\rho_0=\rho_f,\\
              \rho(1,s)=\frac12,
             \end{cases}
  \text{ or } \quad \text{ (d) } \begin{cases}
              \rho(0,s)=1-\rho_f,\\
              \rho(1,s)=\frac12.
             \end{cases}
 \end{equation}
 Hence, our strategy to construct singular solutions to Equation \eqref{eq:reduced equation rho}-\eqref{eq:boundary conditions reduced rho} as in Proposition \ref{prop:singsol} does not give uniqueness, and therefore the methods based on transversality arguments to infer persistence of solutions for $0 < \varepsilon \ll 1$ do not apply. We leave the analysis of this more delicate situation for future work.
\end{remark}

\begin{figure}[!ht]
      \centering
     \begin{minipage}{0.45\textwidth}
          \centering
          \def\svgwidth{1\textwidth}
          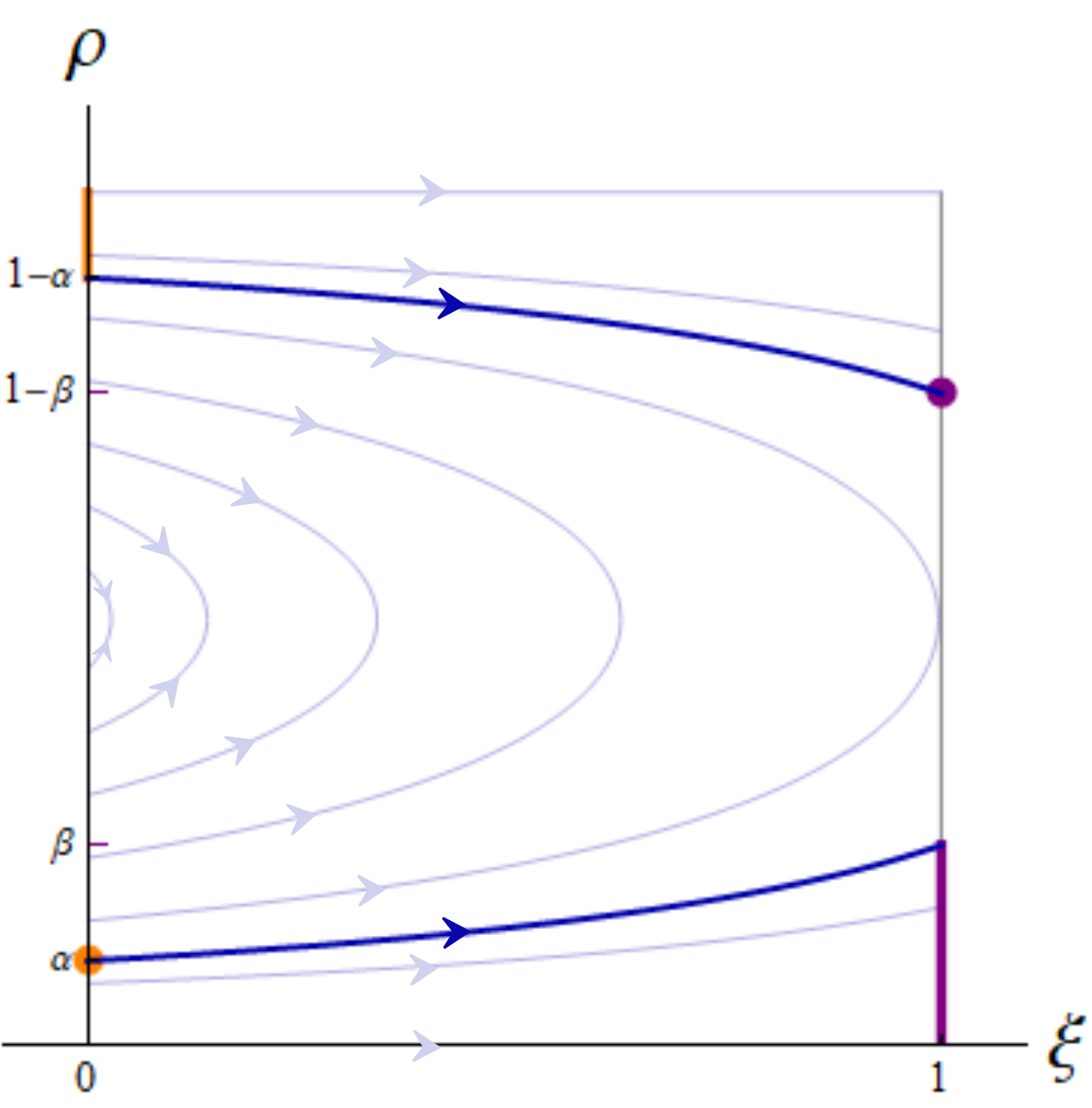\\
          (a)
     \end{minipage}
     \hfill
     \begin{minipage}{0.45\textwidth}
          \centering
          \def\svgwidth{1\textwidth}
          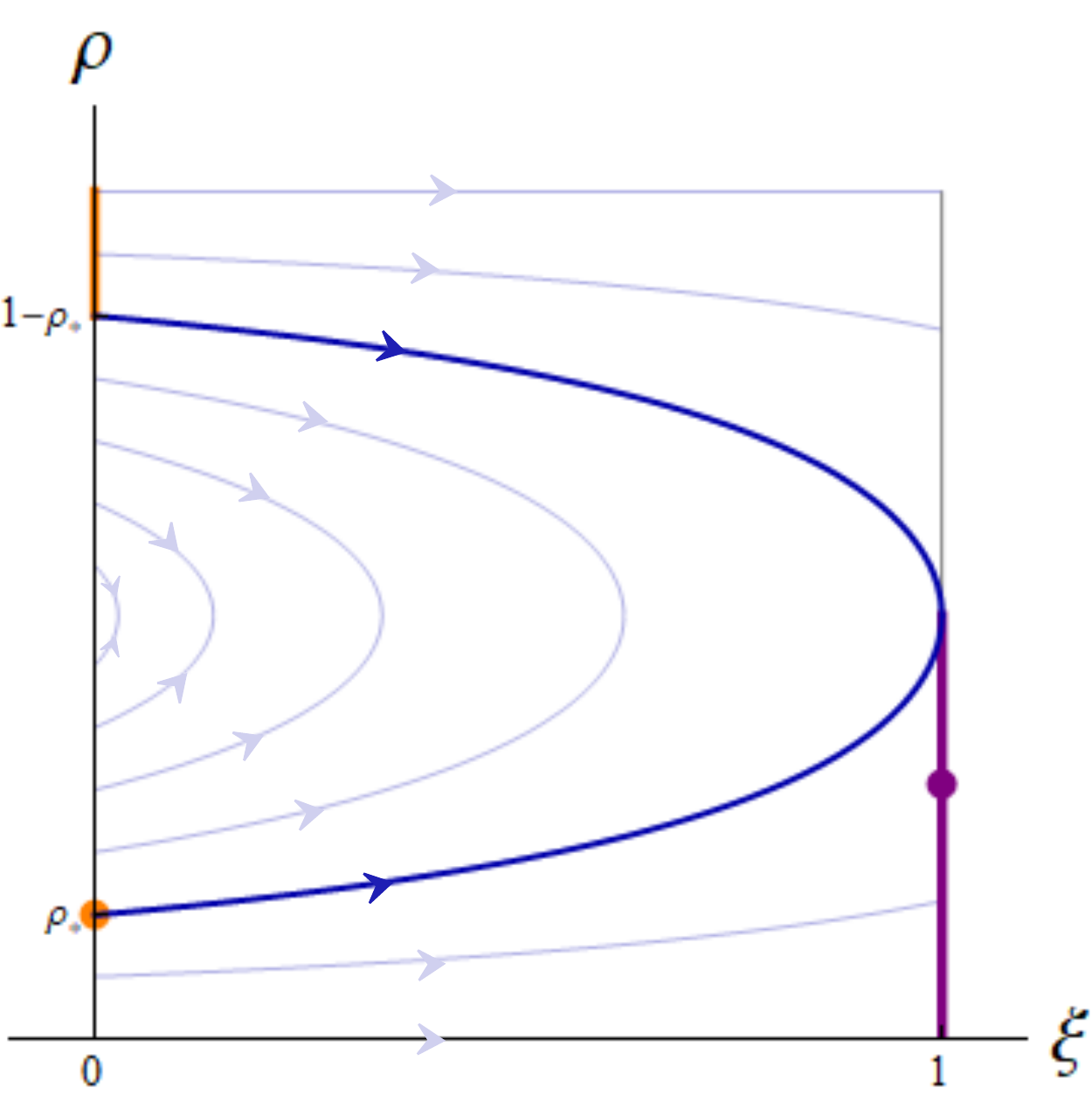\\
          (b)
     \end{minipage}
     \caption{Schematic representation in $(\xi,\rho)$-space of the slow portions (blue curves) of the possible singular orbits for (a) $(\alpha,\beta) \in \gamma_{15}$ and (b) $(\alpha,\beta) \in \gamma_{23}$. The orange and purple curves correspond to the projection of $\mathcal{L}^+$ and $\mathcal{R}$, respectively, on the $(\xi,\rho)$-space. Fast jumps from the slow solution in $\mathcal{C}_0^r$ to the slow solution in $\mathcal{C}_0^a$ are possible at each $\xi \in [0,1]$.}
     \label{fig:sp_cases}
\end{figure}

We now prove that the singular solutions from Proposition \ref{prop:singsol} perturb to solutions of \eqref{eq:reduced equation rho}-\eqref{eq:boundary conditions reduced rho} for $\varepsilon$ sufficiently small.

\begin{theorem} \label{thm:k}
 For each $(\alpha,\beta) \in \mathcal{G}_i$, $i=1,\dots,6$, the boundary value problem \eqref{eq:reduced equation rho}-\eqref{eq:boundary conditions reduced rho} has a unique solution $\rho(x,\alpha,\beta,\varepsilon)$ for $\varepsilon$ sufficiently small. In the phase-space formulation \eqref{eq:fosk}, this solution corresponds to an orbit $\Gamma^i_\varepsilon$ which is $\mathcal{O}(\varepsilon^\mu)$-close to $\Gamma^i$ in terms of Hausdorff distance, where $\mu=1$ for $i=1,2,5,6$ and $\mu=2/3$ for $i=3,4$.
\end{theorem}
\begin{proof}
The solutions for $\varepsilon$ small are obtained by perturbing from the singular solutions $\Gamma^i$, $i=1,\dots,6$. More precisely, we show that the manifold obtained by flowing the line $\mathcal{L}$ of points corresponding to the boundary conditions at $\xi=0$ to $\xi=1$ for $\varepsilon$ small intersects the line $\mathcal{R}$ of points corresponding to the boundary conditions at $\xi=1$ in a point which is close to the corresponding point of the singular solution. Analogously to Proposition \ref{prop:singsol}, this is done by considering three cases.\\

\emph{Case 1: $(\alpha,\beta) \in \mathcal{G}_i$, $i=1,2$}. In this case, the singular solution starts at the point $p_0=l\in \mathcal{L}$ (see Proposition \ref{prop:singsol}). We recall that the singular solution consists of a slow segment obtained by flowing $l$ to $l_1$ with the reduced flow on the repelling part $\mathcal{C}_0^r$ of the critical manifold, followed by a layer from $l_1$ to $p_1 \in \mathcal{R}$ (see Figure \ref{fig:singsol_1}-\ref{fig:singsol_2}).\\
We show below that for $0 < \varepsilon \ll 1$ the flow defined by \eqref{eq:fosk} takes a suitable small segment of $\mathcal{L}$ to a curve $\mathcal{L}_{1,\varepsilon}$ in the plane $\xi=1$. Furthermore, we show that $\mathcal{L}_{1,\varepsilon}$ intersects $\mathcal{R}$ in a point $p_{1,\varepsilon}$ which corresponds to the right end-point of the solution of the boundary value problem. Note that the full solution for $\xi \in [0,1]$ is obtained by following the flow backward from $p_{1,\varepsilon}$ to $\xi=0$. In more detail, Fenichel theory \cite{Fenichel_1979} implies that (compact subsets of) $\mathcal{C}_0^r$ perturbs smoothly to a repelling slow manifold $\mathcal{C}_\varepsilon^r$ with an unstable foliation $\mathcal{F}^u_\varepsilon$.
The line $\mathcal{L}$ intersects $\mathcal{C}_\varepsilon^r$ in a point $l_\varepsilon$ which is $C^1$ $\mathcal{O}(\varepsilon)$-close to $l$. The image of $l_\varepsilon$ by the slow flow on $\mathcal{C}_\varepsilon^r$ at $\xi=1$ is denoted by $l_{1,\varepsilon}$. The unstable fiber $\mathcal{F}_\varepsilon^u(l_{1,\varepsilon})$
is $C^1$ $\mathcal{O}(\varepsilon)$-close to the corresponding fiber $\mathcal{F}_0^u(l_1)$ of the layer problem. Since $\mathcal{L}$ intersects $\mathcal{C}_0^r$ transversely in $l$, most orbits starting in $\mathcal{L}$ are repelled away immediately, and only points on $\mathcal{L}$ exponentially close to $l_\varepsilon$ stay in a neighbourhood of the singular orbit $\Gamma^i$, $i=1,2$, up to $\xi=1$, where they are denoted by $\mathcal{L}_{1,\varepsilon}$. The exponential expansion along the unstable fibers implies that $\mathcal{L}_{1,\varepsilon}$ is $C^1$ $\mathcal{O}(\varepsilon)$-close to $\mathcal{F}_\varepsilon^u(l_{1,\varepsilon})$. Hence, $\mathcal{L}_{1,\varepsilon}$ intersects $\mathcal{R}$ in a point $p_{1,\varepsilon}$ as claimed.   \\

\emph{Case 2: $(\alpha,\beta) \in \mathcal{G}_i$, $i=3,4$}. In this case, the singular solution starts with a layer connecting the point $p_0 \in \mathcal{L}$ to a point on $\mathcal{C}_0^a$, followed by a segment of the reduced flow ending at the fold point $q \in S$, followed by a layer ending at $p_1 \in \mathcal{R}$. As before, we follow the line $\mathcal{L}$ of boundary conditions at $\xi=0$ forward by the flow \eqref{eq:fosk} and show that it intersects the line $\mathcal{R}$ of boundary conditions at $\xi=1$. Since the singular solution involves the point $q$ on the non-hyperbolic fold line $S$, results on extending GSPT to such problems \cite{Szmolyan_2004} are needed here. Consider a small segment of $\mathcal{L}$ containing $p_0$ and denote its extension by the forward flow of \eqref{eq:fosk} by $\mathcal{M}_\varepsilon$ for $\varepsilon$ small. Since $\mathcal{L}$ is a line and $\mathcal{M}_\varepsilon$ is defined by the flow, it is a smooth, two-dimensional manifold. Fenichel theory and the results in \cite{Szmolyan_2004} on the dynamics close to fold-curves imply that parts of $\mathcal{M}_\varepsilon$ are close to $p_1$ and are $\mathcal{O}(\varepsilon^{2/3})$-close to the plane $\{j=\frac14 \}$ in a neighbourhood of $p_1$ (see Figure \ref{fig:thm_case2}). Hence, $\mathcal{R}$ intersects $\mathcal{M}_\varepsilon$ in a point $p_{1,\varepsilon}$ (close to $p_1$) which corresponds to the right end-point of the boundary value problem. Again, the full solution is obtained by following the flow backward to $\xi=0$. \\

\begin{figure}[!ht]
	\centering
	\begin{overpic}[scale=.7]{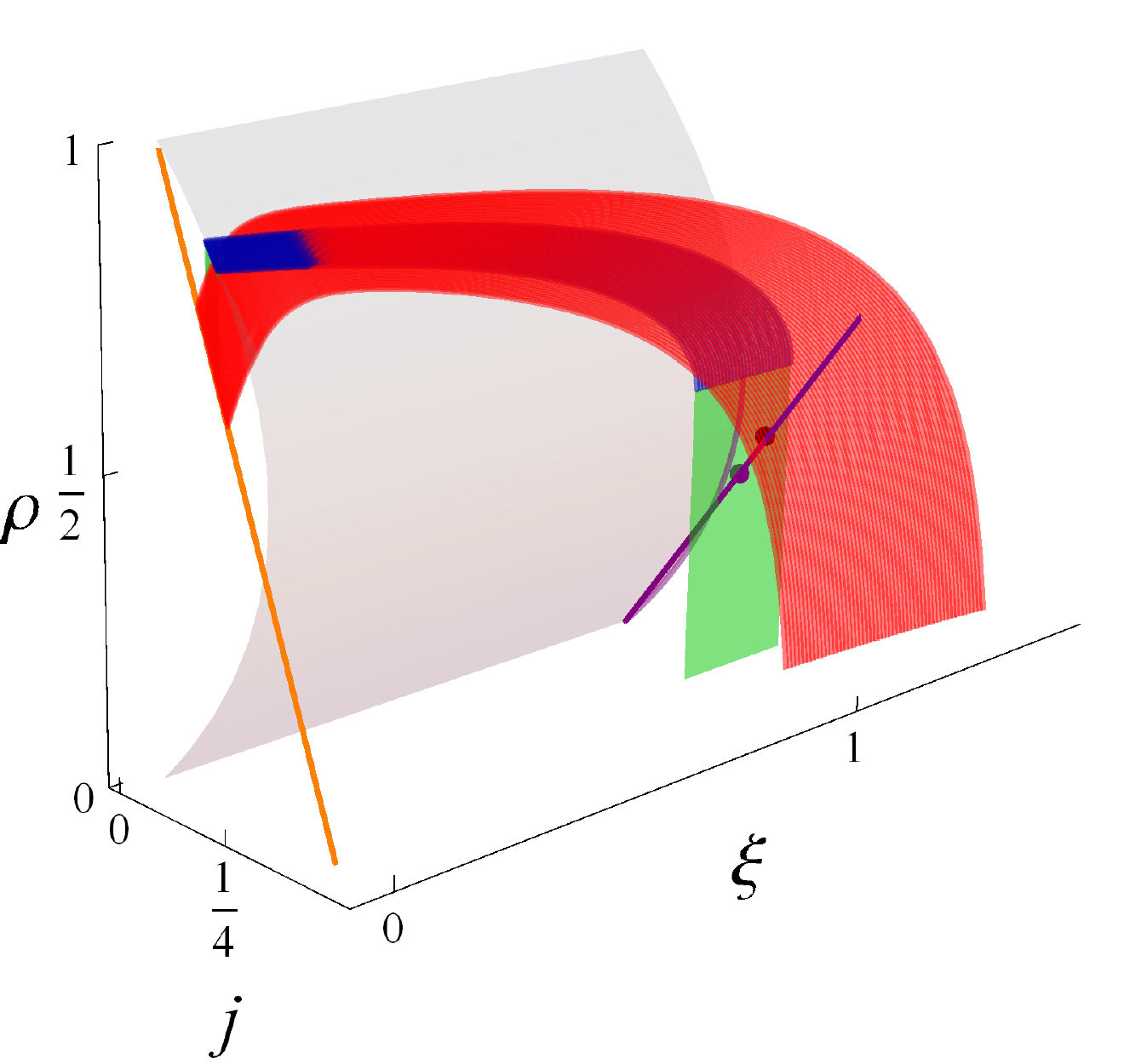}
	\put(71,55){\small{$p_{1,\varepsilon}$}}
	\put(66,49){\small{$p_1$}}
    \put(90,40){$\mathcal{M}_\varepsilon$}
	\end{overpic}
\caption{Schematic representation of the manifold $\mathcal{M}_\varepsilon$ (red surface). This manifold limits on the union of the orbits which solve the layer \eqref{eq:laypb_k} and the reduced \eqref{eq:k_redpb_xrho} problem (green and blue surfaces, respectively) as $\varepsilon \to 0$ starting from a small segment on $\mathcal{L}$ (orange line) at $\xi=0$. The purple dot corresponds to $p_1$, while the black dot represents the point $p_{1,\varepsilon}$ where $\mathcal{M}_\varepsilon$ and $\mathcal{R}$ (purple line) intersect. The full orbit for the boundary value problem is then obtained starting from $p_{1,\varepsilon}$ and tracking it backward up to $\xi=0$.}
\label{fig:thm_case2}
\end{figure}

\emph{Case 3: $(\alpha,\beta) \in \mathcal{G}_i$, $i=5,6$}. This case is completely analogous to Case 1 upon reversal of the flow direction.
\end{proof}

\begin{remark}
 Note that in layers away from the fold line $S$ solutions have exponential decay/growth rates, while in layers which involve the fold points on $S$ solutions have slower, algebraic decay/growth rates. Moreover, the behaviour near fold points is also responsible for the larger discrepancy between singular solutions and true orbits of \eqref{eq:reduced equation rho}-\eqref{eq:boundary conditions reduced rho} in regions $\mathcal{G}_i$, $i=3,4$. These effects are also observed in the computational results presented in Section \ref{sec:numerics}.
\end{remark}

\begin{remark}[Straight corridors.]
 For functions $k$ such that $k(x)\equiv k_0$, with $k_0>0$ - describing a straight corridor - Equation \ref{eq:reduced equation rho} implies that $j$ must be constant along solutions of the boundary value problem. The construction described in Section \ref{sec:vanvis} therefore simplifies, since Equation \eqref{eq:fosk} and \eqref{eq:fosk_fast} become
\begin{equation} \label{eq:fosk_sc}
 \begin{aligned}
  \varepsilon \dt{\rho} &= \rho(1-\rho)-j, \\
  \dt{\xi} &= 1,
 \end{aligned}
\end{equation}
and
\begin{equation} \label{eq:fosk_fast_sc}
 \begin{aligned}
  \rho' &= \rho(1-\rho)-j, \\
  \xi' &= \varepsilon,
 \end{aligned}
\end{equation}
respectively. The reduced problem \eqref{eq:k_redpb} is given by
\begin{equation} \label{eq:redk1k2_sc}
\begin{aligned}
  0 &= \rho(1-\rho)-j, \\
  \dt{\xi} &= 1,
\end{aligned}
\end{equation}
whose orbits consist in straight lines with constant $\rho$ on $\mathcal{C}_0$. In this situation,
we have that $\rho_f = \frac12$. Therefore the region $\mathcal{G}_3$ disappears, leaving only $5$ regions in total. The singular solutions for $(\alpha,\beta) \in \mathcal{G}_1$, $\mathcal{G}_2$, $\mathcal{G}_5$, and $\mathcal{G}_6$ are constructed as in Proposition \ref{prop:singsol}. In $\mathcal{G}_4$, however, the slow transition occurs entirely on $S$, where the only variable which varies is $\xi$ while $\rho$ and $j$ remain constant.

For $\varepsilon > 0$ and $k \equiv 1$, we give explicit expressions for $\rho$ in the \ref{sec:str_chann}. We also show that $J_\mathrm{max} := \max_{\alpha,\beta} \rJ = \frac{1}{4} + o(\varepsilon)$.
Therefore, in the corresponding singular limit $\varepsilon = 0$ we have $\mathrm{J}_\mathrm{max} = \frac14$. This agrees with the fact that the maximum (constant) value achieved by $j$ in region $\mathcal{G}_{6}$ is $j=\frac14$, and $\mathrm{J}_{\mathrm{max}}=j_\mathrm{max}$.
\end{remark}

\begin{remark}
    We briefly comment on more realistic and challenging geometries, e.g.~cases where $k$ is discontinuous and/or non monotonic, which will also be considered in Section \ref{sec:numerics}. \\
    An interesting first class of such problems consists of corridors with piecewise constant functions $k(x)$: in this case, the orbits necessarily cross these discontinuities. 
    Since geometric singular perturbation theory requires at least $C^1$ smoothness,
    standard arguments for the persistence of singular solutions are not directly applicable and additional matching arguments are needed at the discontinuities of $k(x)$. \\
    Another interesting case are corridors corresponding to smooth functions $k(x)$ which have a nondegenerate minimum in the interval $(0,1)$, such as e.g.
\begin{equation} \label{eq:k_smooth}
 k(x)=1+c_1\,\mathrm{cos}(2 \pi x),
\end{equation}
with $0 < c_1 < 1$. Further below, we also refer to those domains with non monotonic $k(x)$ as ``bottlenecks''.
In this case, the analysis can in principle be carried out in a similar way as for monotone functions $k(x)$. However, new phenomena arise at points on the fold line $\rho = \frac 12$ corresponding to the minimum of $k(x)$ where the derivative of $k$ vanishes.  The corresponding points on the fold line are called 
folded saddles. 
At these special singularities solutions of the reduced problem my cross the fold line from $\mathcal{C}_0^r$ to
$\mathcal{C}_0^a$ and vice versa, e.g.~see the folded saddle $(j,\xi,\rho)=\left(\frac14,\frac12,\frac12  \right)$ (see Figure \ref{fig:bottlenecks}). These special solutions are known to persist for $\varepsilon \ll 1$ as 
so called canard solutions \cite{Szmolyan_2001, Kuehn_2015}. The analysis of the impact of canard solutions  
on the existence and types of solutions of the boundary value problem \eqref{eq:reduced equation rho}-\eqref{eq:boundary conditions reduced rho} is left for future work.
\end{remark}

\begin{figure}
    \centering
    \includegraphics[scale=0.5]{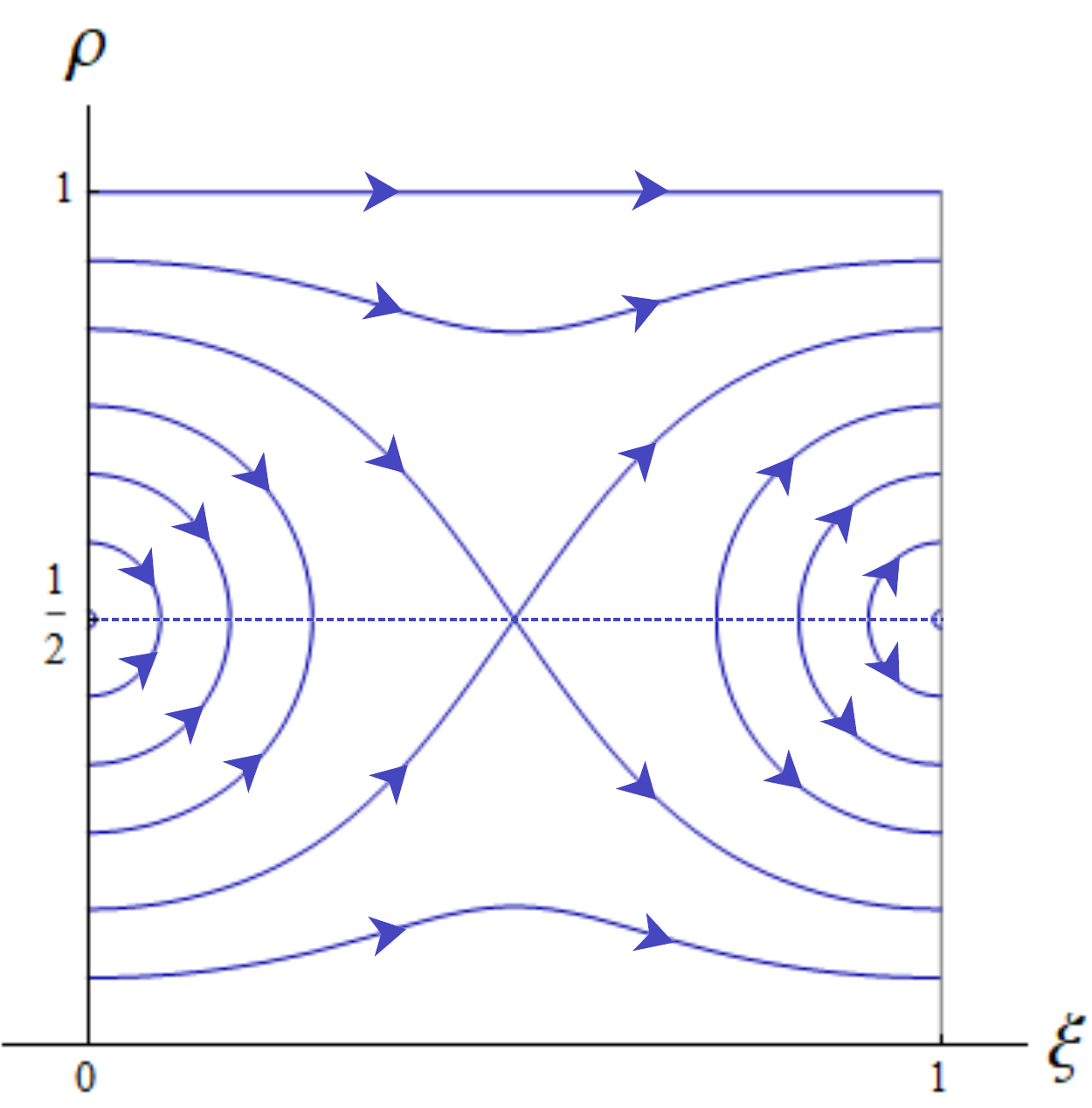}
    \caption{Reduced flow associated to Equations \eqref{eq:fosk}-\eqref{eq:bc_sf} with $k(x)=1+c_1\,\mathrm{cos}(2 \pi x)$ and $c_1=0.3$. The dashed line indicates the line of fold points located at $\rho=\frac12$. The point $(\frac 12, \frac 12)$ is a folded saddle.} 
    \label{fig:bottlenecks}
\end{figure}

\section{Numerical experiments}\label{sec:numerics}
We conclude by complementing our the analytical results by means of computational experiments. System \eqref{eq:reduced equation rho}-\eqref{eq:boundary conditions reduced rho} is quadratic in $\rho$, which makes Newton's method an ideal numerical approach. The nonlinear system is discretised using $P_1$  elements (polynomials of degree 1). The corresponding discrete nonlinear system is then solved using Newton's method. The solver is implemented using the finite element library \texttt{FEniCS}\cite{LoggMardalEtAl2012a,LoggWells2010a}. This proves sufficient for our purposes and we do not use discontinous Galerkin methods as in e.g.~\cite{burger_flow_2016}.

Newton's method requires a sufficiently good initial guess - which is not necessarily available for general functions $k$.  A suitable choice is important for small diffusitivies, where we expect the formation of sharp boundary layers.
By integration of \eqref{eq:reduced equation rho}, the following bound holds:
\begin{equation*}
    \| \partial_x \rho\|_2 \le c(\varepsilon) := \frac{\frac{\sqrt{L}\max k}{4} + \sqrt{\frac{L(\max k)^2}{16} + 4\alpha\varepsilon \min k \max k}}{2 \varepsilon \min k} \stackrel{\varepsilon \rightarrow \infty}{\longrightarrow} 0\,,
\end{equation*}
and we see that for sufficiently large values of $\varepsilon$ a constant initial guess is suitable.
We therefore propose and implement the following iterative procedure:
\begin{enumerate}
    \item Pick a sequence $\varepsilon_1 > ... > \varepsilon_n = \varepsilon$ with $\varepsilon_1 \gtrsim c^{-1}(1)$.
\item Set the initial guess $\rho_0 \equiv \frac{1}{2}$ and $i = 0$.
    \item Solve for $\rho_i$ as the numerical solution of \eqref{eq:reduced equation rho}-\eqref{eq:boundary conditions reduced rho} with $\varepsilon_i$ using Newton's method with initial guess $\rho_{i-1}$.
    \item Set $i = i+1$ and go to step 3 until $i=n$.
\end{enumerate}

\subsection{Comparison of the 2D model and the 1D area averaged model}

We start by comparing the solutions of the 1D model to the \emph{averaged} solutions of the 2D model.
In doing so we consider four different types of geometries, all defined on the interval $[0,1]$:
\begin{enumerate}[(a)]
    \item a straight corridor, where $k \equiv 1$,
    \item a closing corridor, where $k(x) = 2-x$,
    \item a piecewise constant narrowing corridor, where $k(x) = 2 - \mathbbm{1}_{[\frac{1}{2},1]}(x)$
    \item a bottleneck, where $k(x) = 2 - \mathbbm{1}_{[\frac{1}{3}, \frac{2}{3}]}(x)$.
\end{enumerate}
We emphasise the fact that the choices (c) and (d) of $k$ correspond to a width-to-length ratio of $2:1$, which does not agree with the assumptions of the proposed 1D approximation because of their lack of smoothness.

The choice of the vector  $\textbf{u}$ can have a strong impact on the stationary profiles, and thus on the quality of the 1D approximation. To illustrate this, we discuss two different possibilities for $\mathbf{u}$ in the following. In the first instance, $\mathbf{u} = \mathbf{u}_e$ is given by the solution of the eikonal equation, see \cite{CPT2014}. In particular $ \textbf{u}_e = \nabla \phi$ where $\phi$ solves
\begin{align*}
    |\nabla \phi| &= 1 \qquad \text{ for } x \in \Omega\,,
    \\
    \phi &= 0 \qquad \text{ for } x \in \Sigma\,.
\end{align*}
The value of the potential $\phi = \phi(x,y)$ then corresponds to the distance from $(x,y)$ to the exit, with the flow of $\mathbf{u}_e$ corresponding to the shortest path. This choice is reasonable for single individuals but the corresponding streamlines tend to collide at concave corners. This can lead to locally higher densities.
In the second instance, $\mathbf{u} = \mathbf{u}_L$, with $\mathbf{u}_L = \nabla \psi/|\nabla \psi|$, where $\psi$ solves
\begin{align*}
    -\Delta \psi &= 0 \, \, \, \, \qquad \text{ for } x\in \Omega\,,
    \\
    \nabla \psi \cdot n &=
    \begin{cases} -1 &\text{ for } x\in\Gamma\,,
    \\
    1 &\text{ for } x\in\Sigma\,,
    \\
    0 &\text{ for } x\in\partial\Omega\setminus{\Gamma\cup\Sigma}\,.
    \end{cases}
\end{align*}
This choice was used in \cite{burger_flow_2016} and also by Piccoli and Tosin \cite{PT2009}, albeit with Dirichlet boundary conditions. The flow does not correspond to geodesics in this case, and streamlines make a smooth arc around corners.

The results are presented in Table~\ref{tab:profile comparison}, where we show the solution to the one-dimensional model \eqref{eq:reduced equation rho}, and the solutions to the two-dimensional model \eqref{eq:original_system}, averaged along the $y$ direction. The dotted line corresponds to $\mathbf{u} = \mathbf{u}_e$, the dashed line to $\mathbf{u} = \mathbf{u}_L$. For case (a), $\mathbf{u} = (1, 0)^T$ for both the eikonal and the Laplace equation. Since the solution of the 2D problem does not depend on $y$, the 1D approximation is exact. As expected, solutions match in this case.\\
Generally speaking, the 1D approximation is similar or better for the Laplace field i.e.~when $\mathbf{u} = \mathbf{u}_L$. When $\mathbf{u} = \mathbf{u}_e$, there are large discrepancies, in particular in the low density regime. In this case, the 2D solution presents significantly higher densities ahead of narrower sections. One possible explanation is the collision of streamlines on the edge of these narrow sections.
The low density regime ($\rho < \frac{1}{2}$) shows a boundary layer near $x=1$, while in the high density regime ($\rho > \frac{1}{2}$), it is close to $x=0$.

When $k$ is not smooth, that is cases (c) and (d), there is a noticeable discrepancy between the 1D solution and the averaged 2D solution, also for the Laplace field.

In the low density regime, no boundary layer is present close to $x=0$, so that one can expect the value of $\rho(0)$ (and thus of $\rJ \simeq k(0) \rho(0) \left(1-\rho(0)\right)$) to be given by the boundary condition: $\rho \simeq \alpha$. A similar statement holds true for the high density regime at the right boundary. This explains why for both regimes the value of $\rJ$ is very similar across the chosen geometries, if one takes into account the difference in width at the boundary: in (a) the width is $1$ on the left, as opposed to $2$ for other geometries. Similarly, (d) has a width of $2$ on the right instead of $1$.\\
Finally, we observe that in cases (c) and (d), in which the width is piecewise constant, the profile of $\rho$ on each segment is similar to one of the three found in the case (a) where $k$ is constant. This can be explained by the fact that $\rho$ solves an equation similar to \eqref{eq:rho constant width} on each segment - however, with appropriate interface conditions. The profiles are then given by $\T$. An interesting problem is how to define the correct interface conditions. We leave this question for further research. 

\begin{table}
    \centering
    \begin{tabular}{cccccc}
        \toprule
        &\multicolumn{2}{c}{\textbf{Geometry}} & \multicolumn{3}{c}{\textbf{Regime}}
        \\
                                              &\multicolumn{2}{c}{Streamlines} & Low density & High density & High flux \\
                                                    &$\mathbf{u}_e$ & $\mathbf{u}_L$ & $\alpha=0.05$, $\beta=0.2$ &
        $\alpha=0.3$, $\beta=0.1$ &
        $\alpha=0.8$, $\beta=0.8$
        \\ \midrule
        \raisebox{1.6cm}{(a)}
                                  &
        \raisebox{1.3cm}{
    \includegraphics{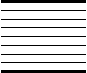}} &
        \raisebox{1.3cm}{
            \includegraphics{tikz/profile_table-figure0.pdf}
        }
        &
        \includegraphics{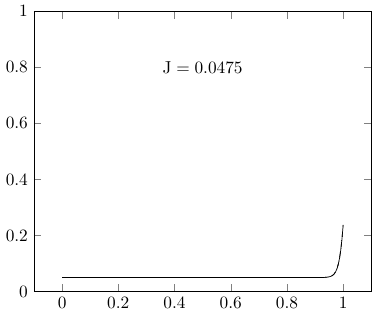}
        &
        \includegraphics{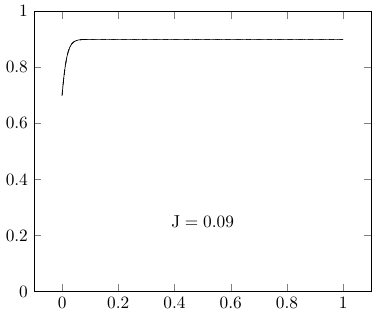}
        &
        \includegraphics{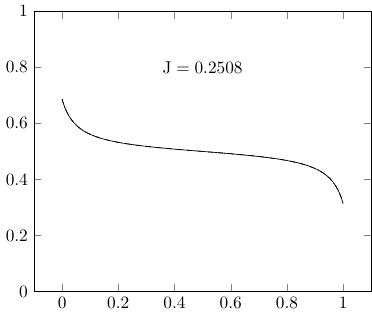}
        \\
    \raisebox{1.6cm}{(b)} &
        \raisebox{1.3cm}{
\includegraphics{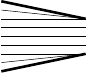}} &
        \raisebox{1.3cm}{
        \includegraphics{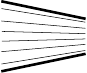}
    }
        &
        \includegraphics{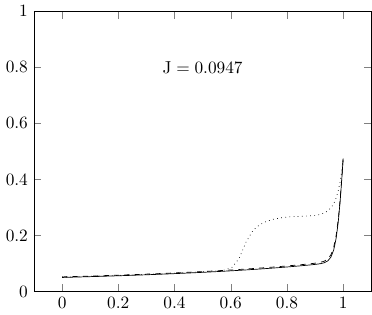}
        &
        \includegraphics{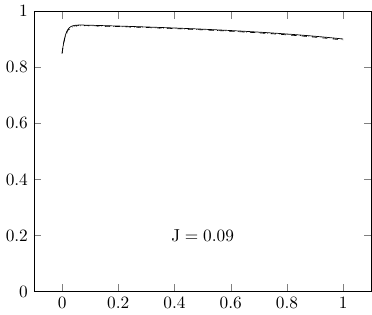}
        &
        \includegraphics{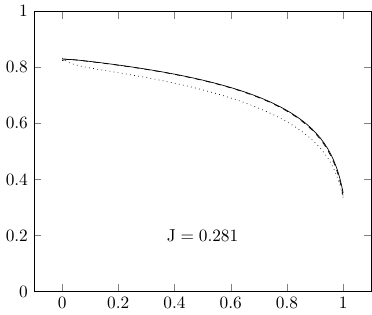}
        \\
   \raisebox{1.6cm}{(c)} &
        \raisebox{1.3cm}{
   \includegraphics{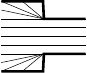}} &
        \raisebox{1.3cm}{
        \includegraphics{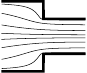}
    }
        &
        \includegraphics{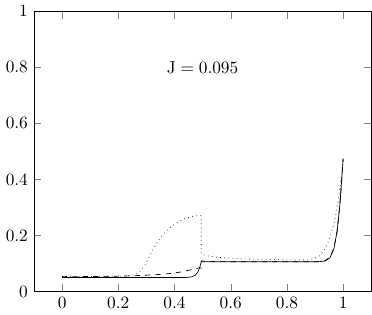}
        &
        \includegraphics{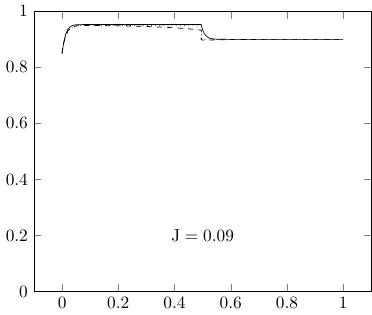}
        &
        \includegraphics{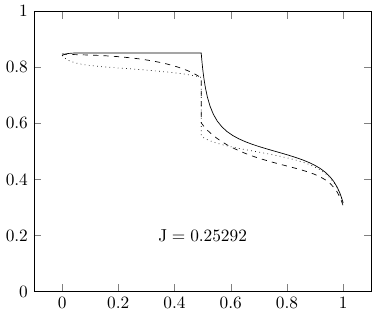}
        \\
   \raisebox{1.6cm}{(d)} &
        \raisebox{1.3cm}{
   \includegraphics{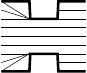}} &
        \raisebox{1.3cm}{
        \includegraphics{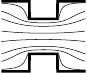}
    }
        &
        \includegraphics{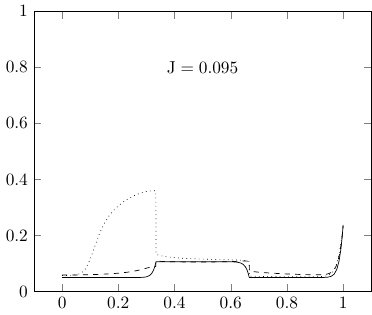}
        &
        \includegraphics{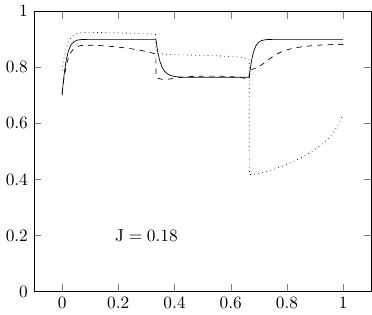}
        &
        \includegraphics{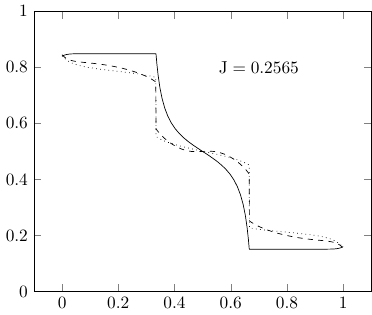}
        \\ \bottomrule
    \end{tabular}
    \caption{Comparison of the stationary profiles for different geometries and different regimes, with $\varepsilon = 10^{-2}$. The first column shows both $\mathbf{u}_e$ (left) and $\mathbf{u}_L$ (right). Legend: Solid lines refer to solutions of \eqref{eq:reduced equation rho}, dotted lines indicate the $y$-average of the solution to \eqref{eq:original_system} for $\mathbf{u} = \mathbf{u}_e$, and dashed lines represent the $y$-average of the solution to \eqref{eq:original_system} for $\mathbf{u} = \mathbf{u}_L$.}
    \label{tab:profile comparison}
\end{table}

\subsection{Singular orbits vs.~viscous profiles}

The GSPT analysis of Section~\ref{sec:vanvis} provides us with limiting profiles as $\varepsilon \rightarrow 0$. We compare them with the solutions of \eqref{eq:reduced equation rho}-\eqref{eq:boundary conditions reduced rho} for small $\varepsilon$ in the case $k(x) = 2-x$.\\
We start with a quantitative comparison for a specific choice of $(\alpha,\beta)$ in Figure~\ref{fig:regI}, which shows the convergence of the solution in the inflow limited regime to a singular solution of type 1.
\begin{figure}[!h]
     \centering
     \raisebox{0.1cm}{\includegraphics[width=0.485\textwidth]{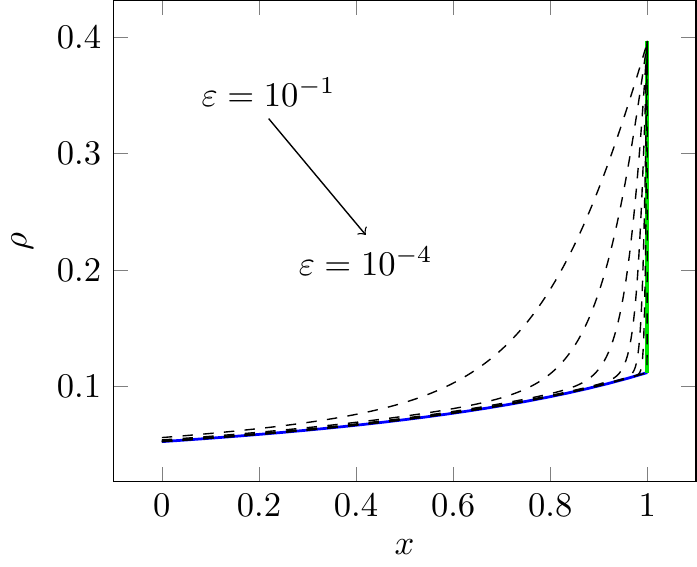}}
     \hfill
     \includegraphics[width=0.475\textwidth]{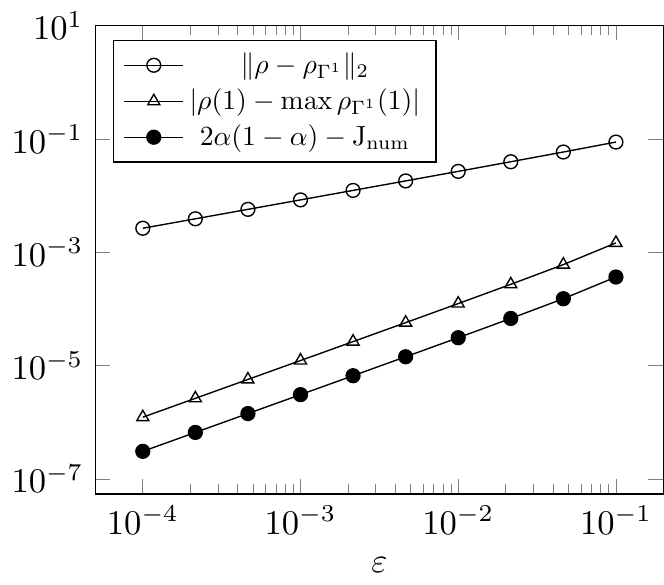}
     \caption{Stationary profiles for the particular choices of parameters $(\alpha,\beta) = (5.2\times 10^{-2}, 2.5\times 10^{-1})$, which corresponds to region $\mathcal{G}_1$. \emph{Left:} numerical solutions for various values of $\varepsilon$ approaching the density profile given by the singular orbit $\rho_{\Gamma^1}$. \emph{Right:} Difference between $\rho_{\Gamma^1}$ and the sequence of numerical solutions as $\varepsilon$ decreases: ($\circ$) $L^2$-error, ($\triangle$) pointwise error at $x=1$, i.e.~the difference between $\rho(1)$ and the maximum value of the boundary layer given by the singular orbit, (\textbullet) limiting flux vs numerical flux $\rJ_\mathrm{num}$. We observe that the error at the boundary is roughly of order $O(\varepsilon)$, and as expected, the convergence in $L^2$ norm is of order roughly $O(\varepsilon^\frac{1}{2})$, due to the boundary layer.}
     \label{fig:regI}
 \end{figure}

 The convergence of solutions towards the singular orbits as given by the GSPT is not uniform w.r.t. $(\alpha, \beta)$. This implies that a direct computation of the phase diagram (Figure \ref{fig:bd_sing}) for a fixed $0 < \varepsilon \ll 1$ is not possible. Therefore, we perform a qualitative comparison across the $(\alpha,\beta)$ parameter space, by looking at the phase diagrams both in the case $\varepsilon = 10^{-3}$ and $\varepsilon = 0$.
The results for these computational experiments are shown in Figure~\ref{fig:phase diagram funnel diffusive}. There is good qualitative agreement between the two cases, in particular
the location and variation of boundary layers (increasing or decreasing) match in all regimes. There are also no boundary layers visible
in the numerical solution for $\varepsilon > 0$ at boundary curves $\gamma_{ij}$, except between $\mathcal{G}_3$ and $\mathcal{G}_4$ where a right boundary layer is present,
as is the corresponding singular case. This indicates that the singular phase diagram provides a good description of the behaviour for $0 < \varepsilon \ll 1$.

\begin{figure}[h]
 \begin{center}
 \includegraphics[width=0.7\textwidth]{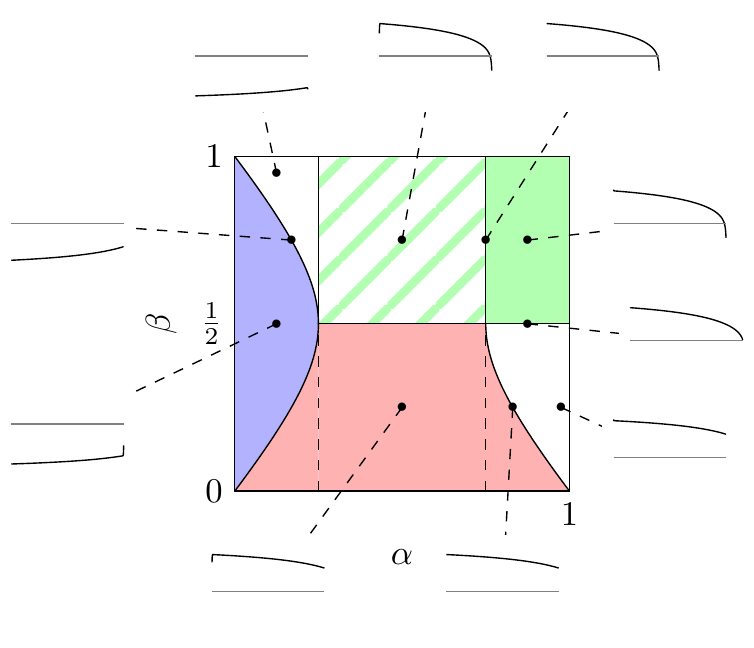}
 \end{center}
 \caption{
 Phase diagram showing different profiles $\rho$ of \eqref{eq:reduced equation rho}-\eqref{eq:boundary conditions reduced rho} for $k(x) = 2 - x$ and $\varepsilon = 10^{-3}$. The underlying regions in the $(\alpha,\beta)$ space are the ones defined for the singular limit $\varepsilon = 0$.
 The 4 cases in the top right-hand corner (green and green cross-hatched regions) involve the fold line $S$ and present a right boundary layer which is not as sharp as cases away from $S$, for example in the blue and red areas. The gray line indicates the value $\rho= \frac{1}{2}$. We refer to Remark \ref{r:layers} for detailed comments.
 }
 \label{fig:phase diagram funnel diffusive}
\end{figure}

\begin{remark}\label{r:layers}
We now give a brief interpretation of the bifurcation diagram in Figure \ref{fig:phase diagram funnel diffusive} in terms of pedestrian dynamics. Within each regime we either observe low or high-density densities, that is above or below the value $\frac{1}{2}$. The width of the boundary layers at the entrances and exits depends on the diffusivity, their height on the respective inflow and outflow rates. We observe high density regimes if the inflow rate is larger than the outflow (red region in \ref{fig:phase diagram funnel diffusive}) with a sharp increase close to the entrance. High outflow rates relate to a fast drop at the exit, shown in $\mathcal{G}_2$, $\mathcal{G}_3$ and $\mathcal{G}_4$.\\
In practice, the location of the boundary layers as well as the density regime are most relevant. The former allows to draw conclusions about the formation of congested (and potentially dangerous) regimes, while the latter can be used to estimate the capacity of rooms.
\end{remark}
\subsection{Phase diagrams for geometries with bottlenecks}
In this section we investigate domains which have a narrow midsection.
In particular we consider the following situations:
\begin{enumerate}
    \item[(d)] A piecewise constant bottleneck, that is $k(x) = 2 - \mathbbm{1}_{[\frac{1}{3}, \frac{2}{3}]}(x)$, as already defined above.
    \item[(d')] A smooth bottleneck
        \begin{equation*}
            k(x) = \begin{cases}
                2 & \text{for } x \in [0, \tfrac{1}{3}) \cup (\tfrac{2}{3}, 1]\,,
                \\
                c_1 + c_2 \cos\left(3 \cdot 2 \pi(x-\frac{1}{2})\right)& \text{otherwise}\,,
            \end{cases}
        \end{equation*}
        where the parameters $(c_1,c_2)$ are chosen such that $k\left(\tfrac{1}{3}\right) = k\left(\tfrac{2}{3}\right) = 2$, $k\left(\tfrac{1}{2}\right) = 1$, $k'\left(\tfrac{1}{2}\right) = 0$. This choice of $k$ differs from the one in \eqref{eq:k_smooth}, to better isolate boundary layers.
\end{enumerate}

We note that the general behaviour of stationary states is comparable for (d) and (d'). In summary, we observe (see Figure~\ref{fig:phase diagram cosine}):
\begin{enumerate}[$(O_1)$]
    \item \label{it:low to high alpha}  The density $\rho$ transitions out of the low density regime very quickly as $\alpha$ increases slightly: for this choice of parameters, $\rho$ is very sensitive on the value of $\alpha$.
    \item \label{it:high flux} The density $\rho$ takes values larger (resp.~smaller) than $\tfrac{1}{2}$ on the left (resp.~right) of the interval. The boundary layers at the endpoints can either be increasing or decreasing.
    \item \label{it:low to high beta} The density $\rho$ transitions to the high density regime very quickly as $\beta$ decreases slightly: $\rho$ depends very sensitively on $\beta$.
    \item \label{it:high density boundary change monotonicity alpha}
        In the high density regime, the left boundary layer changes from increasing to decreasing as $\alpha$ increases.
    \item \label{it:diagonal} Along the diagonal $\alpha = \beta$, the boundary layer moves to the centre of the domain as $\alpha$ and $\beta$ decrease.
    \item \label{it:low density boundary change monotonicity beta}
        In the low density regime, the right boundary layer changes from increasing to decreasing as $\beta$ increases.
\end{enumerate}

Loosely speaking, the behaviour around \ref{it:high density boundary change monotonicity alpha} and \ref{it:low density boundary change monotonicity beta} is comparable with the one in a straight or a  closing corridor, corresponding to high density and low density regimes, respectively. They
relate to the reduced dynamics restricted to the attractive and repulsive manifolds $\mathcal{C}_0^a$ and $\mathcal{C}_0^r$, respectively.\\
In the upper-right quadrant \ref{it:high flux}, we observe a transition from $\mathcal{C}_0^a$ on the left of the interval, to $\mathcal{C}_0^r$ on the right, which cannot happen for 
$k'<0$.
This hints at the involvement of the set $S$ of critical points of $\mathcal{C}_0$, which complicate the analysis.
How this regime appears from the high/low density regime \ref{it:high density boundary change monotonicity alpha} and \ref{it:low density boundary change monotonicity beta} is shown in groups \ref{it:low to high beta} and \ref{it:low to high alpha}.\\
Finally, group \ref{it:diagonal} shows how starting from situation \ref{it:high flux} for $\alpha = \beta = \frac{1}{2}$, the boundary layers migrate to the middle of the interval as the value of $\alpha$ decreases and eventually meet. Here, the behaviour is not well defined in the limit $\varepsilon = 0$, as these are, in fact, the regimes such that canard orbits emerge, i.e.~orbits which transition in or near repelling slow manifolds for long times.
This claim is further supported by the observed numerical degeneracy in the corresponding numerical experiments. 

\begin{figure}[h]
    \begin{center}
    \includegraphics[width=0.45\textwidth]{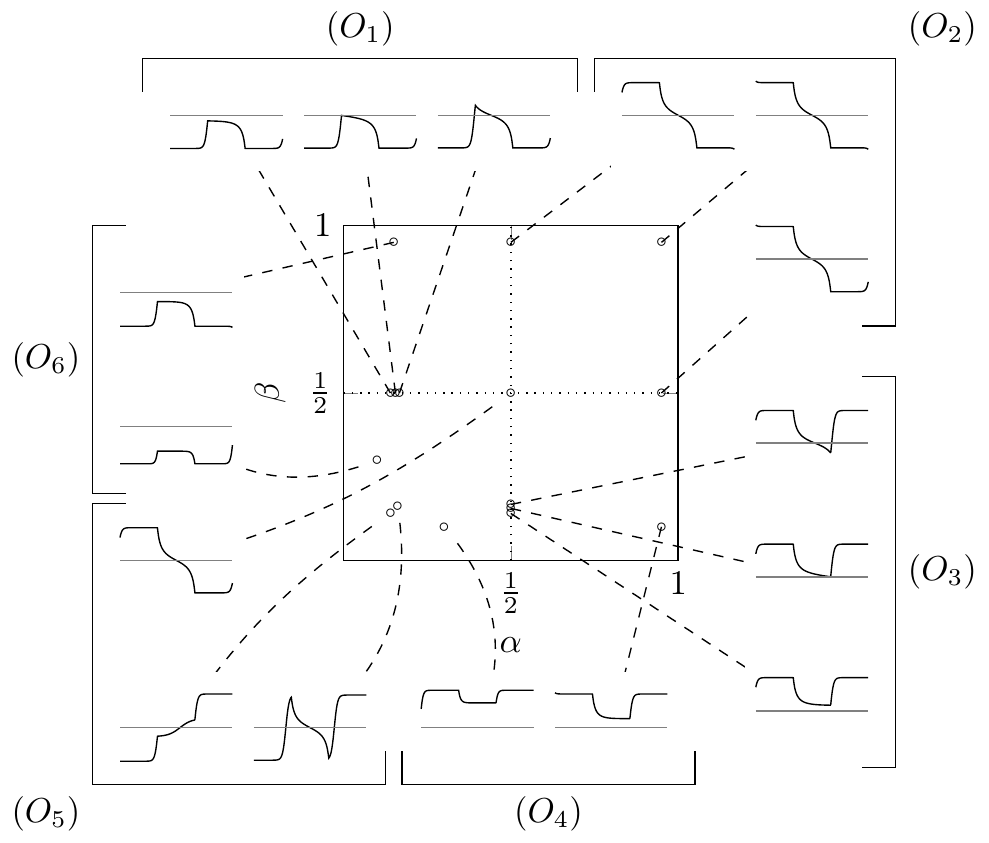}
    \hfill
    \includegraphics[width=0.45\textwidth]{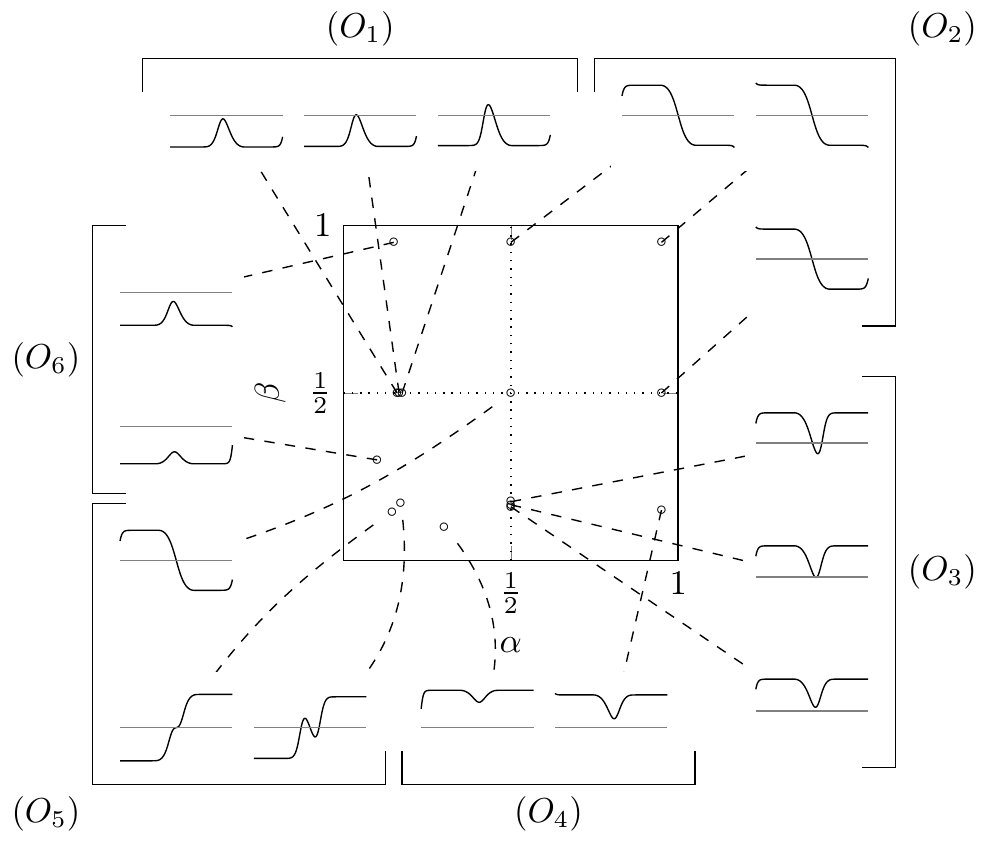}
    \\
    \hfill (d) \hfill\hfill\hfill (d') \hfill~
    \end{center}
    \caption{Phase diagrams for the cases (d) and (d'). Around the central plot, the $\rho$ profiles corresponding to a particular choice of $\alpha$ and $\beta$ are shown. The horizontal gray line highlights the value $\frac12$. The parameter $\varepsilon$ is set to $10^{-2}$.
    }
    \label{fig:phase diagram cosine}
\end{figure}

\section{Conclusion} \label{sec:conclusion}
In this paper, we analyse the stationary profiles of a mean-field model for unidirectional pedestrian flows for different inflow and outflow boundary conditions and geometries. The considered 1D model is based on averaging the pedestrian flow over the cross section of the domain and corresponds to a viscous Burgers' type equation with nonstandard boundary conditions. Its stationary profiles exhibit boundary layers at entrances and exits depending on the inflow and outflow rates. We use PDE techniques as well as GSPT to analyse the structure and location of these layers and support our findings with numerical experiments. These results are valid in straight and closing or opening domains. However, systematic computational experiments as well as preliminary analytic results for more general domains indicate that our strategy can be further developed to cover the case of bottlenecks. The numerical results illustrate the stationary profiles associated to more complex domains, such as corridors of
piecewise constant width.

These can be viewed as a concatenation of the rigorously analysed stationary straight corridor profiles with suitable, but currently unknown, interface conditions. Furthermore, we discuss the approximation quality of the 1D area averaged model for a range of inflow and outflow conditions as well as computational domains. It is shown to depend strongly on the choice of the preferred direction $\mathbf{u}$ for the 2D model. Our experiments show that solutions corresponding to the vector field given by solving a Laplace equation are well approximated, even if the width of the domain is large with respect to the length. In particular, the approximation can be much better than when using the eikonal equation to compute the vector field $\mathbf{u}$.

Future research will focus on open modelling and analysis questions. First, we wish to extend our analysis to domains where the function $k(x)$ has a vanishing derivative (such as bottlenecks). In this situation, we are particularly interested in solutions involving canard segments, which emerge for some parameter regimes and cross the line of fold points at a folded singularity. Moreover, using GSPT, we plan to further investigate the critical case where
an interior layer occurs, the location of which is not known for $\varepsilon=0$, that is 
for $(\alpha,\beta)$ on the curve $\gamma_{15} \cup \gamma_{23}$. Another interesting perspective lies in model development and calibration - we have seen that the area averaged 1D model approximates the full 2D model in certain parameter regimes very well, but fails in others. These shortcomings are caused by the underlying averaging assumption as well as the boundary conditions. We aim to investigate suitable scaling and modelling alternatives in the future. \vspace*{1em}\\

\noindent \textbf{Acknowledgements} AI, GJ and MTW acknowledge partial support from the New Frontier's grant NST-0001 of the Austrian Academy of Sciences \"OAW. MTW thanks Christian Schmeiser for the interesting discussion, which initiated the collaboration with PS and AI on GSPT.\\

\bibliography{ms}{}
\bibliographystyle{abbrv}

\appendix

\section{Geometry of singular solutions of type 2-6} \label{app:sspl}

\begin{figure}[H] 
 \centering
 \begin{minipage}{.3\textwidth}
  \centering
  \def\svgwidth{1\textwidth}
  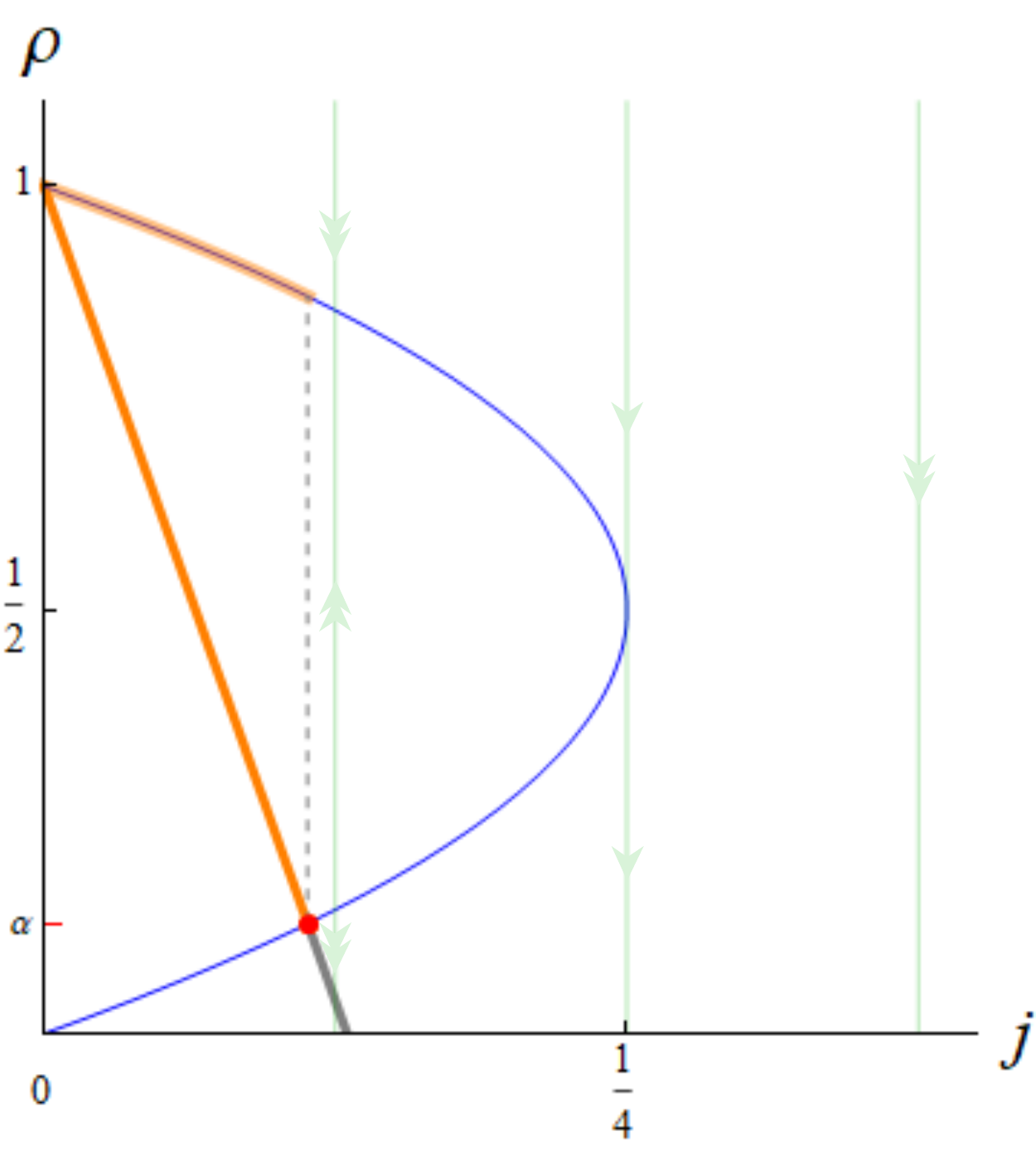\\
  (a) $\xi=0$
 \end{minipage}
 \hspace{.5cm}
 \begin{minipage}{.3\textwidth}
  \centering
  \def\svgwidth{1\textwidth}
  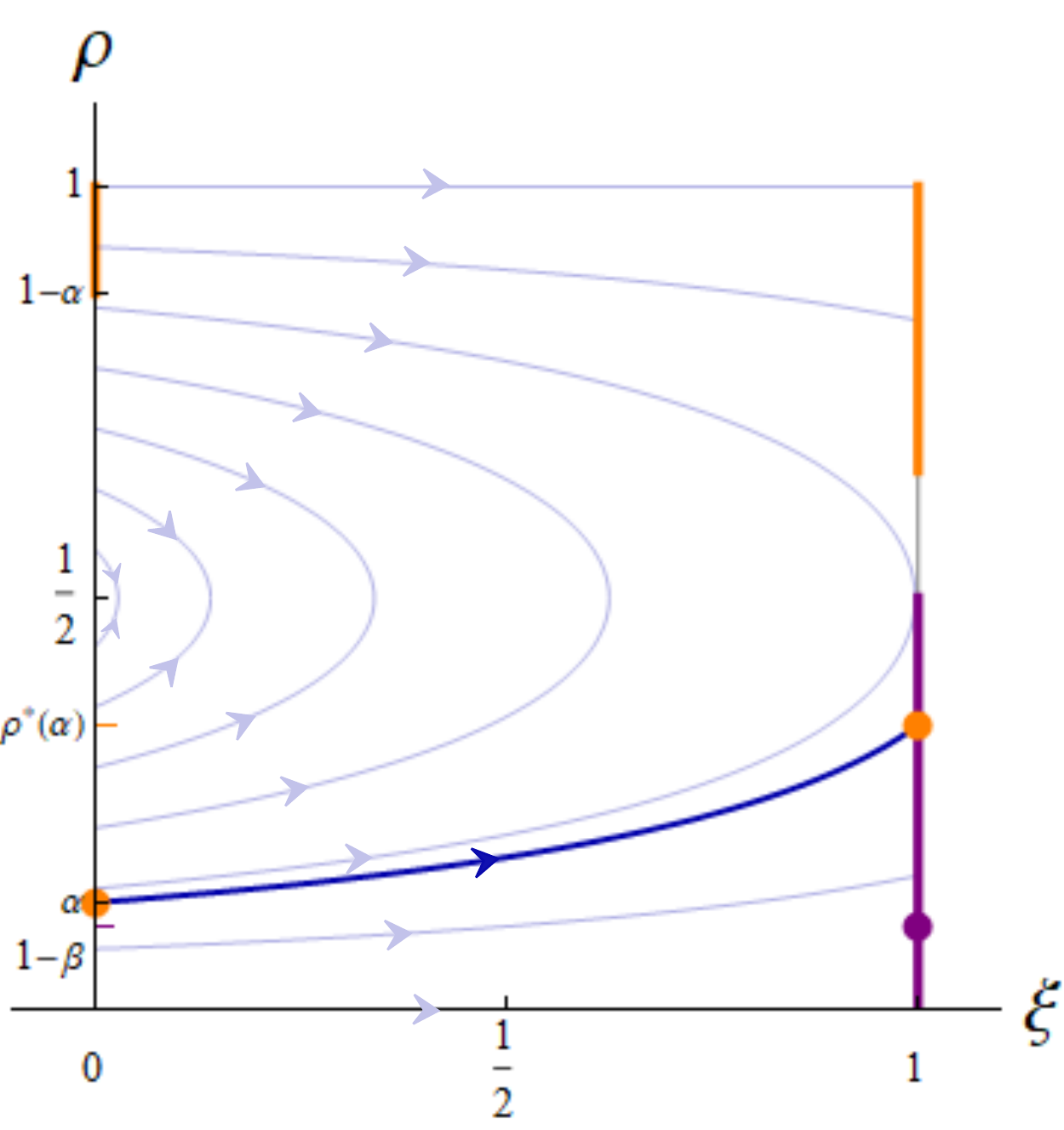\\
  (b) $\xi \in [0,1]$
 \end{minipage}
 \hspace{.5cm}
 \begin{minipage}{.3\textwidth}
  \centering
  \def\svgwidth{1\textwidth}
  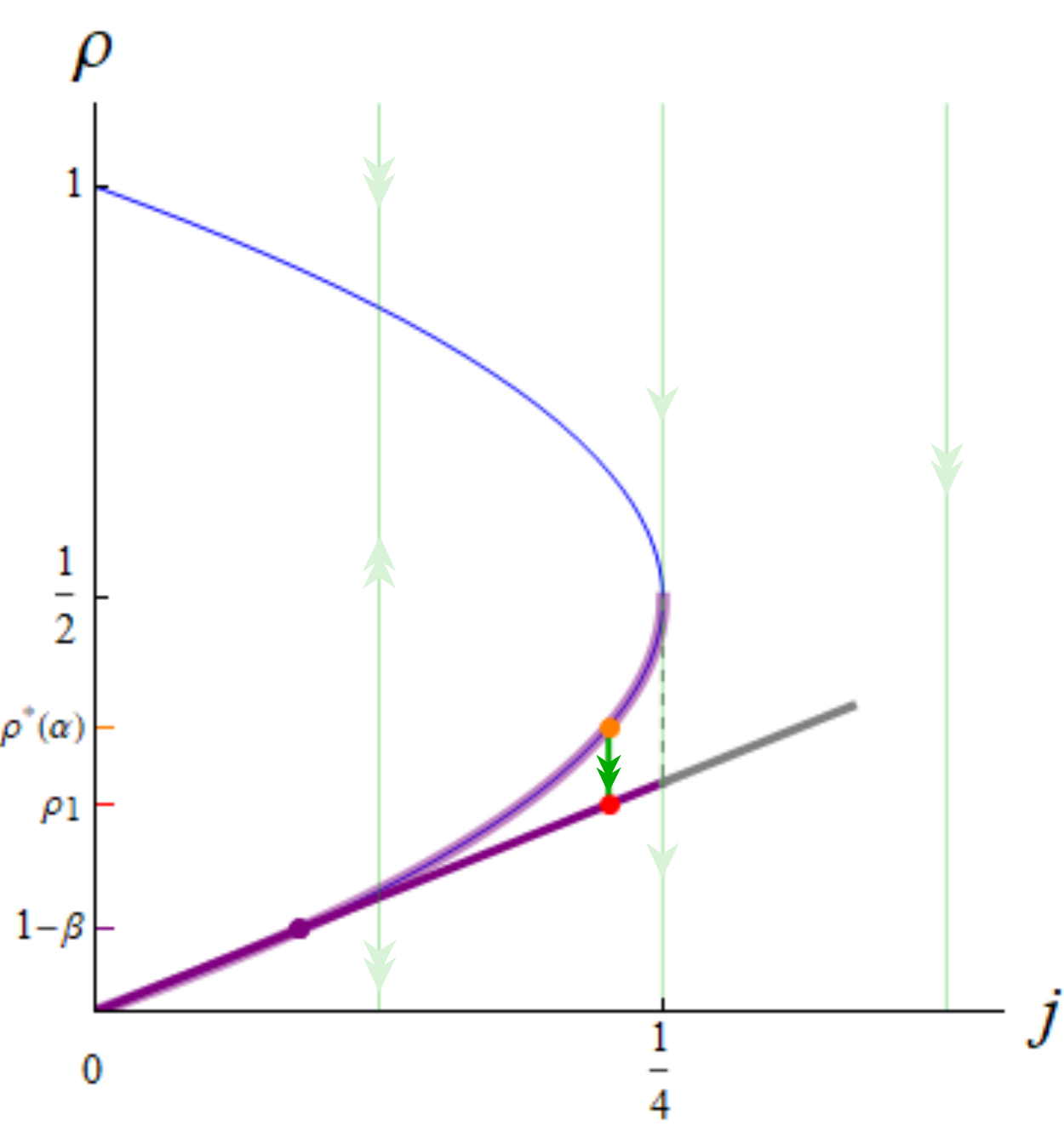\\
  (c) $\xi=1$
 \end{minipage}
  \begin{minipage}{.3\textwidth}
  \centering
  \def\svgwidth{1\textwidth}
  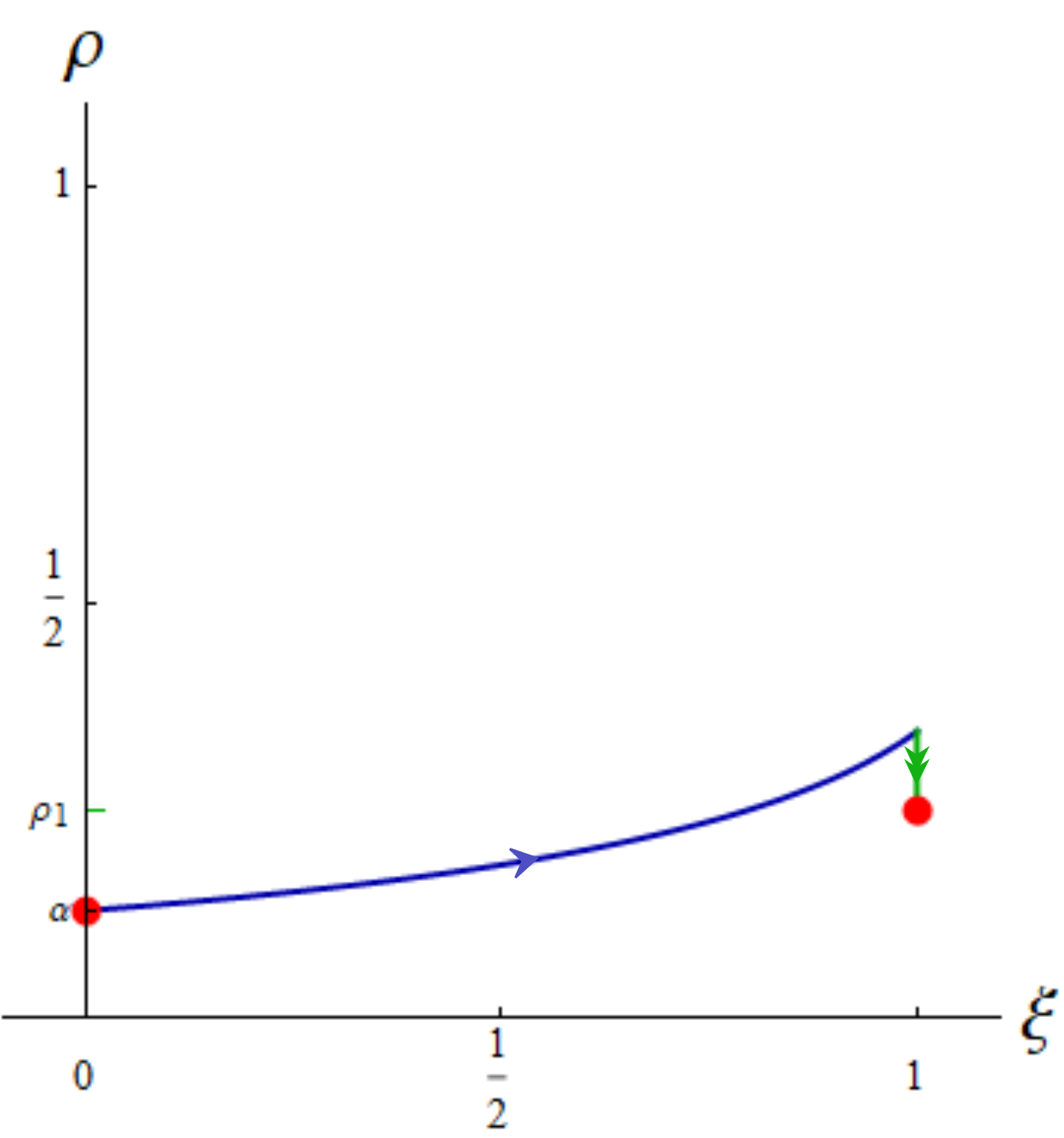\\
  (d) 
 \end{minipage}
 \caption{Schematic representation of a singular solution of type 2. (a) Boundary conditions at $\xi=0$ in $(j,\rho)$-space: the orange line is $\mathcal{L}$, while the orange curve is $\mathcal{L}^+$. The red dot represents $p_0$, where in this case $p_0=l$. (b) Slow evolution on $\mathcal{C}_0$ from $(0,\alpha)$ to $(1,\rho^\ast(\alpha))$ (blue curve). The orange lines are the projection of $\mathcal{L}$ (at $\xi=0$) and $\mathcal{L}^+$ (at $\xi=1$) on $\mathcal{C}_0$, while the purple one represents the projection of $\mathcal{R}^-$ on $\mathcal{C}_0$ at $\xi=1$. The orange dots correspond to $l$ (at $\xi=0$) and $l_1$ (at $\xi=1$), while the purple dot corresponds to $r$. (c) Here, we consider $\xi=1$ in $(j,\rho)$-space. The red dot corresponds to $p_1$, while the purple line and curve represent the manifolds $\mathcal{R}$ and $\mathcal{R}^-$, respectively. The green line corresponds to the layer of the singular orbit where $\rho$  decreases from $\rho^\ast(\alpha)$ to $\rho_1=\frac{\alpha(1-\alpha)k(0)}{\beta k(1)}$. (d) Singular solution of type 2 in $(\xi,\rho)$-space.}
 \label{fig:singsol_2}
 \end{figure}

 \begin{figure}[H] 
 \centering
 \begin{minipage}{.3\textwidth}
  \centering
  \def\svgwidth{1\textwidth}
  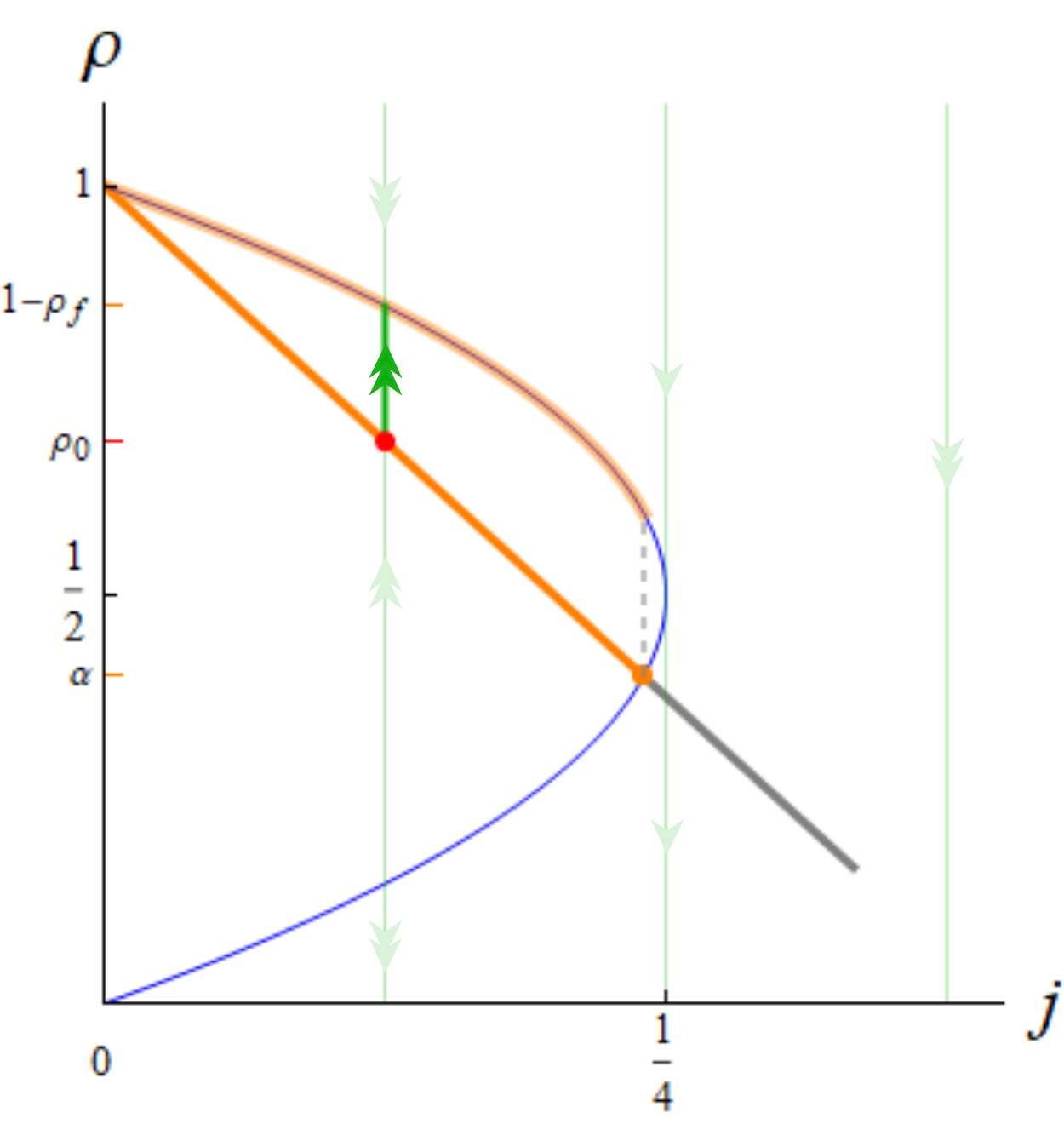\\
  (a) $\xi=0$
 \end{minipage}
 \hspace{.5cm}
 \begin{minipage}{.3\textwidth}
  \centering
  \def\svgwidth{1\textwidth}
  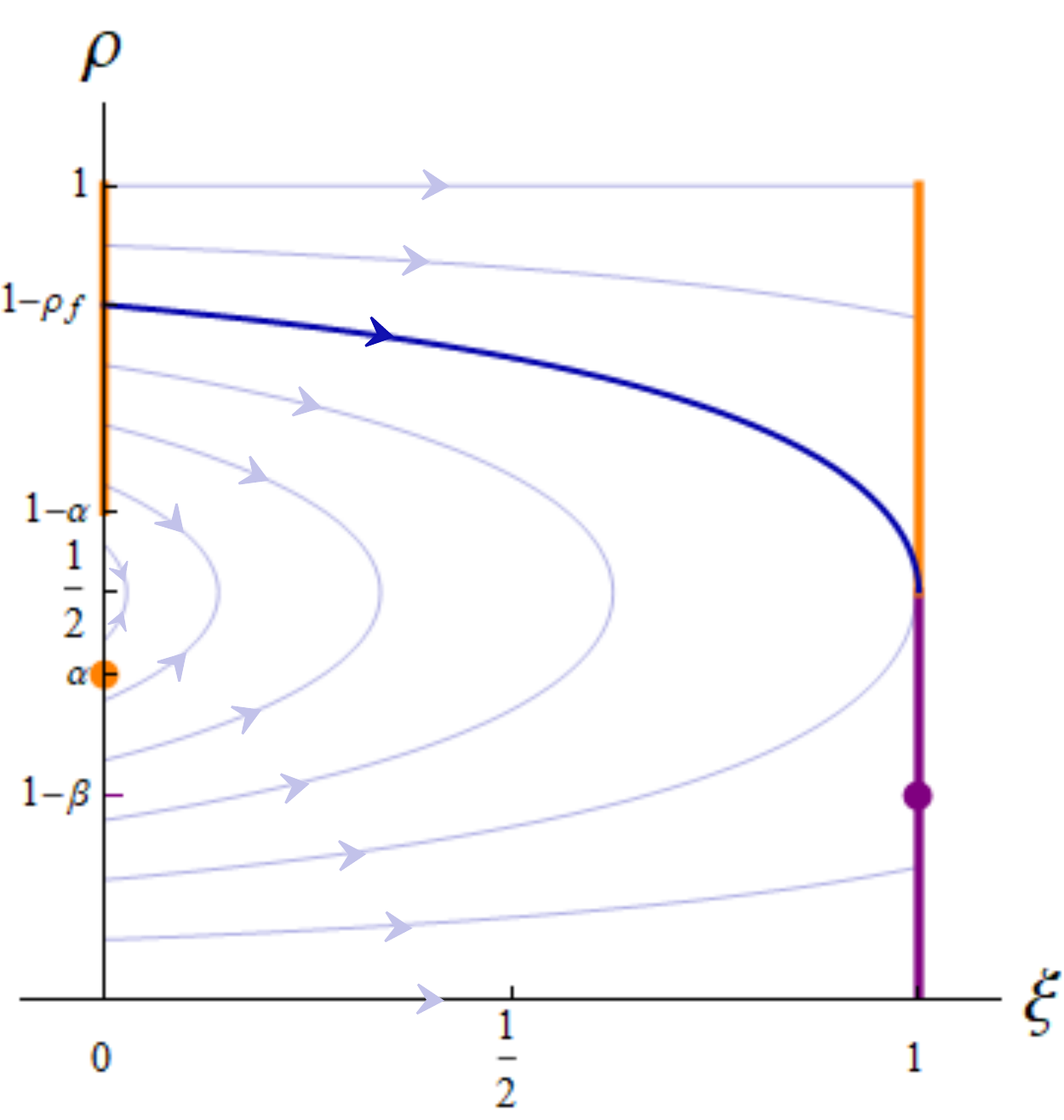\\
  (b) $\xi \in [0,1]$
 \end{minipage}
 \hspace{.5cm}
 \begin{minipage}{.3\textwidth}
  \centering
  \def\svgwidth{1\textwidth}
  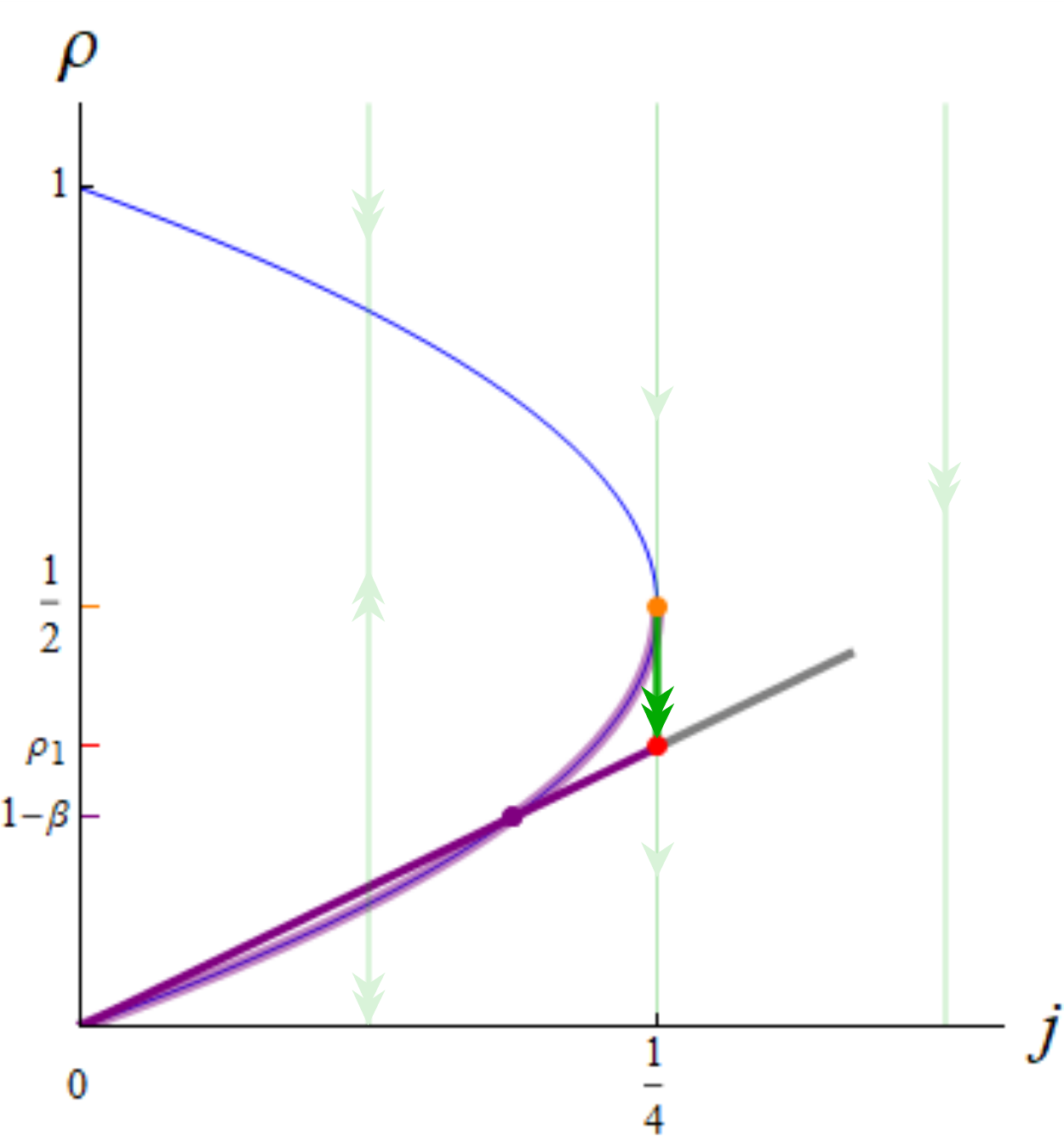\\
  (c) $\xi=1$
  \end{minipage}
  \begin{minipage}{.3\textwidth}
  \centering
  \def\svgwidth{1\textwidth}
  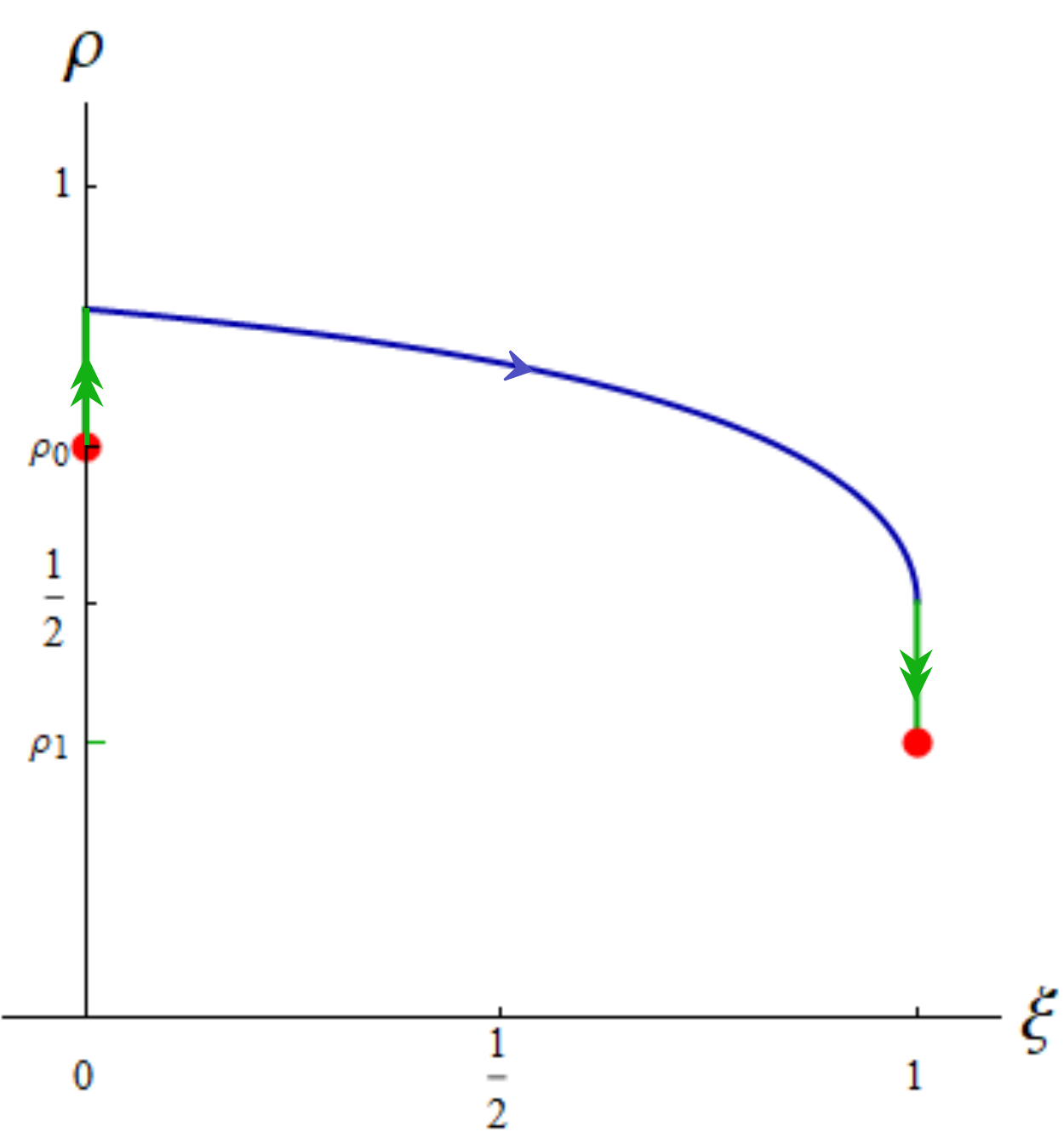\\
  (d) 
 \end{minipage}
 \caption{Schematic representation of a singular solution of type 3. (a) Boundary conditions at $\xi=0$ in $(j,\rho)$-space: the orange line is $\mathcal{L}$, while the orange curve is $\mathcal{L}^+$. The red dot represents $p_0$, and the green line illustrates the layer at $\xi=0$ where $\rho$ increases from $\rho_0=1-\frac{k(1)}{4\alpha k(0)}$ to $\rho(0,\rho_0)=1-\rho_f$. (b) Slow evolution on $\mathcal{C}_0$ from $(0,1-\rho_f)$ to $(1,\frac12)$ (blue curve). The orange lines are the projection of $\mathcal{L}$ (at $\xi=0$) and $\mathcal{L}^+$ (at $\xi=1$) on $\mathcal{C}_0$, while the purple one represents the projection of $\mathcal{R}^-$ on $\mathcal{C}_0$ at $\xi=1$. The orange dots correspond to $l$ (at $\xi=0$) and $l_1$ (at $\xi=1$), while the purple dot corresponds to $r$. (c) Here, we consider $\xi=1$ in $(j,\rho)$-space. The red dot corresponds to $p_1$, while the purple line and curve represent the manifolds $\mathcal{R}$ and $\mathcal{R}^-$, respectively. The green line corresponds to the layer of the singular orbit where $\rho$ decreases from $\frac12$ to $\rho_1=\frac{1}{4\beta}$. (d) Singular solution of type 3 in $(\xi,\rho)$-space.}
 \label{fig:singsol_3}
 \end{figure}

 \begin{figure}[H] 
 \centering
 \begin{minipage}{.3\textwidth}
  \centering
  \def\svgwidth{1\textwidth}
  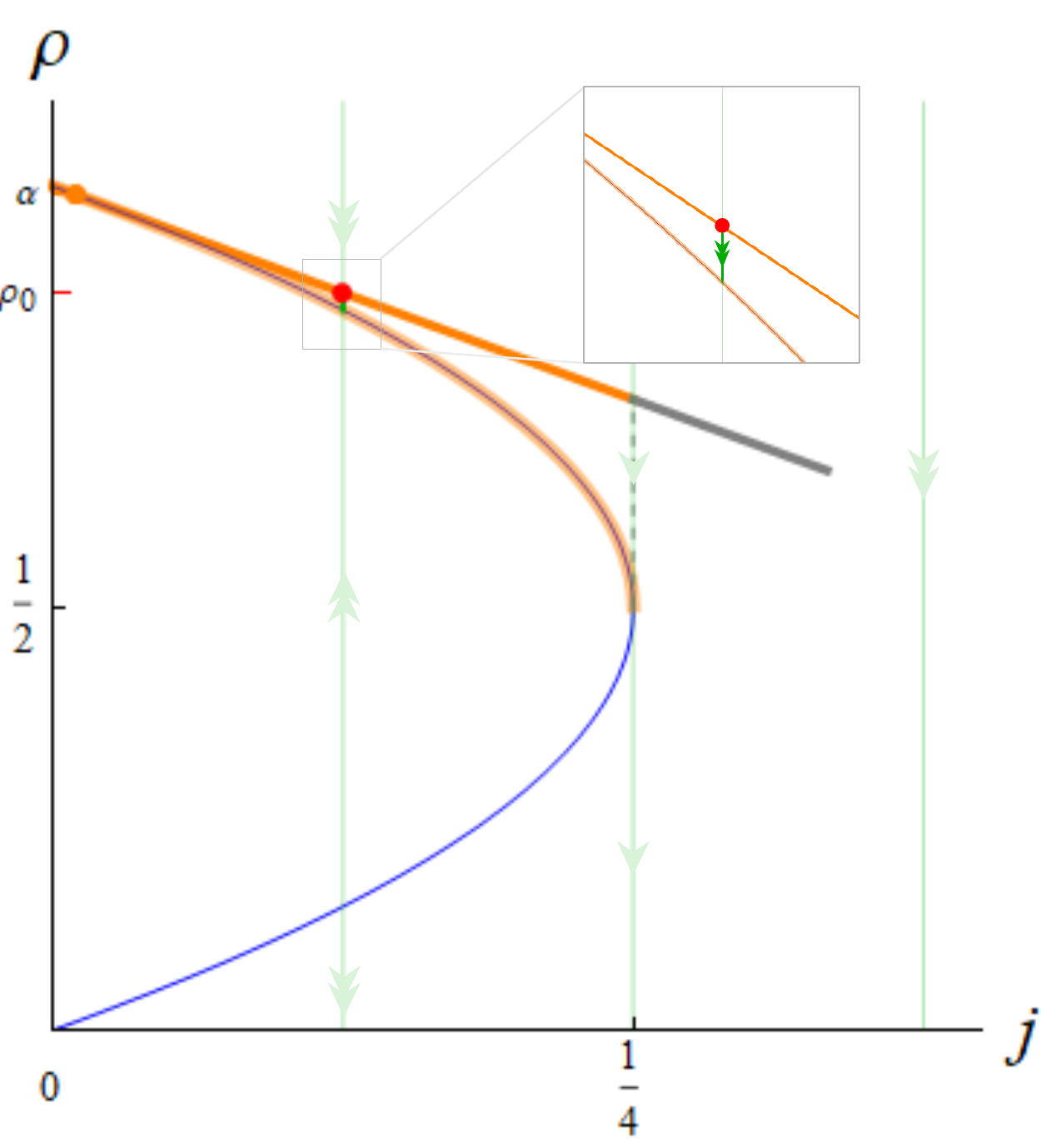\\
  (a) $\xi=0$
 \end{minipage}
 \hspace{.5cm}
 \begin{minipage}{.3\textwidth}
  \centering
  \def\svgwidth{1\textwidth}
  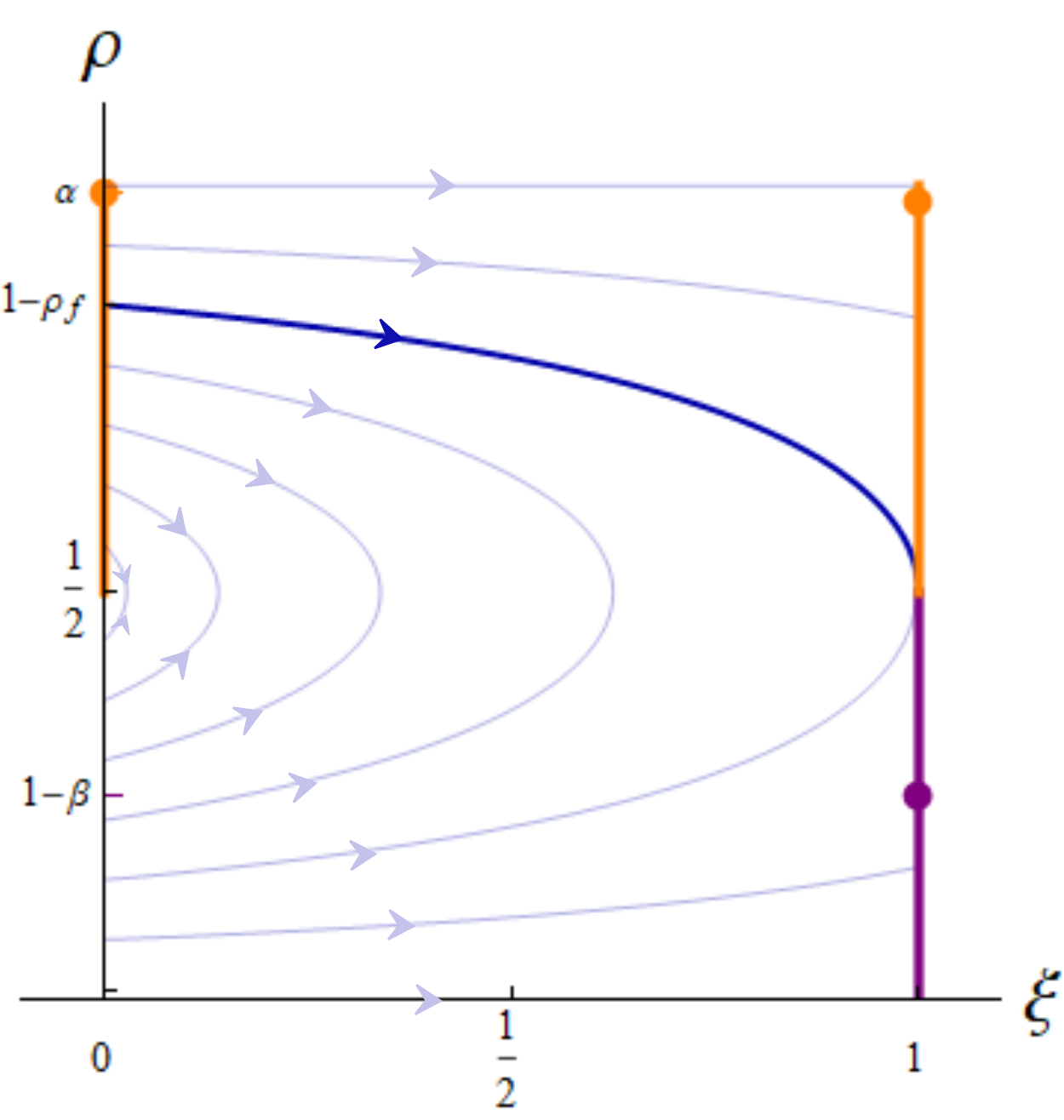\\
  (b) $\xi \in [0,1]$
 \end{minipage}
 \hspace{.5cm}
 \begin{minipage}{.3\textwidth}
  \centering
  \def\svgwidth{1\textwidth}
  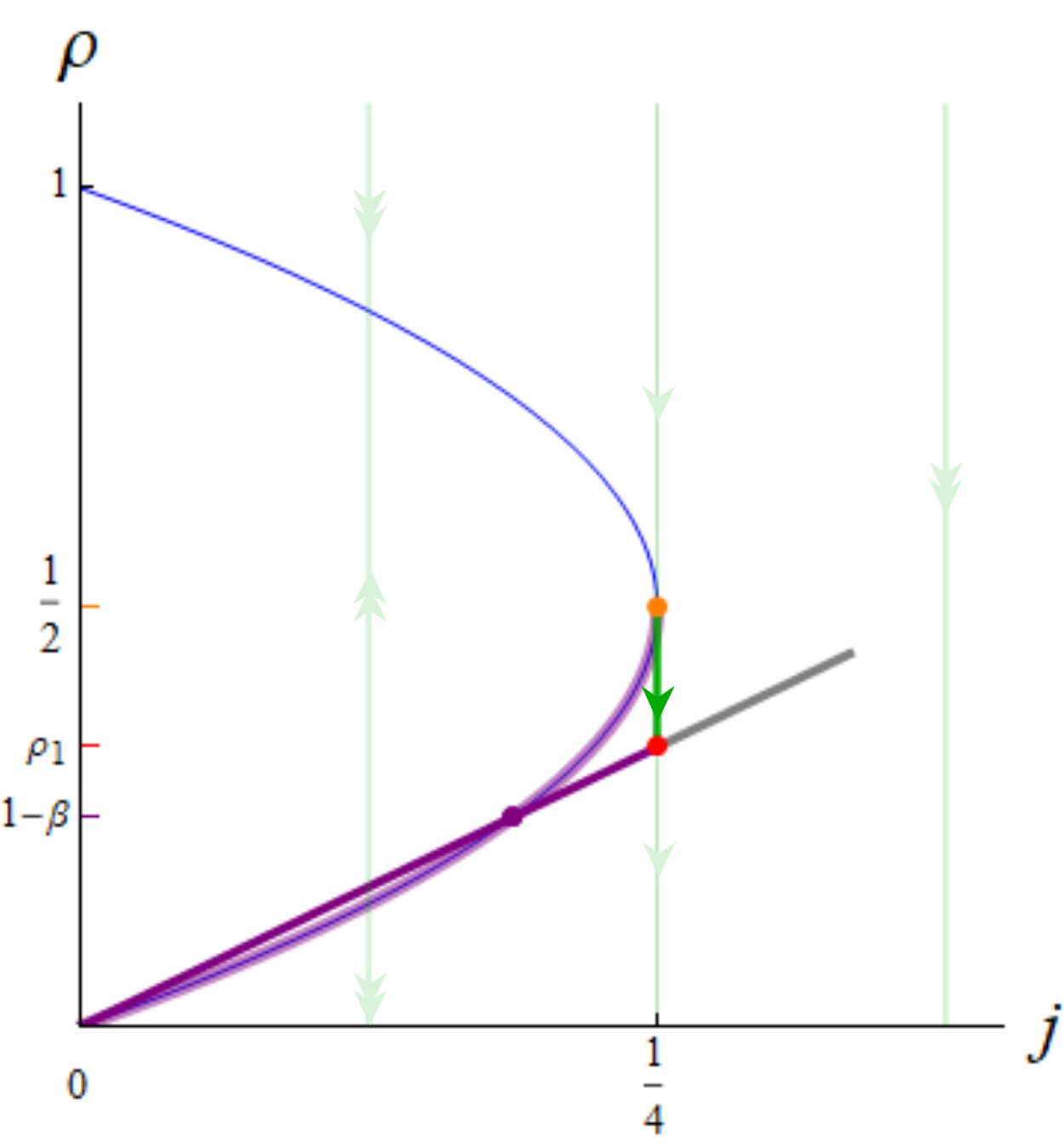\\
  (c) $\xi=1$
 \end{minipage}
  \begin{minipage}{.3\textwidth}
  \centering
  \def\svgwidth{1\textwidth}
  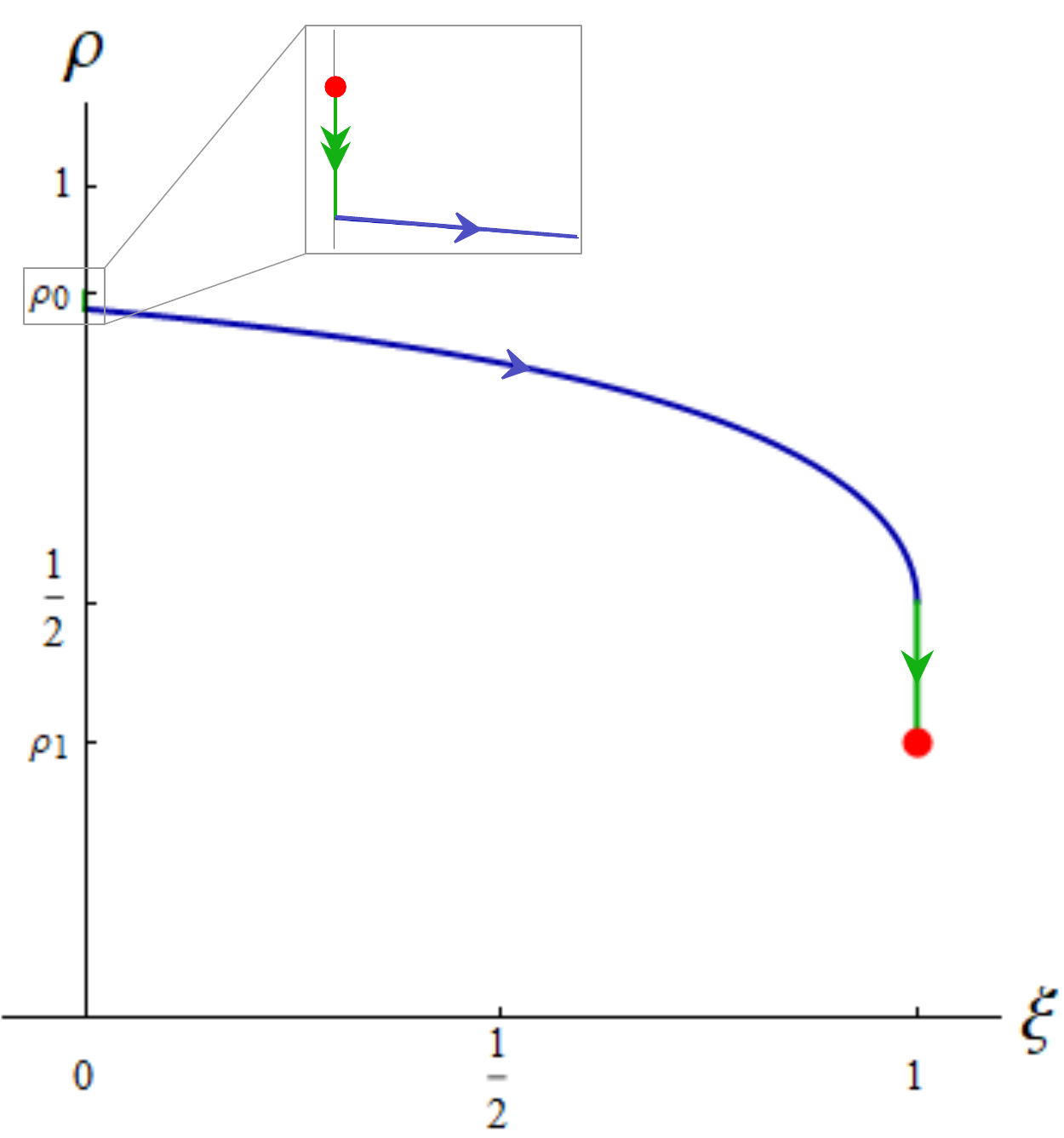\\
  (d) 
 \end{minipage}
 \caption{Schematic representation of a singular solution of type 4. (a) Boundary conditions at $\xi=0$ in $(j,\rho)$-space: the orange line is $\mathcal{L}$, while the orange curve is i.e.~$\mathcal{L}^+$. The red dot represents $p_0$, and the green line illustrates the layer at $\xi=0$ where $\rho$ decreases from $\rho_0=1-\frac{k(1)}{4\alpha k(0)}$ to $\rho(0,\rho_0)=1-\rho_f$ (since the range of $\rho$-values on which this transition occurs is very small, we illustrate a zoomed version in the square on the right). (b) Slow evolution on $\mathcal{C}_0$ from $(0,1-\rho_f)$ to $(1,\frac12)$ (blue curve). The orange lines are the projection of $\mathcal{L}$ (at $\xi=0$) and $\mathcal{L}^+$ (at $\xi=1$) on $\mathcal{C}_0$, while the purple one represents the projection of $\mathcal{R}^-$ on $\mathcal{C}_0$ at $\xi=1$. The orange dots correspond to $l$ (at $\xi=0$) and $l_1$ (at $\xi=1$), while the purple dot corresponds to $r$. (c) Here, we consider $\xi=1$ in $(j,\rho)$-space. The red dot corresponds to $p_1$, while the purple line and curve represent the manifolds $\mathcal{R}$ and $\mathcal{R}^-$, respectively. The green line corresponds to the layer of the singular orbit where $\rho$ decreases from $\frac12$ to $\rho_1=\frac{1}{4\beta}$. (d) Singular solution of type 4 in $(\xi,\rho)$-space.}
 \label{fig:singsol_4}
 \end{figure}

\begin{figure}[H] 
 \centering
 \begin{minipage}{.3\textwidth}
  \centering
  \def\svgwidth{1\textwidth}
  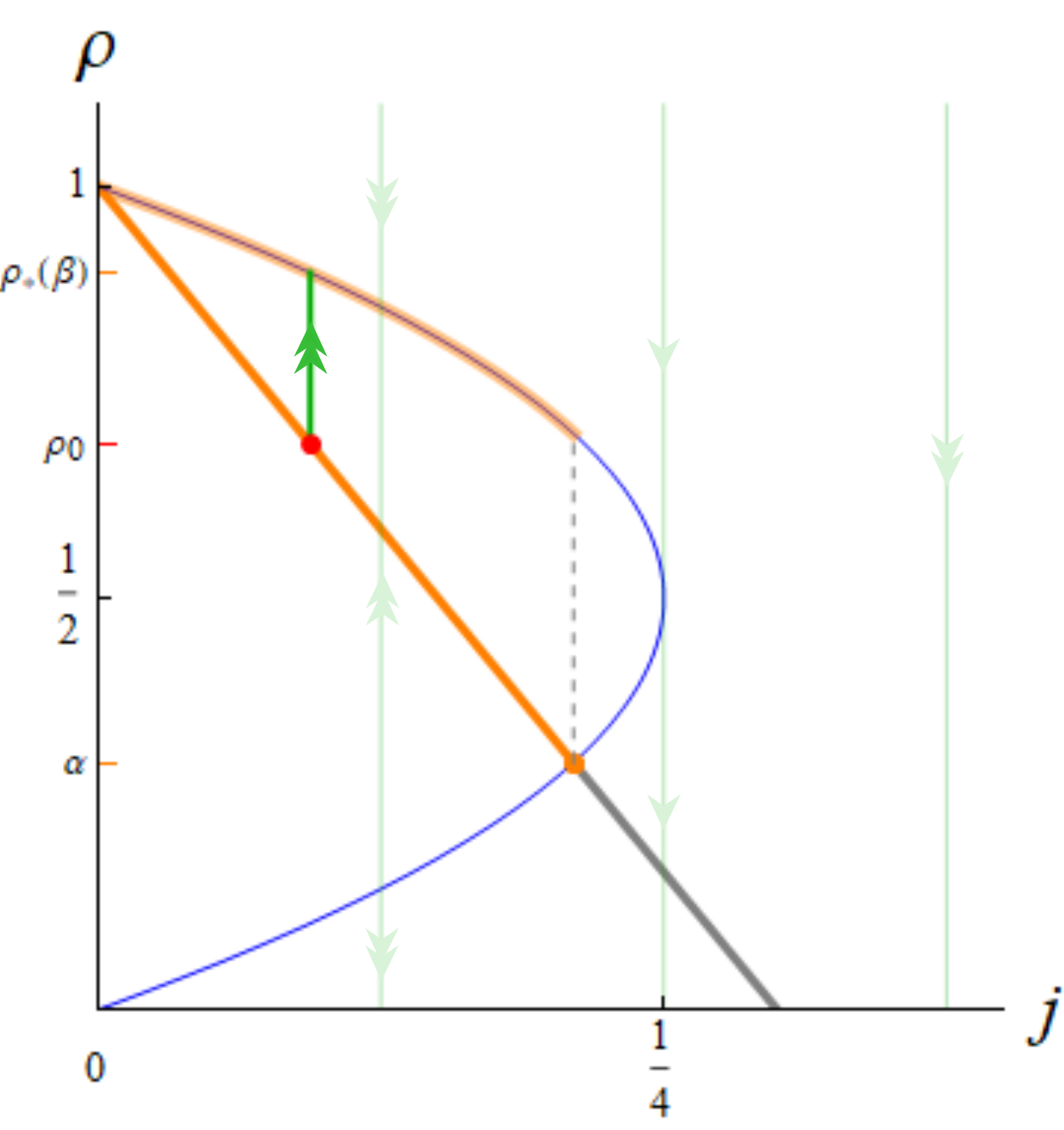\\
  (a) $\xi=0$
 \end{minipage}
 \hspace{.5cm}
 \begin{minipage}{.3\textwidth}
  \centering
  \def\svgwidth{1\textwidth}
  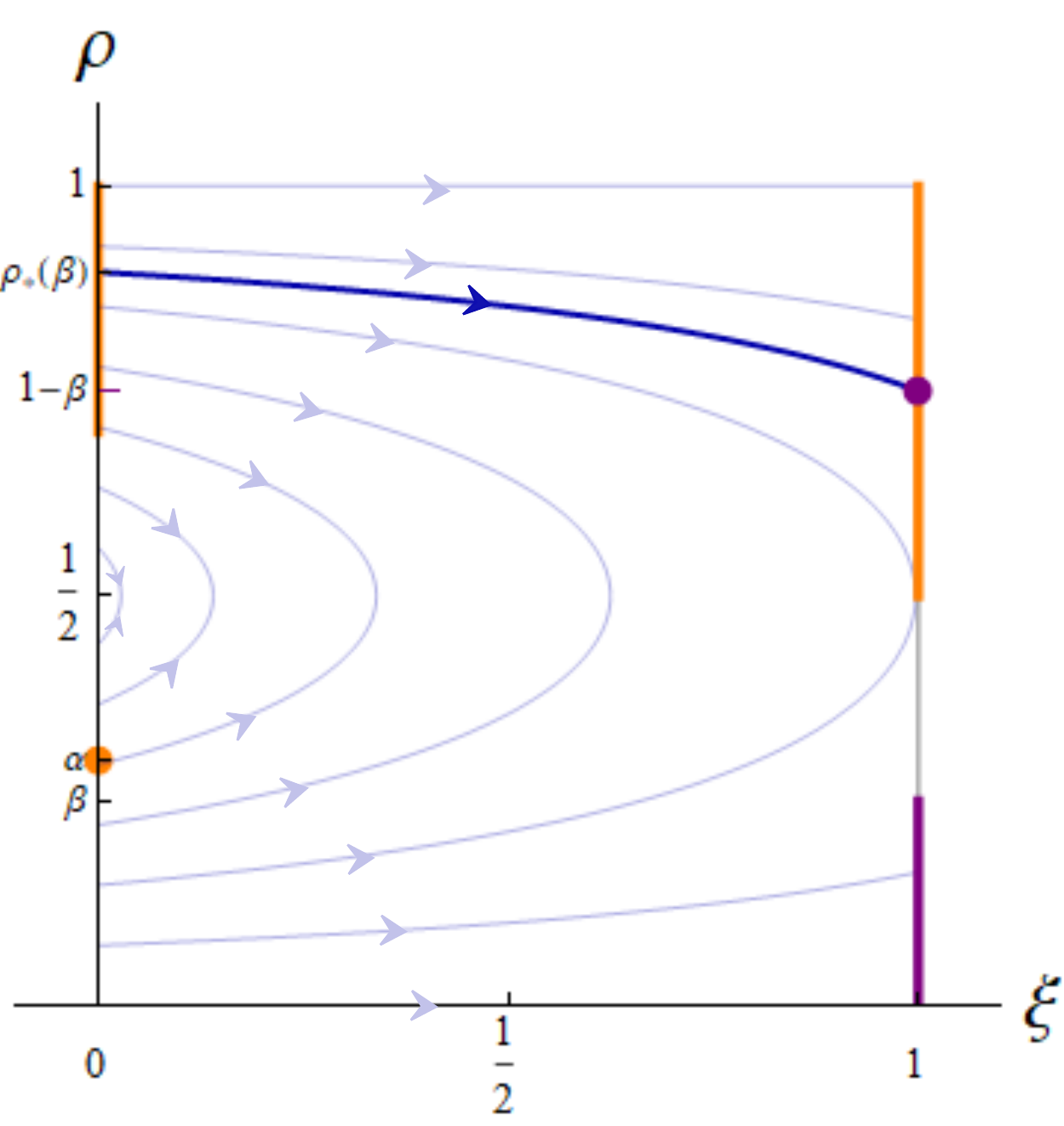\\
  (b) $\xi \in [0,1]$
 \end{minipage}
 \hspace{.5cm}
 \begin{minipage}{.3\textwidth}
  \centering
  \def\svgwidth{1\textwidth}
  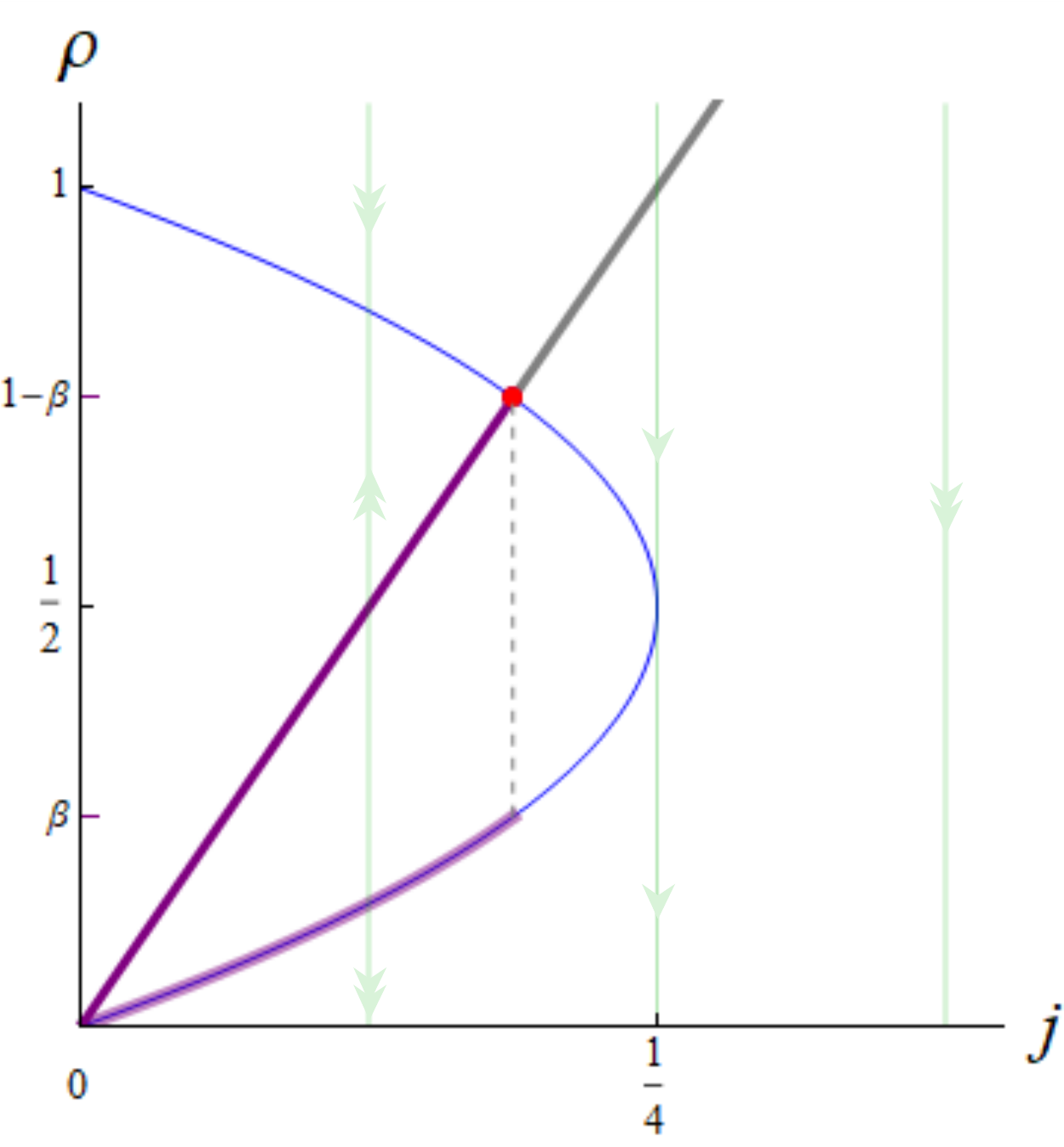\\
  (c) $\xi=1$
 \end{minipage}
  \begin{minipage}{.3\textwidth}
  \centering
  \def\svgwidth{1\textwidth}
  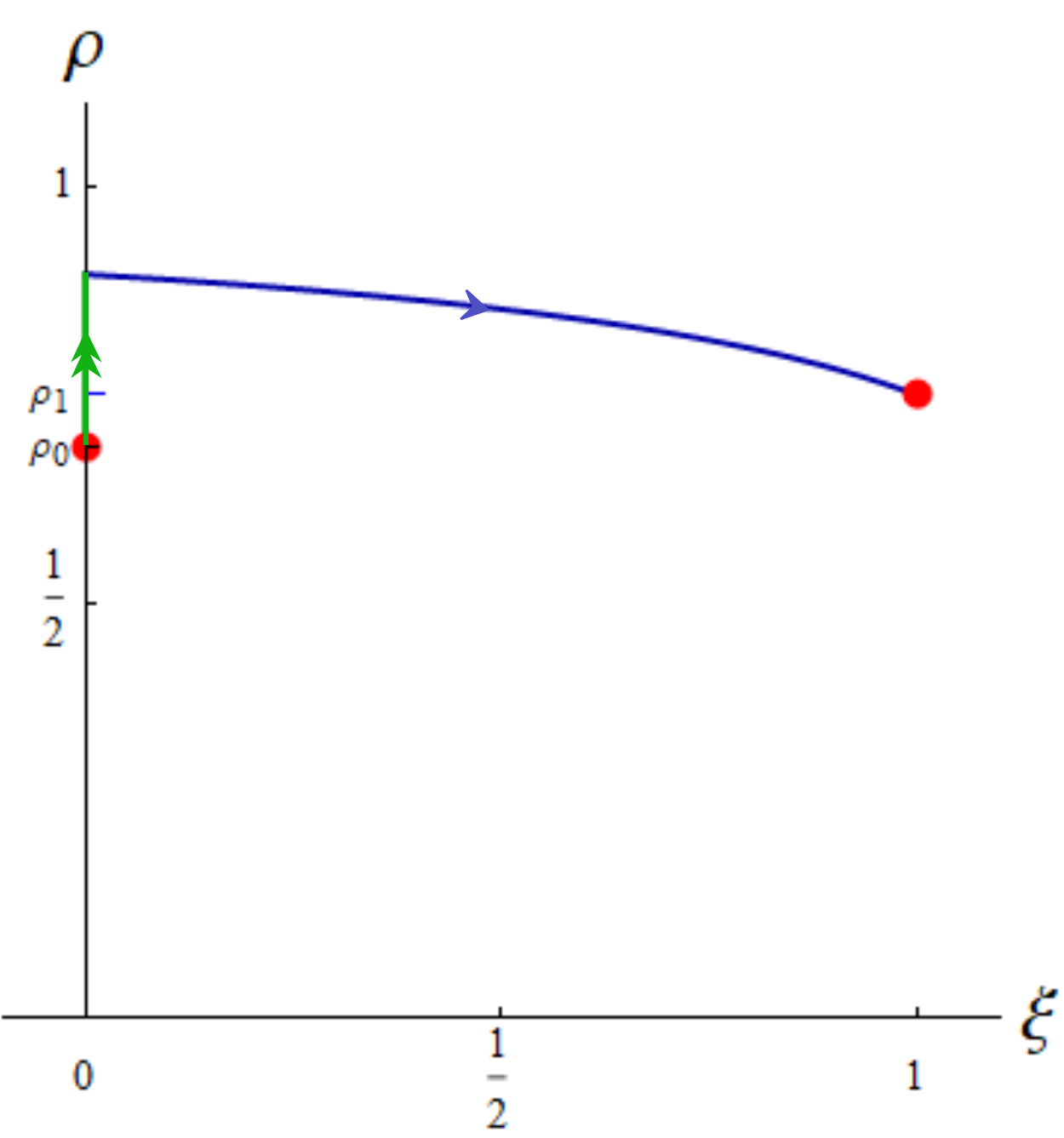\\
  (d) 
 \end{minipage}
 \caption{Schematic representation of a singular solution of type 5. (a) Boundary conditions at $\xi=0$ in $(j,\rho)$-space: the orange line is $\mathcal{L}$, while the orange curve is $\mathcal{L}^+$. The red dot represents $p_0$, and the green line illustrates the layer at $\xi=0$ where $\rho$ increases from $\rho_0=1-\frac{\beta(1-\beta)k(1)}{\alpha k(0)}$ to $\rho(0,\rho_0)=\rho_\ast(\beta)$ (since the range of $\rho$-values on which this transition occurs is very small, we illustrate a zoomed version in the square on the right). (b) Slow evolution on $\mathcal{C}_0$ from $(0,\rho_\ast(\beta))$ to $(1,1-\beta)$ (blue curve). The orange lines are the projection of $\mathcal{L}$ (at $\xi=0$) and $\mathcal{L}^+$ (at $\xi=1$) on $\mathcal{C}_0$, while the purple one represents the projection of $\mathcal{R}^-$ on $\mathcal{C}_0$ at $\xi=1$. The orange dots correspond to $l$ (at $\xi=0$) and $l_1$ (at $\xi=1$), while the purple dot corresponds to $r$. (c) Here, we consider $\xi=1$ in $(j,\rho)$-space. The red dot corresponds to $p_1$, while the purple line and curve represent the manifolds $\mathcal{R}$ and $\mathcal{R}^-$, respectively. (d) Singular solution of type 5 in $(\xi,\rho)$-space.}
 \label{fig:singsol_5}
 \end{figure}

 \begin{figure}[H] 
 \centering
 \begin{minipage}{.3\textwidth}
  \centering
  \def\svgwidth{1\textwidth}
  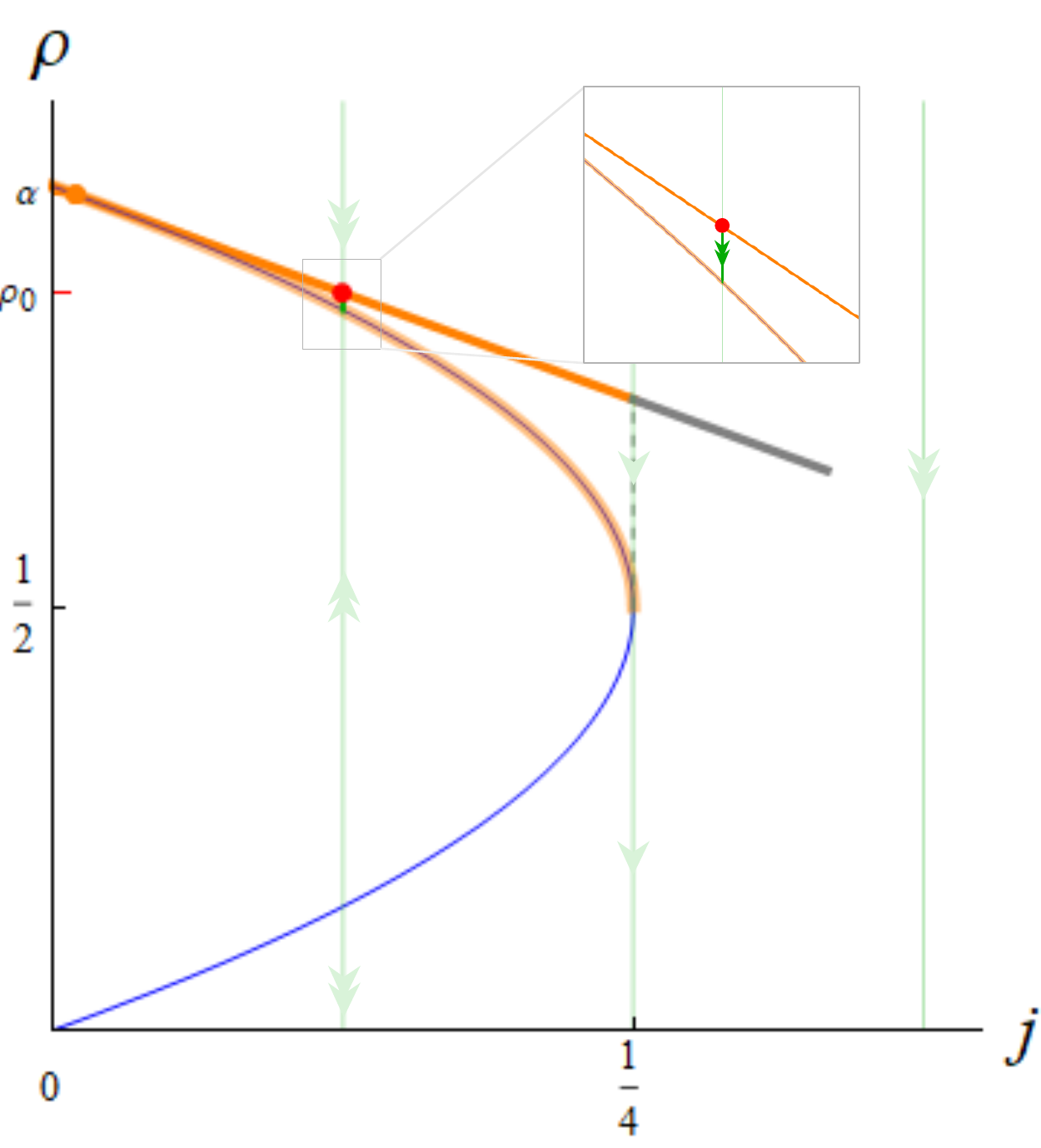\\
  (a) $\xi=0$
 \end{minipage}
 \hspace{.5cm}
 \begin{minipage}{.3\textwidth}
  \centering
  \def\svgwidth{1\textwidth}
  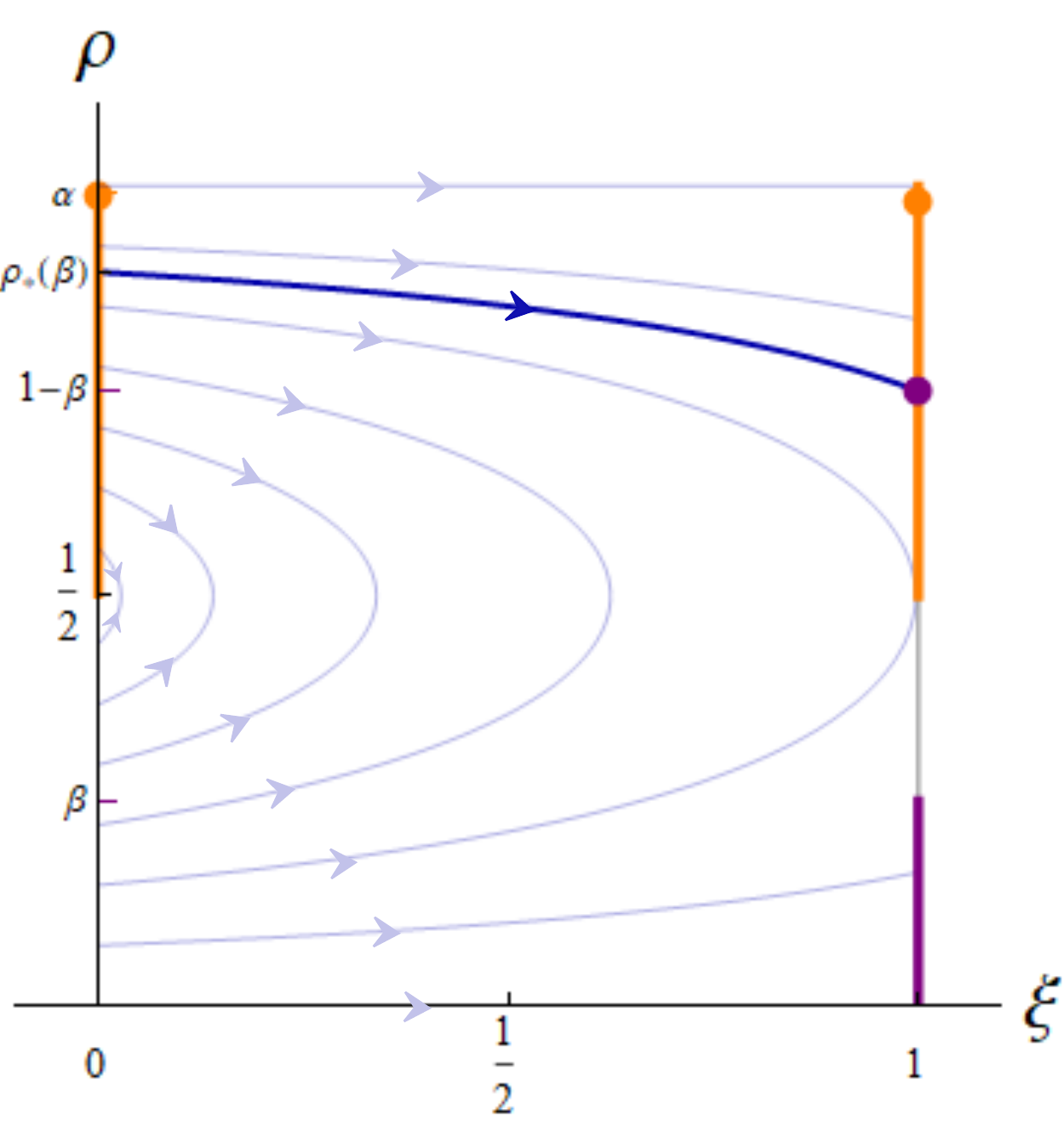\\
  (b) $\xi \in [0,1]$
 \end{minipage}
 \hspace{.5cm}
 \begin{minipage}{.3\textwidth}
  \centering
  \def\svgwidth{1\textwidth}
  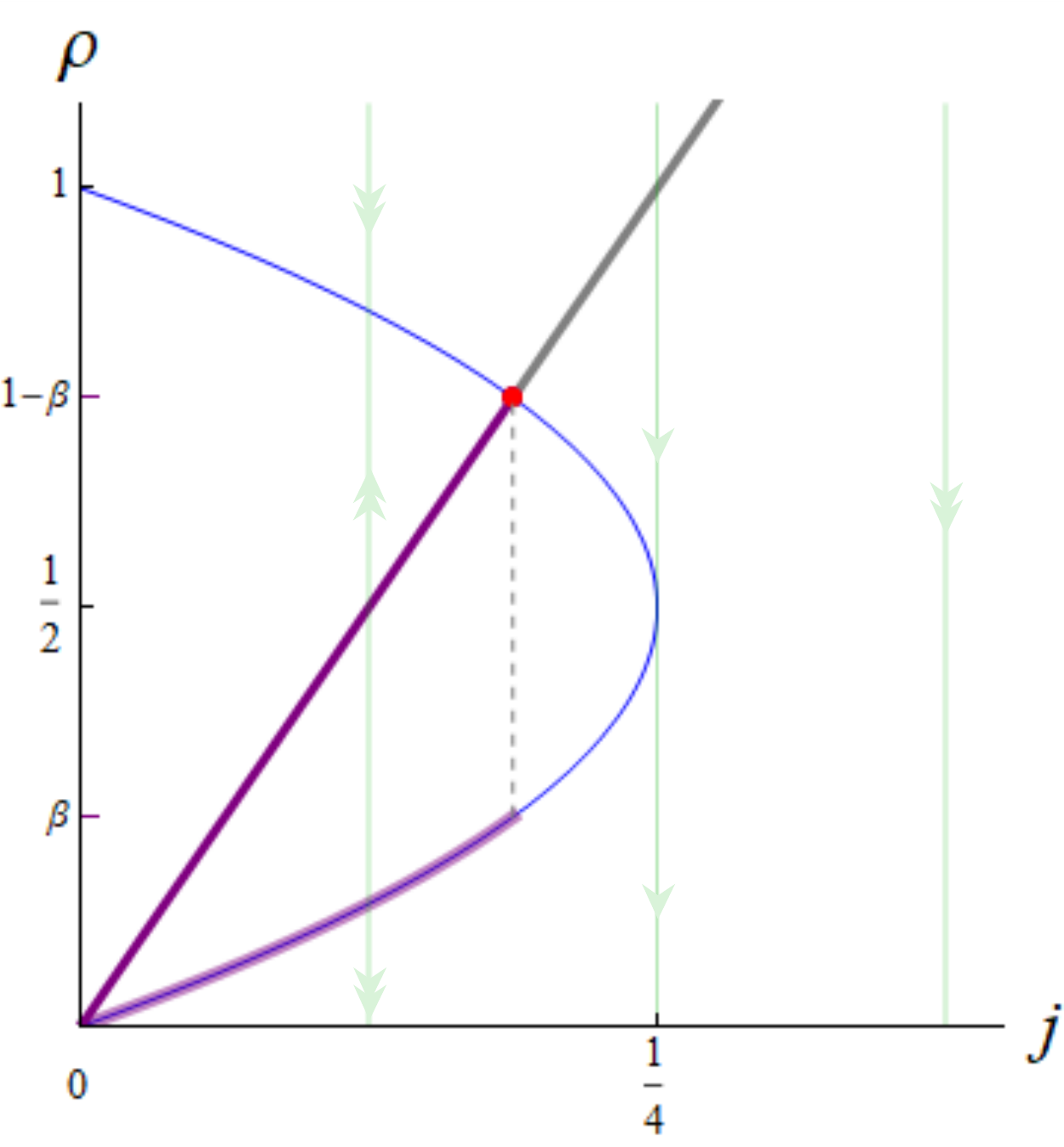\\
  (c) $\xi=1$
 \end{minipage}
  \begin{minipage}{.3\textwidth}
  \centering
  \def\svgwidth{1\textwidth}
  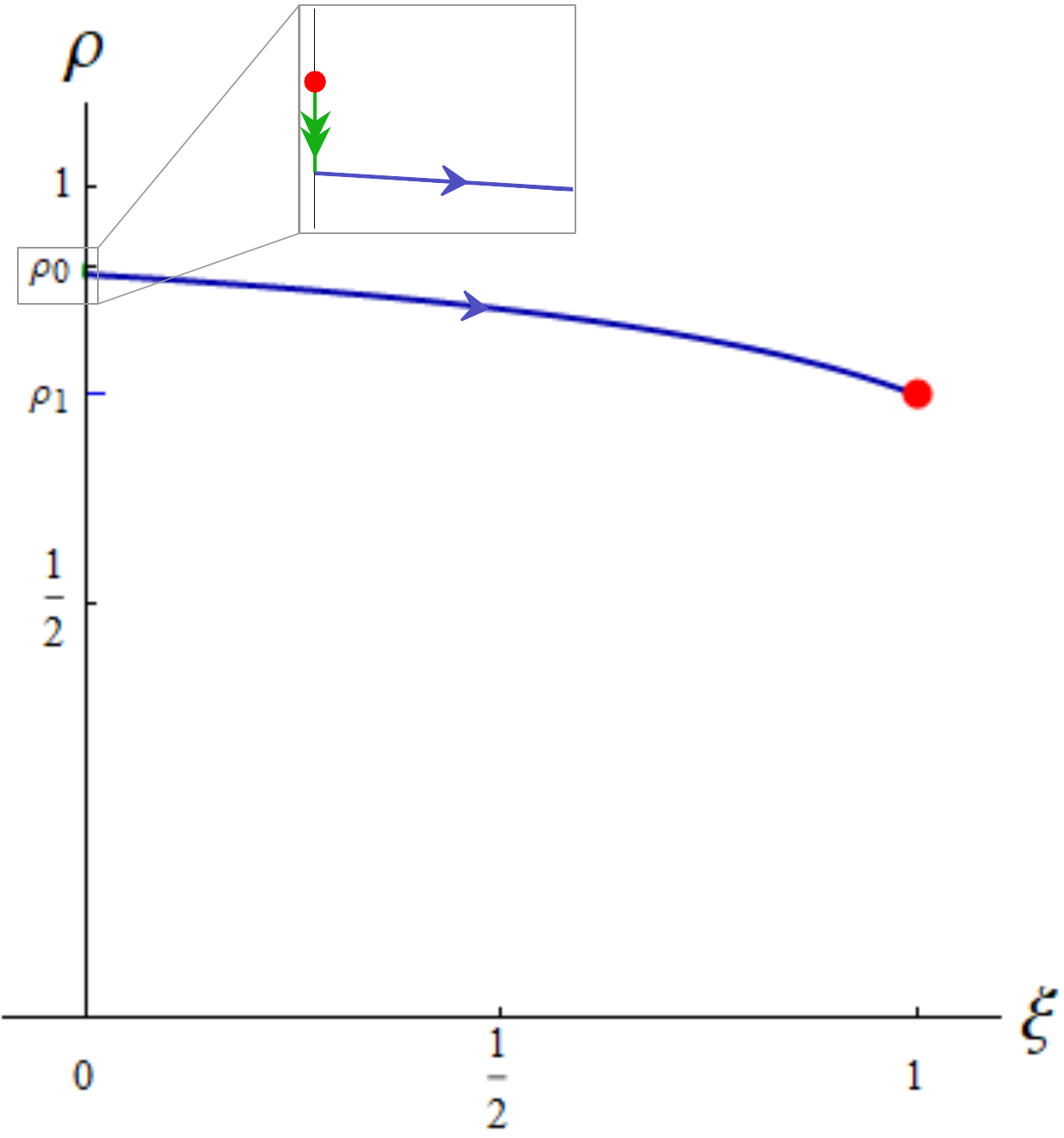\\
  (d) 
 \end{minipage}
 \caption{Schematic representation of a singular solution of type 6. (a) Boundary conditions at $\xi=0$ in $(j,\rho)$-space: the orange line is $\mathcal{L}$, while the orange curve is $\mathcal{L}^+$. The red dot represents $p_0$, and the green line illustrates the layer at $\xi=0$ where $\rho$ decreases from $\rho_0=1-\frac{\beta(1-\beta)k(1)}{\alpha k(0)}$ to $\rho(0,\rho_0)=\rho_\ast(\beta)$ (since the range of $\rho$-values on which this transition occurs is very small, we illustrate a zoomed version in the square on the right). (b) Slow evolution on $\mathcal{C}_0$ from $(0,\rho_\ast(\beta))$ to $(1,1-\beta)$ (blue curve). The orange lines are the projection of $\mathcal{L}$ (at $\xi=0$) and $\mathcal{L}^+$ (at $\xi=1$) on $\mathcal{C}_0$, while the purple one represents the projection of $\mathcal{R}^-$ on $\mathcal{C}_0$ at $\xi=1$. The orange dots correspond to $l$ (at $\xi=0$) and $l_1$ (at $\xi=1$), while the purple dot corresponds to $r$. (c) Here, we consider $\xi=1$ in $(j,\rho)$-space. The red dot corresponds to $p_1$, while the purple line and curve represent the manifolds $\mathcal{R}$ and $\mathcal{R}^-$, respectively. (d) Singular solution of type 6 in $(\xi,\rho)$-space.}
 \label{fig:singsol_6}
 \end{figure}

\section{Explicit profiles for straight corridors} \label{sec:str_chann}
In the following we discuss the different stationary profiles for straight corridors. The profiles depend on the inflow and outflow rate, and have a form similar
to the well-known travelling wave profiles of the viscous Burgers' equation on the real line.\\

In case of a straight corridor $k \equiv 1$ we have
\begin{equation}
    \label{eq:rho constant width}
    -\varepsilon \partial_x \rho + \rho(1-\rho) = \rJ\,,
\end{equation}
for some $\rJ \in \mathbb{R}$. If $\alpha \neq \beta$, then solutions are of the form
\begin{equation}
    \rho = \sfrac{1}{2} + \sqrt{|\rJ - \sfrac{1}{4}|} \; \T\left(\varepsilon^{-1} \sqrt{|\rJ - \sfrac{1}{4}|} (x - \xzero)\right)\,,
    \label{eq:explicit rho}
\end{equation}
where $\rJ$ and $\xzero \in \mathbb{R}$ are determined by $\alpha$ and $\beta$. Note that $\xzero$ is the value of $x$ for which $\rho$ takes the value $\frac{1}{2}$, and that is does not necessarily lies inside the domain. The profile shape is given by
\begin{equation*}
    \T =
    \begin{cases}
        -\tan & \text{ if } \rJ > \sfrac{1}{4}\,,
        \\
        \tanh & \text{ if } \rJ < \sfrac{1}{4} \text{ and } \alpha + \beta < 1\,,
        \\
        \tanh^{-1} & \text{ if } \rJ < \sfrac{1}{4} \text{ and } \alpha + \beta > 1\,.
    \end{cases}
\end{equation*}

The critical case $\alpha + \beta = 1$ corresponds to constant densities, and we get $\rho \equiv \alpha = 1 - \beta$, $\rJ = \alpha(1-\alpha) = \beta(1-\beta)$.
In the particular case $\rJ = \sfrac{1}{4}$, the solutions are of the form
\begin{equation}
    \rho = \frac{1}{2} + \frac{\varepsilon}{x - \xzero}\,,
    \label{eq:explicit rho J one quarter}
\end{equation}
which characterises the interface between the regions $\rJ < \frac{1}{4}$ and $\rJ > \frac{1}{4}$.
The solutions allow us to compute explicit profiles for all combinations of $\alpha$ and $\beta$ (extending the results in \cite{burger_flow_2016}). The respective bifurcation diagram is shown in Figure~\ref{fig:bd_sing_alpha_beta_k1k2}.
\begin{figure}[!ht]
 \centering
\includegraphics[width=0.485\textwidth]{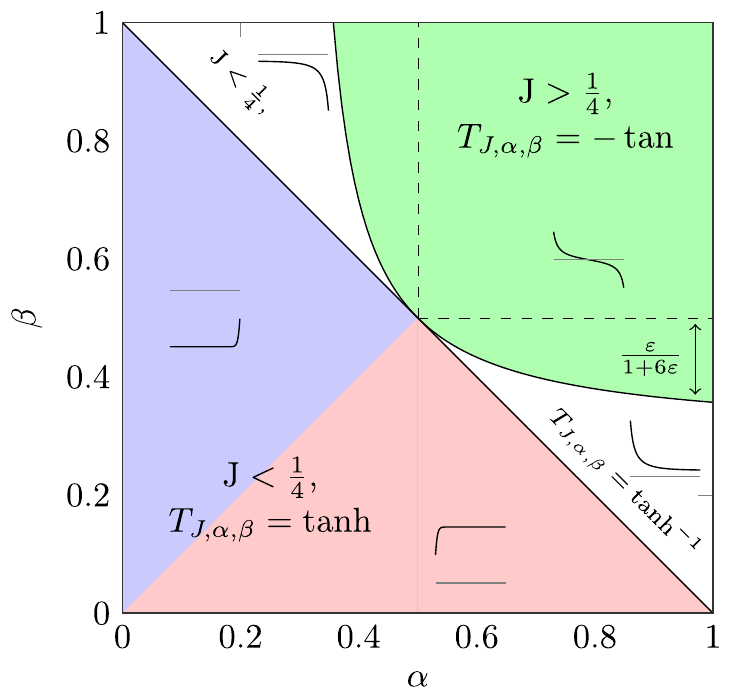}
 \caption{Phase diagram illustrating the stationary profiles for different inflow and outflow parameters $\alpha$ and $\beta$, along with the respective expressions for $\T$, in the case $k(x) \equiv k_0$, $\varepsilon/L=1$.}
 \label{fig:bd_sing_alpha_beta_k1k2}
\end{figure}
The constant $\xzero$ can be computed from the boundary conditions
\begin{align}
    \alpha\left(\sfrac{1}{2} + \sqrt{\sfrac{1}{4} -\rJ}\; \T \left(\varepsilon^{-1}\xzero \sqrt{\sfrac{1}{4} - \rJ}\right)\right) - \rJ &= 0\,,
    \label{eq:explicit boundary condition left}
    \\
    \beta\left(\sfrac{1}{2} + \sqrt{\sfrac{1}{4} -\rJ} \; \T \left( \left(\varepsilon^{-1} (L-\xzero)\right) \sqrt{\sfrac{1}{4} - \rJ}\right)\right) -\rJ &= 0\,.
    \label{eq:explicit boundary condition right}
\end{align}
We assume that $\rJ \neq \sfrac{1}{4}$ which is generic.
In the case $\alpha = \beta$ we get $\xzero = \sfrac{L}{2}$ and $\rJ$ is the fixed point of a monotone convex function, and in the limit $\varepsilon \rightarrow 0$, $\rJ = \alpha(1-\alpha)$.
In general, one can solve for $\xzero$ in \eqref{eq:explicit boundary condition left} and obtains
\begin{equation*}
    \xzero = \frac{\varepsilon}{\sqrt{\sfrac{1}{4}- \rJ}} \; \invT\left(\left(\frac{\rJ}{\alpha} - \frac{1}{2}\right) \frac{1}{\sqrt{\sfrac{1}{4} - \rJ}}\right)\,.
\end{equation*}
Alternatively, one can solve for $\xzero$ in both \eqref{eq:explicit boundary condition left} and \eqref{eq:explicit boundary condition right} to get
\begin{equation}
    \label{eq:x one half averaged}
    \xzero
    =
    \frac{L}{2} + \frac{\varepsilon}{2 \sqrt{\sfrac{1}{4} - \rJ}} \left(
        \invT \left(\frac{\frac{\rJ}{\alpha} - \frac{1}{2}}{\sqrt{\sfrac{1}{4} - \rJ}}\right)
        - \invT\left(\frac{\frac{\rJ}{\beta} - \frac{1}{2}}{\sqrt{\sfrac{1}{4} - \rJ}}\right)
\right)\,.
\end{equation}
From this expression one can see that the point $\xzero$ is one the left or the right of $x = \frac{1}{2}$ in the case $J < 0.25$, if either $\alpha$ or $\beta$ is larger in the case $\alpha+
\beta > 1$.  Note that the position of $\xzero$ changes rapidly --- a characteristic which has been observed for the viscous Burgers' equations on the real line as well. Furthermore the maximum admissible flux in a straight corridor is given by
$\T = -\tan$ and taking the maximum value of $\rJ$ for which $\rho$ is continuous. This yields
\begin{equation} \label{eq:Jmax}
    \rJ_{\mathrm{max}} = \frac{1}{4} + \left(\frac{\pi \varepsilon}{L}\right)^2\,.
\end{equation}
If we take the boundary conditions \eqref{eq:boundary conditions reduced rho} into account and remembering that $\rJ_\mathrm{max} := \sup_{\alpha,\beta} \rJ$ is achieved for $\alpha = \beta = 1$ (see Lemma~\ref{lem:monotonicity j}), we get:
\begin{equation*}
    \rJ_{\mathrm{max}} = \frac{1}{4} + \left(\frac{\pi \varepsilon}{L}\right)^2 \left(1 - 8\;\frac{\varepsilon}{L} + 64\left(\frac{\varepsilon}{L}\right)^2 - \frac{32\left(48-\pi^2\right)}{3}\left(\frac{\varepsilon}{L}\right)^3\right) + o\left(\frac{\varepsilon}{L}\right)^5\,.
\end{equation*}

\end{document}